\documentclass[letter, 12pt]{article}

\usepackage{amsmath}
\usepackage{graphicx, psfrag, epsf}
\usepackage{enumerate}
\usepackage{natbib}
\usepackage{url} 

\newcommand{\blind}{0}
\usepackage{xr}
\usepackage{comment}

\usepackage[margin=1in]{geometry}

\usepackage{setspace}
\usepackage{amssymb, amsfonts, amsthm, bm,bbm}
\usepackage{mathtools, mathrsfs}

\usepackage{relsize} 
\usepackage{hyperref}

\newtheorem{theorem}{Theorem}
\newtheorem{lemma}{Lemma}

\newtheorem{corollary}{Corollary}

\newtheorem{example}{Example}

\DeclareMathOperator{\E}{\mathbb E}

\DeclareMathOperator{\V}{\mathbb V}

\DeclareMathOperator*{\argmax}{argmax} 
\DeclareMathOperator{\Length}{length}
\DeclareMathOperator{\argmin}{argmin}

\newcommand{\xbf}{\mathbf{x}}

\newcommand{\parenth}[1]{\left(#1\right)}
\newcommand{\sqbracket}[1]{\left[#1\right]}

\newcommand{\prob}{{\mathbb{P}}}

\newcommand{\GaSP}{\text{GaSP}}
\newcommand{\bx}{{\mathbf{x}}}
\newcommand{\calW}{{\mathcal{W}}}
\newcommand{\calX}{{\mathcal{X}}}
\newcommand{\calH}{{\mathcal{H}}}

\newcommand{\beginsupplement}{
\setcounter{table}{0}
\renewcommand{\thetable}{S\arabic{table}}
\setcounter{figure}{0}
\renewcommand{\thefigure}{S\arabic{figure}}
\setcounter{equation}{0}
\renewcommand{\theequation}{S\arabic{equation}}
\setcounter{lemma}{0}
\renewcommand{\thelemma}{S\arabic{lemma}}
\setcounter{theorem}{0}
\renewcommand{\thetheorem}{S\arabic{theorem}}
\setcounter{section}{0}
\renewcommand{\thesection}{S\arabic{section}}
}
\usepackage{color}

\begin{document}

	\if1\blind
	{
		\title{ \bf A theoretical framework of the scaled Gaussian stochastic process in prediction and calibration}
		\author{Author 1\thanks{
				The authors gratefully acknowledge \textit{please remember to list all relevant funding sources in the unblinded version}}\hspace{.2cm}\\
			Department of YYY, University of XXX\\
			and \\
			Author 2 \\
			Department of ZZZ, University of WWW}
		\maketitle
	} \fi
	
	\if0\blind
	{
        \title{ A theoretical framework of the scaled Gaussian stochastic process in prediction and calibration}
        \vspace{-.15in}
        \author{Mengyang Gu$^*$,  Fangzheng Xie$^{**}$ and Long Wang$^{**}$ \vspace{-.05in}\\
        $^{*}$Department of Statistics and Applied Probability, UC Santa Barbara  \vspace{-.1in}\\
        $^{**}$Department of Applied Mathematical and Statistics, Johns Hopkins University}
        \date{}
        \maketitle
    } \fi
	
\begin{abstract}
\noindent 


Model calibration or data inversion is one of fundamental tasks in uncertainty quantification.  In this work, we study the theoretical properties of the scaled Gaussian stochastic process (S-GaSP), to model the discrepancy between reality and imperfect mathematical models.  We  establish the explicit connection between Gaussian stochastic process (GaSP) and S-GaSP through the orthogonal series representation. The predictive mean estimator in the S-GaSP calibration model converges to the reality at the same rate as the GaSP with a suitable choice of the regularization and scaling parameters. We also show the calibrated mathematical model in the S-GaSP calibration converges to the one that minimizes the $L_2$ loss between the reality and mathematical model,  whereas the GaSP model with other widely used covariance functions does not have this property. 
Numerical examples confirm the excellent finite sample performance of our approaches compared to a few recent approaches. 



\bigskip

\end{abstract}


\section{Introduction}
\label{sec:intro}
In scientific and engineering studies, mathematical models are developed by scientists and engineers based on their expert knowledge to reproduce the physical reality. With the rapid development of the computational technique in recent years, many mathematical models are implemented in computer code, often referred as computer models or simulators. 

Some parameters of the mathematical model are often unknown or unobservable in experiments. For example, the K{\= \i}lauea volcano recently has one of the biggest eruptions in 2018. The location and volume of the magma chamber, as well as the magma supply and storage rate of this volcano, however, is unobservable. Some field data, such as the satellite interferograms and GPS measurement of the ground deformation were used to estimate these parameters for the K{\= \i}lauea volcano  \citep{anderson2017abundant,anderson2019magma}. Using the field observations to estimate the parameters in the mathematical model, and to identify the possible discrepancy between the mathematical model and the reality is widely known as the model calibration or data inversion. 

For any observable input $\mathbf x \in \mathcal X$, denote $y^F(\mathbf x)$ as the field observation and $f^M(\mathbf x, \bm \theta)$ as a mathematical model with unobservable calibration parameters $\bm \theta \in \bm \Theta$. Furthermore, let $y^R(\mathbf x) = \E[y^F(\mathbf x)]$ represent the reality. A routinely used framework  to calibrate imperfect mathematical model is
\begin{align}
y^F(\mathbf x)= f^M(\mathbf x, \bm \theta) + \delta(\mathbf x) + \epsilon,
\label{equ:gp_calibration}
\end{align}
where $\epsilon$ is the noise and $\delta(\cdot)$ is a discrepancy function between the reality and mathematical model. Since the mathematical model is often developed by experts, we assume the mean and trend of the observations are properly modeled in the mathematical model. The discrepancy function was modeled as a Gaussian stochastic process (GaSP) in \cite{kennedy2001bayesian} and the framework has been widely studied in recent years \citep{goldstein2004probabilistic,bayarri2007framework,higdon2008computer}. As both the mathematical model and discrepancy function are jointly estimated in the calibration, the  predictions of the field data were found to be more accurate compared to the ones based on the mathematical model or a nonparametric regression alone. It was found in following-up studies that, however, the variability of the observations can be explained mostly by the discrepancy function in this approach, leaving the calibrated mathematical model far away from the reality, which results in an identifiability problem between the calibration parameters and discrepancy function \citep{tuo2016theoretical, plumlee2017bayesian}.

A few recent studies measure the goodness of calibration in terms of the $L_2$ loss between the calibrated mathematical model and reality  \citep{tuo2015efficient,wong2017frequentist}. These studies seek to find an estimator of $\bm \theta$ that converges to $\bm \theta_{L_2}$, which minimizes the $L_2$ distance between the reality and mathematical model, i.e.,
\begin{equation}
\bm \theta_{L_2} := \argmin_{\bm\theta \in \bm\Theta} \int_{\mathbf{x} \in \mathcal X} [y^R(\mathbf{x}) - f^M(\mathbf{x}, \bm\theta)]^2  d\mathbf{x}.
\label{equ:theta_L_2}
\end{equation}
In \cite{tuo2015efficient}, for instance, the reality is first estimated through a nonparametric regression model without the assistance of the mathematical model. The calibration parameters are then estimated by minimizing the $L_2$ loss between the calibrated mathematical model and the estimator of the reality. Consequently, the calibrated mathematical model by the two-step approach typically fits the observation in terms of $L_2$ distance. For some complex applications, however, it is crucial to jointly estimate the reality and calibration parameters, as the mathematical model is often developed based on expert knowledge, and thus helpful for predicting the reality.

In this work, we study the theoretical properties of the scaled Gaussian stochastic process (S-GaSP), a new approach for modeling the discrepancy function proposed in \cite{gu2018sgasp}. We establish the connection between GaSP and S-GaSP through the orthogonal representation of the process.
We show that the predictive mean from S-GaSP converges to the reality at the same rate as the one from GaSP with a suitable choice of the regularization and scaling parameters. Furthermore, with the same regularization and scaling parameters, the calibration parameters in the S-GaSP can also converge to $\bm \theta_{L_2} $, whereas GaSP calibration with other widely used kernels does not enjoy this property. Although these two convergence properties can be achieved using the aforementioned two-step approaches \citep{tuo2015efficient,wong2017frequentist},  finite sample studies suggest that the predictive accuracy of the reality improves in the  S-GaSP calibration, as the calibration parameters and discrepancy are jointly estimated.  Besides, since the sampling model is fully specified, the model and parameter uncertainty in the S-GaSP calibration can be naturally assessed through the posterior distributions in a Bayesian approach. A close comparison of S-GaSP and other approaches is detailed in Section \ref{sec:comparison}.

This paper is organized as follows. In Section \ref{sec:SGP}, we introduce the S-GaSP along with the orthogonal series representation and joint estimation in calibration.  Two convergence properties are discussed in Section \ref{sec:convergence}.
In Section~\ref{sec:discretized_sgasp}, we introduce the discretized S-GaSP along with the parameter estimation under the Frequentist framework and Bayesian framework. A comparison between the S-GaSP calibration with other alternatives are discussed in Section~\ref{sec:comparison}.
Section~\ref{sec:numerical} provides some numerical studies comparing the S-GaSP calibration approach and other approaches. We conclude this work in Section \ref{sec:conclusion}. The proof of the theoretical results and other supporting evidence of our approaches are given in the supplementary materials. The GaSP calibration, S-GaSP calibration and calibration without a discrepancy function are implemented in the  {\sf RobustCalibration} {\tt R} package available on CRAN. 


\section{The scaled Gaussian stochastic process}
\label{sec:SGP}

Denote $\delta(\cdot) \sim \GaSP(0, \,\sigma^2 K(\cdot,\, \cdot))$ with variance $\sigma^2$ and correlation function $K(\cdot,\cdot)$ such that, for any inputs $\{\mathbf x_i\}_{i=1}^n $, the marginal distribution $(\delta(\mathbf x_1),...,\delta(\mathbf x_n))^T$ follows a multivariate normal distribution with covariance $\mbox{Cov}(\mathbf x_i, \mathbf x_j)=\sigma^2 K(\mathbf x_i, \,\mathbf x_j)$. 
In order to have the mathematical model explain more variability, we introduce a new prior distribution of the discrepancy function, which places 
more probability mass on the smaller random $L_2$ distance between the mathematical model and reality, as this measure quantifies how well a mathematical model fits the reality. The scaled Gaussian stochastic process calibration model is defined as the following hierarchical model:
\begin{equation}
\label{equ:scaled_GP}
\begin{split}
&y^F(\mathbf x) = f^{M}(\mathbf x, \bm \theta)+  \delta_z(\mathbf x) +\epsilon, \\
&\delta_z(\mathbf x) = \left\{\delta(\mathbf x) \mid { \mathsmaller{\int}_{\bm{\xi}\in\mathcal X} }   \delta^2( \bm \xi ) d\bm \xi=Z \right\}, \\
&\delta(\cdot) \sim \text{GaSP}(0, \sigma^2 K(\cdot, \cdot)),\\
&Z\sim p_Z(\cdot), \quad \epsilon \sim \mbox{N}(0, \sigma^2_0 ),
\end{split}
\end{equation}
where  conditional on all parameters, the default choice of $p_Z(\cdot)$ is defined as
\begin{equation}
p_Z(z)=\frac{g_Z(z) p_{\delta}(Z=z)}
{\int_0^\infty g_Z(t) p_{\delta}(Z = t)dt},
\label{equ:p_z}
\end{equation}
with $g_Z(z)$ being a non-increasing scaling function and $p_{\delta}(Z=z)$ being the density of $Z$ at $z$ induced by a GaSP with mean $0$ and covariance $\sigma^2 K(\cdot,\, \cdot)$.

We call $\delta_z(\cdot)$ in (\ref{equ:scaled_GP}) the scaled Gaussian stochastic process (S-GaSP). Given $Z=z$, the S-GaSP becomes a GaSP constrained at the space $\int_{\mathbf x \in \mathcal X} \delta^2(\mathbf x)d \mathbf x = z$.  Note that $Z$ is  the $L_2$ distance between the reality and mathematical model. By construction, the measure for $Z$ induced by S-GaSP has more prior probability mass near 0 than the one by GaSP as $g(\cdot)$ is a non-decreasing function, reflecting one's belief that the mathematical model should be calibrated to fit the reality. 






It is easy to see that when $g_Z(\cdot)$ is a constant function, S-GaSP reduces to GaSP without any constraint. Conditioning on all parameters, we assume 
\begin{equation}
g_Z(z) = \frac{\lambda_z }{2\sigma^2 } \exp\left(-\frac {\lambda_z z}{2\sigma^2 }\right),
\label{equ:g_z}
\end{equation}
with a scaling parameter $\lambda_z$. We select $p_Z(\cdot)$ in (\ref{equ:p_z}) and $g_Z(\cdot)$ in (\ref{equ:g_z}) for the computational results in this work, as any marginal distribution of $\delta_z$ still follows a multivariate normal distribution \cite[Lemma 2.3]{gu2018sgasp}. Other scaling functions may also be used, but we do not pursue this direction in this study. 
 

\subsection{Orthogonal series representation and marginal distribution}
\label{subsec:orthogonal_seq}

Based on Karhunen-Lo{\`e}ve theorem, GaSP with a stationary kernel admits the following representation for any $\mathbf x \in \mathcal X$
\begin{equation}
\delta(\mathbf x)=\sigma \sum^{\infty}_{k=1}\sqrt{\rho_k} Z_k \phi_k(\mathbf x),
\label{equ:KL_gp}
\end{equation}
where $Z_k \overset{i.i.d.}{\sim} N(0, 1)$, $\rho_k$ and $\phi_k(\cdot)$ are the $k$th eigenvalue and eigenfunction of the kernel $K(\cdot, \cdot)$, respectively.  The S-GaSP can also be represented as an orthogonal series given below.
%
\begin{lemma}[Karhunen-Lo{\`e}ve expansion for the S-GaSP]
Assume $p_Z(\cdot)$ and $g_Z(\cdot)$  are defined in (\ref{equ:p_z}) and (\ref{equ:g_z}), respectively. For any $\mathbf x \in \mathcal X$, the S-GaSP defined in (\ref{equ:scaled_GP}) has the following representation
\[\delta_z(\mathbf x)=\sigma \sum^{\infty}_{k=1}\sqrt{\frac{\rho_k}{1+\lambda_z \rho_k}}Z_k \phi_k(\mathbf x),\] where $Z_k \overset{i.i.d.}{\sim} N(0, 1)$, $\rho_k$ and $\phi_k(\cdot)$ are the $k$th eigenvalue and eigenfunction of the kernel $K(\cdot, \cdot)$, respectively.
\label{lemma:KL_sgp}
\end{lemma}
The covariance function of the S-GaSP  can also be decomposed as an infinite orthogonal series, which is an immediate consequence of the fact that the S-GaSP is indeed a GaSP with a transformed kernel (see Lemma 2.3 in \cite{gu2018sgasp} and Lemma~\ref{lemma:KL_sgp}).


\begin{corollary}
Assume $p_Z(\cdot)$ and $g_Z(\cdot)$ are defined in (\ref{equ:p_z}) and (\ref{equ:g_z}), respectively. The marginal distribution of the S-GaSP defined in (\ref{equ:scaled_GP}) follows a multivariate normal distribution 
\begin{equation*}
[\delta_z(\mathbf x_1),..., \delta_z(\mathbf x_n) \mid \sigma^2 \mathbf R_z] \sim \mbox{MN} (\mathbf 0, \sigma^2 \mathbf R_z),
\end{equation*}
where the $(i, j)$ entry of $\mathbf R_z$ is 
\begin{equation}
K_z(\mathbf x_i, \mathbf x_j) = \sum^{\infty}_{k=1} \frac{\rho_k}{1+\lambda_z \rho_k} \phi(\mathbf x_i) \phi(\mathbf x_j).
\label{equ:Mercer_sgp}
\end{equation}
\label{corollary:Mercer_sgp}
\end{corollary}
Corollary~\ref{corollary:Mercer_sgp} implies that the $i$th eigenvalue of the kernel function $K_z(\cdot, \, \cdot)$ in the S-GaSP is $\rho_{z,k}:= \rho_k/(1+\lambda_z\rho_k)$ and the $k$th eigenfunction $\phi_k(\cdot)$ is the same as the one in the GaSP. The form (\ref{equ:Mercer_sgp}) does not give an explicit expression for the kernel in the S-GaSP. Instead of truncating the series,  one may discretize the integral $\int_{\mathbf x \in \mathcal X} \delta^2(\mathbf x)d \mathbf x$, which leads to an explicit expression of the covariance matrix, discussed in Section~\ref{sec:discretized_sgasp}.

The following Corollary~\ref{corollary:Z_sgp} provides a decomposition of $Z$ in the S-GaSP, which follows from Lemma 2.1 in \cite{gu2018sgasp} and Corollary~\ref{corollary:Mercer_sgp}.
\begin{corollary}
\label{corollary:Z_sgp}
Assume the same conditions in Lemma \ref{lemma:KL_sgp} hold. The distribution of $Z = \int_{\mathbf x \in \mathcal X} \delta^2(\mathbf x)d \mathbf x$ induced by the S-GaSP follows 
\[Z \sim \sigma^2 \sum^{\infty}_{k=1}\frac{\rho_k}{1+\lambda_z \rho_k} \chi^2_k(1),\]
where $\{\chi^2_k(1)\}_{k=1}^\infty$ are independent chi-squared random variables with one degree of freedom.
\end{corollary}  
%

Denote $\mathcal H$ and $\mathcal H_z$ as the reproducing kernel Hilbert space attached to GaSP with kernel $K(\cdot, \cdot)$ and S-GaSP with kernel $K_z(\cdot, \cdot)$, respectively.  Let the native norm associated with $K(\cdot, \cdot)$ and $K_z(\cdot, \cdot)$ be $\langle\cdot, \cdot\rangle_{H}$ and $\langle\cdot, \cdot\rangle_{H_z}$, respectively. We conclude this subsection by the explicit connection between the inner product of GaSP and that of S-GaSP. 
\begin{lemma}
Assume $p_Z(\cdot)$ and $g_Z(\cdot)$  are defined in (\ref{equ:p_z}) and (\ref{equ:g_z}), respectively. Let  $h(\cdot)=\sum^{\infty}_{i=1}  h_i \phi_i(\cdot)$ and $g(\cdot)=\sum^{\infty}_{i=1}  g_i \phi_i(\cdot)$ be the elements in $\mathcal H$. It holds that
\[ \langle h, g \rangle_{\mathcal H_z} =\langle h, g \rangle_{\mathcal H} + \lambda_z \langle h, g  \rangle_{L_2(\mathcal X)}.
\]
\label{lemma:RKHS_comparison}
\end{lemma}
\subsection{Joint estimation in the S-GaSP calibration}
\label{subsec:estimation_SGaSP}

With the specification of $p_Z(\cdot)$ in (\ref{equ:p_z}) and $g_Z(\cdot)$ in (\ref{equ:g_z}), after marginalizing out $\bm \delta_z = [\delta_z(\mathbf x_1),$ $..., \delta_z(\mathbf x_n)] $, the marginal distribution of the field observations in (\ref{equ:scaled_GP}) follows a multivariate normal distribution
\begin{equation}
[\mathbf y^F  \mid \bm \theta,   \, \sigma^2_0, \, \lambda, \, \lambda_z] \sim \mbox{MN}(\mathbf f^M_{\bm \theta},\,  \sigma^2_0((n \lambda)^{-1} \mathbf {R}_z +\mathbf I_n) )
\label{equ:y_F_SGP}
\end{equation}
with the regularization parameter $\lambda := \sigma^2_0/ (n\sigma^2)$ and the $(i,\,j)$ entry of $\mathbf R_z$ defined in (\ref{equ:Mercer_sgp}).



Denote $ \mathcal L_z( \bm \theta)$ as the likelihood for $\bm \theta$  in (\ref{equ:y_F_SGP}). We show below that the following joint estimator of $(\bm \theta, \delta_z(\cdot))$ can be written as a penalized kernel ridge regression estimator (KRR), where both the RKHS norm and $L_2$ norm of the discrepancy function are penalized simultaneously. 

\begin{lemma}
The maximum likelihood estimator $\hat {\bm \theta}_{\lambda, \lambda_z, n}:= \argmax_{\bm \theta \in \bm \Theta} \mathcal L_{z}(\bm \theta)$ and predictive mean  $\hat \delta_{\lambda, \lambda_z, n}(\cdot):=\E[ \delta_z(\cdot) \mid \mathbf y^F, \hat {\bm \theta}_{\lambda, \lambda_z, n},  \lambda,\lambda_z ]$ are the same as the estimators of the penalized KRR,
\begin{align}
(\hat {\bm \theta}_{\lambda,\lambda_z, n}, \hat \delta_{\lambda,\lambda_z, n}(\cdot)) &= \underset{\delta(\cdot) \in \mathcal H,\, \bm \theta \in \bm \Theta}{\argmin} \frac{1}{n}\sum^n_{i=1}(y^F(\mathbf x_i) -f^M(\mathbf x_i, \bm \theta) -\delta(\mathbf x_i))^2  +\lambda \|\delta\|^2_{\mathcal H_z}
\label{equ:KRR_sgp}
\end{align}
with  $\|\delta\|^2_{\mathcal H_z}= \|\delta\|^2_{\mathcal H}+\lambda_z \|\delta\|^2_{L_2(\mathcal X)}$.
\label{lemma:sgp_est}
\end{lemma} 


In Lemma \ref{lemma:sgp_est}, both the $L_2$ norm and native norm of the discrepancy function are penalized in S-GaSP calibration. When the discrepancy function is modeled as a GaSP, however,  the $L_2$ norm of the discrepancy function is not penalized (see supplementary materials).  This property of the S-GaSP calibration is the key to guarantee that, under some regularity conditions, the estimated calibration parameters in joint estimator converges to $\bm \theta_{L_2}$ defined in (\ref{equ:theta_L_2}).
A more detailed discussion is provided in Section~\ref{sec:convergence}.





\section{Convergence properties of the S-GaSP calibration}
\label{sec:convergence}
We discuss two convergence properties of the S-GaSP calibration in this section. First, the predictive mean estimator of the reality converges to the truth at the optimal rate with a suitable choice of the regularization and scaling parameters. Second, the estimated calibration parameters by S-GaSP calibration converge to $\bm \theta_{L_2}$ when sample size increases. These two properties are obtained by jointly estimating the discrepancy function and calibration parameters in (\ref{equ:KRR_sgp}).

\subsection{Convergence to the reality}
\label{subsec:nonparametric_SGP}

Let us first consider the following nonparametric regression model, 
\begin{equation}
y(\mathbf x_i)=f(\mathbf x_i)+\epsilon_i, \quad \epsilon_i \overset{i.i.d.}{\sim} \mbox{N}(0, \, \sigma^2_0),\quad i = 1,\ldots,n,
\label{equ:nonparametric}
\end{equation}
where $f$ is assumed to follow the zero-mean S-GaSP prior with the default choice of $p_Z(\cdot)$ and $g_Z(\cdot)$ in (\ref{equ:p_z}) and (\ref{equ:g_z}), respectively. This is a special case where the mathematical model is zero and we will soon extend it to the general case when the mathematical model is not zero. For illustration purposes, we follow \cite{tuo2015efficient} to assume that $\mathbf x_1,\ldots,\bx_n$ are independently sampled from $\mbox{Unif}( [0,\, 1]^{p})$. 

Assume the underlying truth $f_0(\cdot):= \E_y[y(\cdot)]$ resides in the $p$-dimensional Sobolev space
\begin{equation}
\mathcal W^{m}_2(\mathcal X)= \left\{f(\cdot)=\sum^{\infty}_{k=1} f_k \phi_k(\cdot)\in L_2(\mathcal X): \sum^{\infty}_{k=1} k^{2m/p} f^2_k <\infty\right\}
\label{equ:sobolev_W}
\end{equation}
with smoothness $m>p/2$ and $\{\phi_k(\cdot)\}^{\infty}_{k=1}$ being a sequence of the orthonormal basis of $L_2(\mathcal X)$. For any integer vector $\mathbf k=(k_1,...,k_{p})^T$ and a function $f(x_1,...,x_{p}): \mathcal X  \to \mathbb R$, denote $D^{\mathbf k}$ the mixed partial derivative operator $D^{\mathbf k} f(\cdot):= \partial^{|\mathbf k|}f(\cdot)/\partial^{k_1}x_1...\partial^{k_{p}}x_{p}$ with  $|\mathbf k|=\sum^{p}_{i=1} k_i $. For any function in $\mathcal W^{m}_2(\mathcal X)$, we have $\|D^{\mathbf k} f(\cdot)\|_{L_2(\mathcal X)}<\infty$ for any $|\mathbf k|<m$.  

Recall that $\lambda=\sigma^2_0/(n\sigma^2)$ in (\ref{equ:y_F_SGP}). By Lemma~\ref{lemma:sgp_est}, the posterior mean estimator of $f(\cdot)$ with a S-GaSP prior is equivalent to the KRR estimator below 
\begin{equation}
\hat f_{\lambda, \lambda_z, n}=\underset{f \in \mathcal H}{\argmin} \left[ \frac{1}{n}\sum^n_{i=1}(y(\mathbf x_i) - f(\mathbf x_i))^2+\lambda \| f\|^2_{\mathcal H}+\lambda \lambda_z  \| f\|^2_{ L_2(\mathcal X)}\right].
\end{equation}

Recall $\{\rho_k\}_{k = 1}^\infty$ and $\{\phi_k\}^{\infty}_{k=1}$ are the sequence of the eigenvalues and eigenfunctions of the reproducing kernel $K(\cdot, \cdot)$ associated with $\mathcal H$, respectively. For all $k$, we assume the eigenvalues satisfy 
\begin{equation}
c_\rho k^{-2m/p} \leq \rho_k \leq C_\rho k^{-2m/p},
\label{equ:eigen_value_decay}
\end{equation}
for some constants $c_\rho$ and $C_\rho>0$. For all $k \in \mathbb N^+$ and $\mathbf x\in \mathcal X$, we assume the eigenfunctions are bounded uniformly,
\begin{equation}
\sup_{\bx\in\calX}|\phi_k(\mathbf x)| \leq  C_{\phi},
\label{equ:eigen_function_decay}
\end{equation}
where $C_{\phi} > 0$ is a constant depending on the kernel $K(\cdot, \cdot)$. 







We are now ready to state the convergence rate of the S-GaSP for the nonparametric regression model in (\ref{equ:nonparametric}). 

\begin{theorem}
Assume the eigenvalues and eigenfunctions of $K(\cdot, \cdot)$ satisfy (\ref{equ:eigen_value_decay}) and (\ref{equ:eigen_function_decay}), respectively. Further assume $f_0 \in \mathcal W^{m}_2(\mathcal X) $ and denote $\beta := (2m-p)^2 / \{2m(2m+p)\}$. Consider the nonparametric regression model (\ref{equ:nonparametric}). For sufficiently large $n$,  any $\alpha>2$ and $C_{\beta} \in (0,1)$, with probability at least $1-\exp\{-(\alpha-2)/3\}-\exp\left(-n^{C_{\beta}\beta}\right)$,
\begin{align*}
\|\hat f_{\lambda, \lambda_z, n}- f_0 \|_{L_2(\mathcal X)}
\leq 2 \left[\sqrt{2}\left(\|f_0\|_{L_2(\mathcal X)}+\|f_0 \|_{\mathcal H}\right)+C_K\sigma_0\sqrt{\alpha} \right] n^{-\frac{m}{2m+p}}
\end{align*}
and
\begin{align*}
\|\hat f_{\lambda, \lambda_z, n}- f_0 \|_{\calH}
\leq 2 \left[\sqrt{2}\left(\|f_0\|_{L_2(\mathcal X)}+\|f_0 \|_{\mathcal H}\right)+C_K\sigma_0\sqrt{\alpha}\right]
\end{align*}
by choosing $\lambda=n^{-2m/(2 m+ p )}$ and $\lambda_z= {\lambda}^{-1/2}$, where $C_K$ is a constant only depending on the kernel $K(\cdot, \cdot)$. 
\label{thm:nonparametric}
\end{theorem}


The proof of Theorem \ref{thm:nonparametric} is more challenging compared to the proof of convergence of Gaussian stochastic process regression in \cite{yang2017frequentist}. First  $\lambda_z$ can go to infinity in Theorem \ref{thm:nonparametric}, and consequently $||\delta||_{\mathcal H_z}$ is unbounded,  wheres $||\delta||_{\mathcal H}$ was typically assumed bounded in proving the convergence of a nonparametric regression approach. Second we generalize the proof to multivariate inputs. Thus, we substantially modify the tools to  prove Theorem \ref{thm:nonparametric}, given in the supplementary materials. 




The conditions in Theorem \ref{thm:nonparametric} can be relaxed in various ways. From the proof of Theorem~\ref{thm:nonparametric}, it is easy to see that if $\lambda=O(n^{-2m/(2 m+ p )})$ and $\lambda_z\leq O(\lambda^{-1/2})$, the estimator still converges to the truth in $L_2$ distance with the same rate $O(n^{-m/(2m+p)})$. Second the design can be generalized to other space filling design. Besides, although the stationarity of the process is often assumed for the computational purpose, it is not required in Theorem~\ref{thm:nonparametric}. Note the regularity parameter and kernel parameters are held fixed in Theorem \ref{thm:nonparametric}. We discuss the estimation of $\lambda$ and the parameters in the kernel function in Section~\ref{subsec:parameter_est}.




We are ready to discuss the convergence of estimating the reality in the calibration. The estimator for the reality in the S-GaSP calibration model is defined as follows
\[\hat y^R_{\lambda, \lambda_z, n}(\mathbf x):= f^M(\mathbf x, \hat {\bm \theta}_{\lambda, \lambda_z, n})+ \hat \delta_{\lambda, \lambda_z, n}(\mathbf x)\]
for any $\mathbf x\in \mathcal X$, where $(\hat {\bm \theta}_{\lambda, \lambda_z, n}, \hat \delta_{\lambda, \lambda_z, n})$ is the estimator of the penalized KRR obtained by minimizing the loss in (\ref{equ:KRR_sgp}). The following Corollary~\ref{corollary:prediction_bound} gives the convergence rate of the S-GaSP calibration model in predicting the reality. Similar to the extensions for Theorem \ref{thm:nonparametric}, the conditions in Corollary \ref{corollary:prediction_bound} can be relaxed by letting $\lambda  = O(n^{-2m/(2m+p)})$ and $\lambda_z\leq O(\lambda^{-1/2})$  to obtain the same convergence rate.


\begin{corollary}
Assume $y^R(\cdot)-f^M(\cdot, {\bm{\theta}}) \in  \mathcal W^{m}_2(\mathcal X)$ for any ${\bm{\theta}}\in\mathbf{\Theta}$ and $sup_{{\bm{\theta}}\in\mathbf{\Theta}}\|y^R(\cdot)- f^M(\cdot, {\bm{\theta}})\|_{\mathcal H}<\infty$.
Let the eigenvalues and eigenfunctions of $K(\cdot, \cdot)$ satisfy (\ref{equ:eigen_value_decay}) and (\ref{equ:eigen_function_decay}), respectively. For sufficiently large $n$, any $\alpha>2$ and $C_{\beta} \in (0,1)$, with probability at least $1-\exp\{-(\alpha-2)/3\}-\exp(-n^{C_{\beta}\beta})$,
\begin{align*}
\| \hat y^R_{\lambda, \lambda_z, n}(\cdot) - y^R(\cdot)\|_{L_2(\mathcal X)} &\leq 2 \bigg[ \bigg. \sqrt{2}\bigg( \bigg.\sup_{{\bm{\theta}}\in\mathbf{\Theta}}\|y^R(\cdot)- f^M(\cdot, {\bm{\theta}})\|_{L_2(\mathcal X)} \\
& \quad +\sup_{{\bm{\theta}}\in\mathbf{\Theta}}\|y^R(\cdot)- f^M(\cdot, {\bm{\theta}})\|_{\mathcal H} \bigg) \bigg.+C_K\sigma_0 \sqrt{\alpha}  \bigg. \bigg] n^{-\frac{m}{2m+p}},
\end{align*}
by choosing $\lambda=n^{-2m/(2 m+ p )}$ and  $\lambda_z =\lambda^{-1/2}$, where $C_K$ is a constant only depending on the kernel $K(\cdot, \cdot)$ and $\beta = (2m-p)^2 / (2m(2m+p))$.  
\label{corollary:prediction_bound}
\end{corollary}

We illustrate the convergence using the following function studied in \cite{yang2017frequentist}, where  $y^R(\cdot)$ lies in the Sobolev space $\mathcal W^m_2(\mathcal X)$ with $m=3$ and $\mathcal X=[0,1]$.
\begin{example}
Let the reality be $y^R(x)=2\sum^{\infty}_{j=1}j^{-6}\cos(5\pi (j-0.5)x)\sin( 5j)$, and consider $y^F( x)=y^R(x)+\epsilon$, where $\epsilon \sim N(0,0.05^2)$ independently. Let  $f^M(x,\theta)=\theta$.  The goal is to predict $y^R(x)$ at $x\in [0,\,1]$ and estimate $\theta$. 
\label{eg:eg1}
\end{example}
As a motivating example, we let $K(\cdot, \cdot)$ follow the Mat{\'e}rn kernel in (\ref{equ:matern_5_2}) with the range parameter $\gamma=1$, as the reproducing kernel Hilbert space attached to the GaSP with this kernel is equal to Sobolev space $\mathcal W^3_2(\mathcal X)$.  We test 50 configurations with the number of observations $n \in [\exp(5), \exp(10)]$, and the design points $\{x_i\}_{i = 1}^n$ are equally spaced in $[0,1]$. In each configuration, $N=100$ simulation replicates are implemented. We first compute the average root of the mean squared error below: 
\begin{equation}
\mbox{AvgRMSE}_{f^M+\delta}=\frac{1}{N}\sum^N_{i=1}  \sqrt{\frac{1}{n^*}\sum^{n^*}_{j=1} (\hat y^R_i(\mathbf  x^*_j) - y^R_i(\mathbf x^*_j) )^2},
\label{equ:AvgRMSE_fm_delta}
\end{equation}
where $\hat y^R_i(\mathbf x^*_j)$ is an estimator of the reality at $\mathbf x^*_j$ for $j=1,...,n^*$. The subscript $f^M+\delta$ indicates both the calibrated mathematical model and discrepancy function can be used for prediction. 

\begin{figure}[t]
\centering
\begin{tabular}{cc}
\includegraphics[height=.35\textwidth,width=.5\textwidth ]{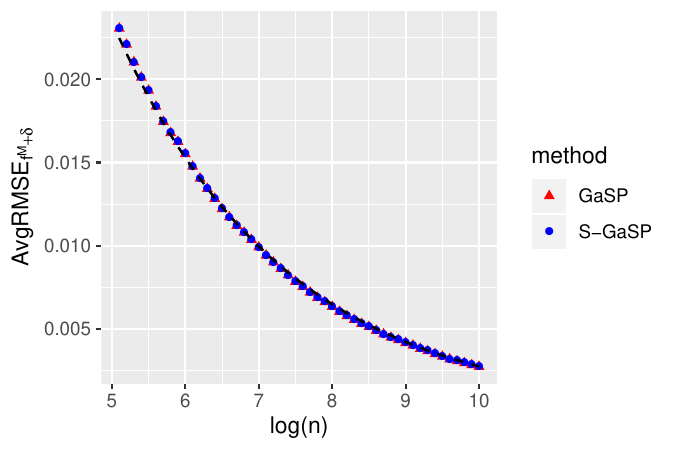}
\includegraphics[height=.35\textwidth,width=.5\textwidth]{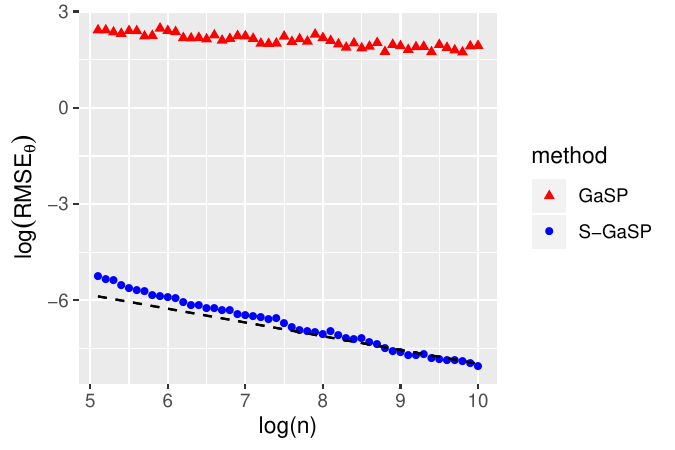} 
\end{tabular}
\caption{Prediction and calibration by the GaSP and discretized S-GaSP calibration models for Example~\ref{eg:eg1}.  In the left panel, the $\mbox{AvgRMSE}_{f^M+\delta}$ of the GaSP calibration and that of the discretized S-GaSP calibration are graphed as the red triangles and blue dots, respectively; the black curve is $n^{-m/(2m+p)}/5$, representing the  upper bound by Corollary~\ref{corollary:prediction_bound} (up to a constant).   In the right panel, the natural logarithm of the $\mbox{RMSE}_{\bm \theta}$ of the GaSP calibration and that of the discretized S-GaSP calibration are graphed as the red triangles and blue circles, respectively; the black line is $\log(n^{-m/(2m+p)}/40)$, the  upper bound from Theorem \ref{thm:L_2_convergence} (up to a constant).   $\lambda=n^{-2m/(2m+p)}\times 10^{-4}$ and $\lambda_z=\lambda^{-1/2}$ are assumed. 
 }
\label{fig:fixed_lambda_GP_SGP}
\end{figure}


 For both GaSP and S-GaSP calibration, the joint estimator, i.e. the predictive mean of the reality and maximum marginal likelihood estimator of the calibration parameters discussed in Lemma \ref{lemma:sgp_est}, is implemented for each experiment at each configuration.  In the left panel of Figure~\ref{fig:fixed_lambda_GP_SGP}, the predictive mean estimator of the reality in both GaSP and S-GaSP estimator converges to the reality at the same rate, which matches the theoretical upper bound from Corollary~\ref{corollary:prediction_bound}. Here for computational purpose, we graph the results of discretized S-GaSP calibration, which replaces the integral $\int_{\mathbf  x \in \mathcal X} \delta^2(\mathbf x) d\mathbf x$ in the S-GaSP model in (\ref{equ:scaled_GP}) by $(1/n) \sum^n_{i=1} \delta^2(\mathbf x_i)$ in Figure \ref{fig:fixed_lambda_GP_SGP}. The  discretized S-GaSP calibration is discussed in Section \ref{sec:discretized_sgasp}.

To evaluate whether the calibrated mathematical model (here only a mean parameter) fits the data, we use the root of the mean squared error between the estimator of the calibration parameters and the $L_2$ minimizer $\bm \theta_{L_2}$ as follows $\mbox{RMSE}_{\bm \theta}=\sqrt{\frac{1}{N}\sum^N_{i=1} (\hat {\bm \theta}_i- \bm \theta_{L_2} )^2}$, where $\hat {\bm \theta}_i$ is an estimator of $\bm \theta$ in the $i$th experiment. 

Although the GaSP and the S-GaSP perform equally well in prediction for Example \ref{eg:eg1}, the estimator of the calibration parameter in the discretized S-GaSP calibration converges to the $L_2$ minimizer, but that in the GaSP calibration does not converge to $\theta_{L_2}$, shown in the right panel of Figure~\ref{fig:fixed_lambda_GP_SGP}. This problem is caused by the difference between the RKHS norm and the $L_2$ norm. As illustrated in  Lemma~\ref{lemma:sgp_est}, both the RKHS norm and $L_2$ norm of the discrepancy function are penalized in the S-GaSP calibration model, whereas the GaSP calibration model does not penalize the $L_2$ norm of the discrepancy function. In  Section~\ref{subsec:theta_convergence}, we further show that under some regularity conditions, the calibrated parameters in the S-GaSP calibration do converge to  the $L_2$ minimizer with the same choice of the regularization parameter and scaling parameter in Corollary \ref{corollary:prediction_bound}.

\subsection{Convergence of calibration parameters}
\label{subsec:theta_convergence}
We first list some regularity conditions for the convergence of calibration parameters.
  

\begin{enumerate}[{A}1]
\item $\bm \theta_{L_2}$ is the unique solution of (\ref{equ:theta_L_2}) and it is an interior point of $\bm \Theta$.
\item The Hessian matrix $\int  \frac{ \partial^2 (y^R(\mathbf x)-f^M(\mathbf x, {\bm{\theta}} ) )^2}{\partial \bm \theta \partial \bm \theta^T} d\mathbf x$ is invertible in a neighborhood of $\bm \theta_{L_2}$. 
\item  For all $j=1,...,q$, it holds that $\sup_{\bm \theta \in \bm\Theta}    \left\Vert \frac{\partial f^M(\cdot, \bm \theta)}{\partial \theta_j} \right\Vert_{\mathcal H}<\infty$. 
\item The function class $\{y^R(\cdot)- f^M(\cdot, \bm \theta):{\bm{\theta}}\in\bm\Theta\}$ is Donsker.
\item  $\sup_{\bm \theta \in \bm \Theta} \|y^R(\cdot)- f^M(\cdot, \bm \theta)\|_{\mathcal H}<\infty$.
\item The eigenvalues and eigenfunctions of $K(\cdot, \cdot)$ satisfy (\ref{equ:eigen_value_decay}) and (\ref{equ:eigen_function_decay}), respectively.


\end{enumerate}

Assumptions A1 to A3 are regularity conditions of $\bm \theta_{L_2}$ and the mathematical model $y^M(\cdot,{\bm{\theta}})$ around $\bm \theta_{L_2}$. Assumptions A4 to A6 guarantee the KRR estimator $\hat \delta_{z,\bm \theta}$ converges to $y^R(\cdot)-f^M(\cdot, \bm \theta)$ uniformly for each $\bm \theta \in \bm \Theta$ in terms of the $L_2$ loss. We have the following result in Theorem \ref{thm:L_2_convergence} that guarantees the convergence of calibration parameters. As the calibration parameters and discrepancy function are estimated jointly in our approach, we extend the tools of proving the convergence of the two-step calibration approach  \citep{tuo2015efficient} to prove Theorem \ref{thm:L_2_convergence}, given in the supplementary materials.



\begin{theorem}
Under assumptions A1 to A6, the estimated parameters in (\ref{equ:KRR_sgp}) follow
\[\hat {\bm \theta}_{\lambda,\lambda_z, n} = \bm \theta_{L_2}+O_p(n^{-\frac{m}{2m+p}}),\]
by choosing $\lambda=O(n^{-2m /(2m+p)})$ and $\lambda_z=  O({\lambda}^{-1/2})$. 
\label{thm:L_2_convergence}
\end{theorem}

Note that  $\lambda=O(n^{-2m/(2m+p)})$ and $\lambda_z=O({\lambda}^{-1/2})$ also guarantee the predictive mean estimator in the S-GaSP calibration converges to the reality at the rate $O(n^{{-m}/(2m+p)})$ in terms of the $L_2$ loss.  The convergence rate of the calibration parameter is slightly slower than  $O(n^{-1/2})$ obtained in the two-step approach \citep{tuo2015efficient}. Though the $O(n^{-1/2})$ rate may be obtained by choosing $\lambda=O(n^{-2m/(2m+p)})$ and $\lambda_z=O({\lambda}^{-1/2})$, we should be aware that, however, the $L_2$ minimizer is not the true calibration parameter, but the one that minimizes the $L_2$ distance between the calibrated mathematical model and reality. At the finite sample scenarios, the residuals between reality and calibrated mathematical model with the $L_2$ minimizer may behave like noises, rather than a smooth function, which may be hard to be accurately estimated by a nonparametric model of the discrepancy function. In comparison, the joint estimate of the discrepancy function and calibration parameters was found have a smaller predictive error in Example \ref{eg:compare_rates_all}, \ref{eg:sin_2d} and \ref{eg:box}. 

On the contrary, the calibrated parameters of the GaSP calibration typically do not converge to the $L_2$ minimizer. Let $ \frac{\partial f^M(\cdot, \bm {\hat \theta})}{\partial\theta_j}:=\frac{\partial f^M(\cdot, \bm {\theta})}{\partial\theta_j} \rvert_{\bm \theta=\hat {\bm \theta}}$. A key difference between the GaSP and the S-GaSP calibration is stated in the following Corollary \ref{corollary:diff_GP_SGP_calibration}, which is an immediate consequence from the proof of Theorem \ref{thm:L_2_convergence}. 



\begin{corollary}
Under assumptions A1 to A6,  the estimator for the calibration parameters in the S-GaSP calibration in (\ref{equ:KRR_sgp}) satisfies 
\begin{align*}
\frac{1}{\lambda}_z  \bigg \langle \hat \delta_{\lambda, \lambda_z, n}(\cdot),\frac{\partial f^M(\cdot,\hat {\bm \theta}_{\lambda, \lambda_z, n })}{\partial \theta_j}   \bigg \rangle_{\mathcal H} + \bigg \langle \hat \delta_{\lambda, \lambda_z, n}(\cdot), \frac{\partial f^M(\cdot,\hat {\bm \theta}_{\lambda, \lambda_z, n })}{\partial \theta_j}  \bigg \rangle_{L_2(\mathcal{X})}=0;
\end{align*}
Further assuming the mathematical model is differentiable at  $\hat {\bm \theta}_{\lambda, n}$ in (\ref{equ:KRR_gp}), the estimator of the calibration parameters in the GaSP calibration satisfies
\[\bigg\langle \hat \delta_{\lambda, n}(\cdot),\frac{\partial f^M(\cdot, \hat {\bm \theta}_{\lambda,  n} )}{\partial \theta_j}  \bigg\rangle_{\mathcal H}=0,\]
for any $\theta_j$, $j=1,...,q$.
\label{corollary:diff_GP_SGP_calibration}
\end{corollary}


To ensure the convergence of an estimator $\bm {\hat \theta}$ to the $L_2$ minimizer, one typical requirement is that $\langle \hat \delta_{L_2}(\cdot), \frac{\partial f^M(\cdot, \hat {\bm \theta})}{\partial \theta_j} \rangle_{L_2(\mathcal{X})}=o_p(1)$. It is easy to see that the S-GaSP satisfies this condition with $1/\lambda_z=o(1)$. However, because of the difference between the RKHS norm and $L_2$ norm, the estimated parameters $\hat {\bm \theta}_{\lambda, n}$ in the GaSP calibration model can be far away from the $L_2$ minimizer. As a result, the calibrated mathematical model may not fit the data in the GaSP calibration model, as found in  previous studies \citep{tuo2015efficient,wong2017frequentist}.  For Example~\ref{eg:eg1}, the estimated parameters in the discretized S-GaSP calibration  converges to the $L_2$ minimizer when the sample size increases, whereas the parameters in GaSP calibration with an unscaled kernel function do not converge, as shown in the right panel of Figure~\ref{fig:fixed_lambda_GP_SGP}. 

\section{Discretized scaled Gaussian stochastic process}
\label{sec:discretized_sgasp}

We address the computational issue in the S-GaSP calibration in this section. Instead of truncating the kernel function in (\ref{equ:Mercer_sgp})  by the first several terms, we  select $N_C$ distinct points  to discretize the input space $[0, 1]^p$ and replace ${ {\int}_{\bm{\xi}\in\mathcal X} }  \delta( \bm \xi )^2 d\bm \xi$ by $(1/N_C)\sum^{N_C}_{i=1} \delta( \mathbf x^C_{i} )^2 $ in the S-GaSP model in (\ref{equ:scaled_GP}).

Here we let the discretization points be the observed inputs, i.e. $\mathbf x^C_{i}=\mathbf x_i $ for $i=1,...,N_C$ and $N_C=n$. The discretized S-GaSP is to replace $\delta_z$ in Equation (\ref{equ:scaled_GP}) by
\begin{align}
&\delta_{z_d}(\mathbf x) = \left\{ \delta(\mathbf x) \mid \frac{1}{n}{\sum^{n}_{i=1} \delta( \mathbf x_{i} )^2} =Z_d\right\} 
\label{equ:delta_z_d}
\end{align}
with density $p_{Z_d}(\cdot)$ defined in (\ref{equ:p_z}).
 After marginalizing out $Z_d$, it follows from  Lemma 2.4 in \cite{gu2018sgasp} that $\delta_{z_d}(\cdot)$  is a zero-mean GaSP with the covariance function  
\begin{equation}
\sigma^2 K_{z_d}(\mathbf x_a, \mathbf x_b)=\sigma^2 (K(\mathbf x_a,\mathbf x_b ) - \mathbf r^T(\mathbf x_a) \tilde {\mathbf R}^{-1} \mathbf r(\mathbf x_b))
\label{equ:sigma_2_K_z_d}
\end{equation}
for any $\mathbf x_a, \mathbf x_b \in \mathcal X$, where $  \tilde{\mathbf R}:=\mathbf R+n\mathbf I_n / \lambda_z $.

Recall $\lambda=\sigma^2_0/(n\sigma^2)$.  We have the following predictive distribution of the discretized S-GaSP calibration model. 


\begin{theorem}
Assume $\delta_{z_d}(\cdot)$ in (\ref{equ:delta_z_d}) with $p_{Z_d}(\cdot)$ and $g_{Z_d}(\cdot)$ defined in (\ref{equ:p_z}) and (\ref{equ:g_z}), respectively. The predictive distribution of the field data at any $\mathbf x \in \mathcal X$ by the discretized S-GaSP calibration model in (\ref{equ:delta_z_d}) is a multivariate normal distribution	
\[ y^F(\mathbf x) \mid \mathbf y^F, \bm \theta, \sigma^2_0, \lambda, \lambda_z \sim \mbox{MN}( \hat \mu_{z_d}(\mathbf x), \sigma^2_0 ( (n\lambda)^{-1} K^*_{z_d}(\mathbf x, \mathbf x)+1) ), \]
where $\hat \mu_{z_d}(\mathbf x)=f^M(\mathbf x, \bm \theta)+\frac{\mathbf r^T(\mathbf x)}{1+\lambda \lambda_z}   \left(\mathbf R+\frac{n\lambda}{1+\lambda \lambda_z} \mathbf I_n \right)^{-1}\left(\mathbf y^F- \mathbf f^M_{\bm \theta} \right), $ and 
$K^*_{z_d}(\mathbf x, \mathbf x)=K(\mathbf x, \mathbf x)-\mathbf r^T(\mathbf x) \bigg[ \bigg. \mathbf I_n +\left( \mathbf R+\frac{n\lambda}{1+\lambda \lambda_z} \mathbf I_n \right)^{-1} \frac{n}{(1+\lambda \lambda_z)\lambda_z} \bigg. \bigg]   {\mathbf {\tilde R}^{-1} } \mathbf r(\mathbf x)$,	
with $\mathbf r(\mathbf x)=(K(\mathbf x, \mathbf x_1),...,K(\mathbf x, \mathbf x_n))^T$ and $\tilde{\mathbf R}=\mathbf R+\frac{n}{\lambda_z} \mathbf I_n$ with the $(i, \, j)$ entry of $\mathbf R$ being $K(\mathbf x_i, \mathbf x_j)$. 
\label{thm:pred_SGP}
\end{theorem}

Theorem~\ref{thm:pred_SGP} indicates that the predictive mean of the discretized S-GaSP calibration model shrinks the predictive mean towards the mean function. When $\lambda_z=0$, the shrinkage is zero and the discretized S-GaSP becomes the GaSP. 


Interestingly, when the observations contain no noise, the predictive mean and variance of the field data from the GaSP calibration model and the discretized S-GaSP calibration model are \textit{exactly} the same, stated in the following Lemma~\ref{lemma:equivalence_mean_GP}. 
 
\begin{lemma}
\label{lemma:equivalence_mean_GP}
Assume the conditions in Theorem~\ref{thm:pred_SGP} hold. If $\sigma^2_0=0$, the predictive distribution of the field data at any $\mathbf x \in \mathcal X$ by the discretized S-GaSP model in (\ref{equ:delta_z_d}) is a multivariate normal distribution with the predictive mean and variance as follows
\begin{align*}
\E[y^F(\mathbf x) \mid \mathbf y^F, \bm \theta, \lambda, \lambda_z]&= f^M(\mathbf x, \bm \theta)+ \mathbf r^T(\mathbf x) \mathbf R^{-1} (\mathbf y^F- \mathbf f^M_{\bm \theta}), \\
\V[y^F(\mathbf x) \mid \mathbf y^F, \bm \theta, \lambda, \lambda_z]&=\sigma^2\left(K(\mathbf x, \mathbf x)-\mathbf r^T(\mathbf x) \mathbf R^{-1} \mathbf r(\mathbf x)\right),
\end{align*}
where $\mathbf r(\mathbf x)=(K(\mathbf x, \mathbf x_1),...,K(\mathbf x, \mathbf x_n))^T$ and   the $(i, j)$ entry of $\mathbf R$ is $K(\mathbf x_i, \mathbf x_j)$. 
\end{lemma}
 
\subsection{Emulating slow computer models and parameter estimation}
\label{subsec:parameter_est}

We discuss the computational issue and the parameter estimation in this section. All the approaches are implemented in the new {\tt RobustCalibration} {\sf R} package available on CRAN \cite{gu2018robustcalibrationpackage}. First, some mathematical models are  the numerical solutions of partial differential equations implemented as computer code,which is computationally expensive to run. In these cases,  one often uses a statistical emulator to approximate the computer model based on a set of model runs \citep{sacks1989design,santner2003design}. The GaSP emulator  from the {\tt RobustGaSP} {\sf R} package  is used to emulate the computer model when it is expensive to run.


We next discuss estimating the regularization parameters and kernel parameters, which were  held fixed in some studies. In practice, estimating these parameters can improve predictive performance. 
For any $\mathbf x_a=(x_{a1},....,x_{ap})^T$ and $\mathbf x_b=(x_{b1},....,x_{bp})^T$, the kernel is often assumed to have a product form in  model calibration \citep{kennedy2001bayesian}:
\begin{equation}
K(\mathbf x_a, \mathbf x_b)= \prod^p_{i=1}K_i(d_i),
\label{equ:prod_cor}
\end{equation}
where $d_i=|x_{ai}-x_{bi} |$ for $i=1,...,p$, and $K_i(\cdot)$ is a one dimensional kernel function. One widely used kernel function is the Mat{\'e}rn  kernel \citep{handcock1993bayesian}.
%
 When the roughness parameter is a half-integer,  the Mat{\'e}rn kernel has an explicit expression. For example, the Mat{\'e}rn kernel with roughness parameter being $5/2$ has the following expression
 \begin{equation}
K_i(d_i)=\left(1+\frac{\sqrt{5}d_i}{\gamma_i}+\frac{5d_i^2}{3\gamma_i^2}\right)\exp\left(-\frac{\sqrt{5}d_i}{\gamma_i}\right) \,,
\label{equ:matern_5_2}
\end{equation}
for $i=1,...,p$. A good feature of the Mat{\'e}rn kernel is the sample path of a GaSP is $\left\lfloor{\nu_{i}-1}\right\rfloor$ times differentiable, where $\nu_i$ is the roughness parameters. 
 



Denote $\bm \gamma=(\gamma_1,...,\gamma_p)^T$ the unknown range parameters in the covariance. The parameters in the discretized S-GaSP calibration model are the calibration parameters $\bm \theta$, the variance parameter of the noise $\sigma^2_0$, regularization parameter $\lambda=\sigma^2_0/(n\sigma^2)$, scaling parameter $\lambda_z$ and range parameters $\bm \gamma$. 
Simple algebra shows $\hat \sigma^2_{0, MLE}=\lambda S^2_{z_d}$, where $S^2_{z_d}=(\mathbf y^F-\mathbf f^M_{\bm \theta} )^T \mathbf {\tilde R}_{z_d}^{-1}(\mathbf y^F-\mathbf f^M_{\bm \theta} )$ with  $\mathbf {\tilde R}_{z_d}= \left( \mathbf R^{-1} + \lambda_z\mathbf I_n/n \right)^{-1} +\lambda n \mathbf I_n$ and $\mathbf {\tilde R}^{-1}_{z_d}=\lambda_z/(ng)+(\mathbf R+n\lambda\mathbf I_n/g)^{-1}/g^2$ with $g=\lambda\lambda_z+1$.   Marginalizing out $\delta_z(\cdot)$ and plugging $\hat \sigma^2_{0,MLE}$ into the likelihood of the discretized S-GaSP calibration model in (\ref{equ:delta_z_d}), one has the profile likelihood  
\begin{equation}
\ell_{z_d}(\bm \theta, \bm \gamma, \lambda,\lambda_z ) \propto -\frac{1}{2} \log| \mathbf {\tilde R}_{z_d} | -\frac{n}{2} \log(S^2_{z_d}). 
\label{equ:profile}
\end{equation} 
One may numerically maximize the profile likelihood in (\ref{equ:profile}) to estimate the parameters. Note $\lambda_z$ reflects one's tolerance of how good a mathematical model should predict the reality without the discrepancy function and thus this parameter may be chosen based on the expert knowledge. Because of the conditions discussed in Theorem~\ref{thm:L_2_convergence}, $\lambda_z$ may  be fixed to be proportional to $\lambda^{-1/2}$ or be related to the sample size. For all numerical examples, we fix $\lambda_z=( \lambda ||\tilde {\bm \gamma} || )^{-1/2} $, where $\lambda=\sigma^2_0/(\sigma^2 n)$ and $\tilde {\bm \gamma}=(\tilde { \gamma}_1,...,\tilde { \gamma}_{p})^T$, with $\tilde { \gamma}_i$  being the normalized range parameter (normalized by the length of each coordinate of the input variable). The inclusion of the range parameters is due to the confounding issue between the range parameter and the variance parameter of the process, whereas the ratio of these parameters can typically be estimated consistently from the data \citep{zhang2004inconsistent}.



 In the {\tt RobustCalibration} package, we implement the Bayesian method of estimating the parameters and making predictions. The  prior is assumed to follow $\pi(\bm \theta, \bm \gamma, \eta, \sigma^2)\propto \pi(\bm \theta) \pi(\bm \gamma, \eta)/\sigma^2$ with $\eta=\sigma^2_0/\sigma^2$ being the nugget parameter. Here $\pi(\bm \gamma, \eta)$ is chosen as the joint robust prior for the kernel parameters \citep{gu2018jointly,Gu2018robustness}, and $ \pi(\bm \theta) $ may be specified by the user to reflect experts' knowledge. We assume a uniform distribution of $ \pi(\bm \theta) $ in the numerical examples for demonstration purposes. Compared to the MLE approach, the uncertainty of the parameters can be assessed naturally through the posterior samples.

\section{Comparison between different calibration approaches}
\label{sec:comparison}
We compare a few calibration approaches in this section. First, one of most popular framework is the GaSP calibration approach  \citep{kennedy2001bayesian}, which models the discrepancy in (\ref{equ:gp_calibration}) as a GaSP. The mathematical model and discrepancy function are jointly estimated under the Bayesian framework. The S-GaSP approach is an extended version of GaSP calibration by placing more prior probability mass of the $L_2$ norm of the discrepancy near zero.  Consequently,  the calibrated mathematical models in the S-GaSP calibration fit the data better than the ones in the GaSP calibration.

The S-GaSP calibration approach was inspired by a few pioneering approaches seeking to find the $L_2$ minimizer of the calibration parameters \citep{tuo2015efficient,wong2017frequentist,plumlee2017bayesian}. The orthogonal Gaussian process proposed in \cite{plumlee2017bayesian} constrains the derivatives of the random $L_2$ norm of the discrepancy to be zeros, equivalently giving more prior probability mass of calibration parameters at the stationary points of the calibration parameters in terms of the $L_2$ loss. The S-GaSP model explores another transformation that avoids putting more prior probability mass at local maxima of $L_2$ loss of the discrepancy, and that has a closed-form likelihood function. The $L_2$ calibration was proposed  in \citep{tuo2015efficient}, where the reality is first estimated by nonparametric regression, and the calibration parameters are estimated by minimizing the $L_2$ loss between the reality and mathematical model. The LS calibration is proposed in \citep{wong2017frequentist}, where the calibration parameters are estimated by first minimizing the squared loss between the mathematical model and observations, and a nonparametric approach is applied to estimate the difference between the reality and mathematical model. 
 
 We use the following example to illustrate that jointly estimating  the discrepancy function and calibration parameters can be helpful for predicting the reality. 



\begin{example}
Let $y^F( x)=y^R(x)+\epsilon$, where $\epsilon$ independently follows $N(0,0.05^2)$ and $y^R(x)=g_1+g_2$, with $g_1=\sum^{\infty}_{j=1}j^{-1}\cos(5\pi (j-0.5)x)\sin(5j)$ and $g_2=\sum^{\infty}_{j=1}j^{-6}\cos(5\pi (j-0.5)x)\sin(5j)$. Let mathematical model be $f^M(x)=g_1 \theta$.  The goal is to predict $y^R(\cdot)$ and estimate $\theta$. 
\label{eg:compare_rates_all}
\end{example}

We compare the GaSP, S-GaSP, $L_2$ and LS calibration approaches using Example \ref{eg:compare_rates_all}. The GaSP model is used as the nonparametric regression approach in  the first step of $L_2$ calibration and the second step of the LS calibration. We assume kernel function $K(\cdot,\cdot)$ follows Mat{\'e}rn kernel function in (\ref{equ:matern_5_2}) for all methods. The calibration parameter, variance and  kernel parameters are estimated by the maximum likelihood estimator.

\begin{figure}
\centering
\begin{tabular}{ccc}
{\vspace{-.2in} \includegraphics[height=.3\textwidth,width=.34\textwidth ]{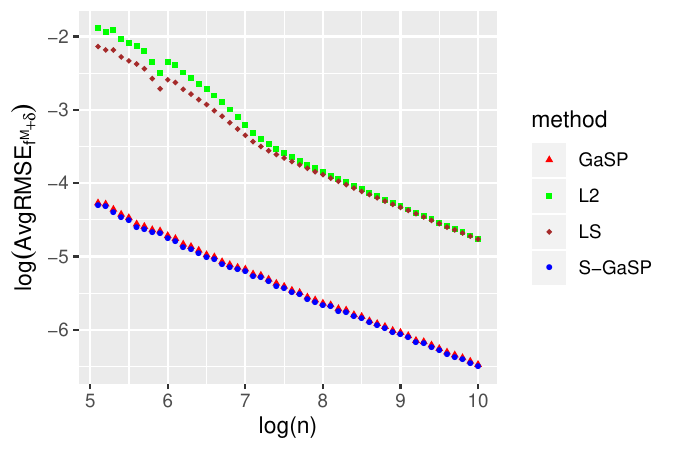}}
\includegraphics[height=.3\textwidth,width=.34\textwidth]{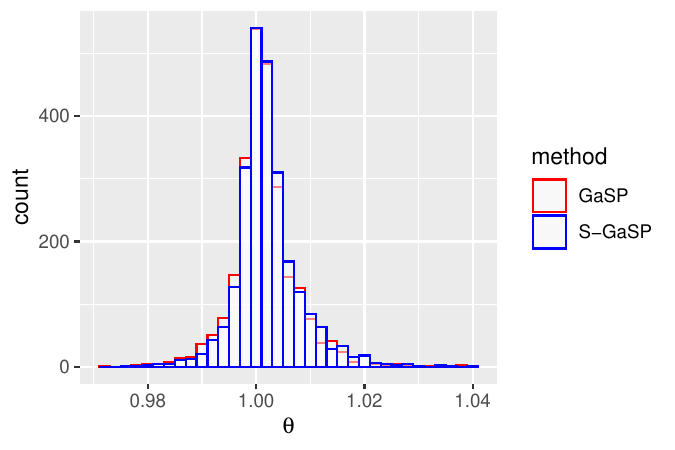} 
\includegraphics[height=.3\textwidth,width=.32\textwidth]{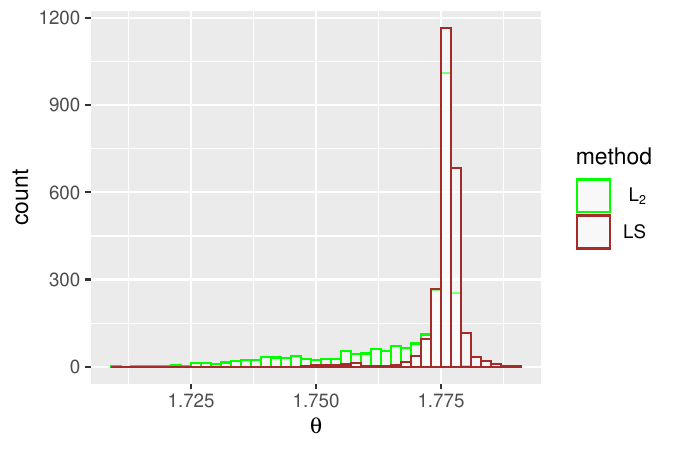} 
\end{tabular}
\caption{Comparison of different approaches in Example \ref{eg:compare_rates_all}.  In  the left panel, the logarithm of the $\mbox{AvgRMSE}_{f^M+\delta}$ of four calibration approaches are graphed at the logarithm of different sample sizes. The histogram of the estimated calibration parameter of each experiment  of different approaches are given in the middle panel and the right panel. }
\label{fig:compare_rate_all}
\end{figure}

In the left panel of Figure \ref{fig:compare_rate_all}, we found the predictive error by the GaSP and S-GaSP calibration is considerably smaller than the LS calibration and $L_2$ calibration approach. This is because the reality contains $g_1=\sum^{\infty}_{j=1}j^{-1}\cos(5\pi (j-0.5)x)\sin(5j)$, which is hard to be predicted by a nonparametric regression approach alone. The mathematical model specified herein, however, can explain this term with calibration parameter close to 1. The estimated calibrated parameter in both GaSP and S-GaSP is indeed close 1 in all experiments, which leads to better predictive performance.

The $L_2$ minimizer of the calibration parameters in this example is around 1.775. Note that the estimated calibration parameter S-GaSP calibration may not converge to the $L_2$ minimizer when sample size increases. This is because  $||y^R(\cdot)-f^M(\cdot,  \theta)||_{\mathcal H}$ is unbounded when $\theta \neq 1$, which violates the assumption A5. The calibrated computer computer with calibration parameter being around 1 improves the  predictive accuracy and interpretation, as the residual discrepancy term is a smooth term that is easy to be predicted. 

The GaSP and S-GaSP calibration approaches are always more accurate in predicting the reality than the two-step approaches. Indeed, for Example \ref{eg:ion_channel} discussed in the numerical comparison, the predictions in the $L_2$ calibration approach is the best among all methods, as the observations can be easily predicted through a nonparametric regression without the mathematical model. When the reality is complicated, the joint estimation by GaSP and S-GaSP calibration approaches seems to have a smaller predictive error, which will be illustrated by a few more numerical examples.

 As the sampling model of the observations is well-defined in both GaSP calibration and S-GaSP calibration, a Bayesian approach can be implemented and the uncertainty of the parameters can be assessed naturally through their posterior distributions. For the $L_2$ calibration, the asymptotic distribution of the estimator of the calibration parameter may be used to approximately quantify the uncertainty in parameter estimation \citep{tuo2015efficient}. A bootstrap approach is developed to assess the uncertainty of parameters for the LS calibration approach \citep{wong2017frequentist}. 








For a fast mathematical model, The computational complexity of all methods are typically dominated by nonparametric estimation of the discrepancy function, which is $O(n^3)$ in general with $n$ being the number of observations. Though the computation complexity is similar, the two-step approaches such as $L_2$ calibration and LS calibration are typically faster than the joint estimation in GaSP and S-GaSP calibration approach, as the kernel parameters and calibration parameters are estimated in two different steps. 











\section{Numerical study}
\label{sec:numerical}

In this section, we numerically compare the performance of several methods for calibration and prediction. We consider two loss functions as follows.
\begin{enumerate}[i.]
\item The $L_2$ loss between the reality and its estimator: $L_2(\hat y^R(\cdot) )=\| y^R(\cdot)- \hat y^R(\cdot) \|^2_{L_2}$.
\item The $L_2$ loss between the reality and calibrated mathematical model:  $L_2(\hat {\bm \theta})=\| y^R(\cdot)- \hat f^M(\cdot, \hat{\bm \theta})  \|^2_{L_2}$, where $ \hat{\bm \theta}$ is the estimator of the calibration parameter.
\end{enumerate}

  The first criterion is our primary consideration, because the out of sample prediction for the reality examines how well we can reproduce the reality.
The second criterion describes how well the calibrated mathematical model fits reality in terms of the $L_2$ loss \citep{tuo2015efficient}. The parameters in a mathematical model often have scientific interpretation, whereas a non-linear discrepancy function might not be interpretable. Thus the discrepancy function is not used for prediction in the second criterion. 


To numerically evaluate the performance under two prediction criteria, we first compute $\mbox{AvgRMSE}_{f^M+\delta}$ in Equation (\ref{equ:AvgRMSE_fm_delta}) of different methods. We  compute $\mbox{AvgRMSE}_{f^M}$ through replacing $\hat y^R_i(\mathbf x^*_j)$ in Equation (\ref{equ:AvgRMSE_fm_delta}) by the calibrated computer model output. For assessing the uncertainty in predictions, we also compute the average length of predictive interval and  proportion of the observations covered by the $95\%$ predictive interval   defined below: 
\begin{align*}
{L_{CI}(95\%)} =& \frac{1}{{N{n^{*}}}}\sum\limits_{j = 1}^N \sum\limits_{i = 1}^{{n^{*}}} {\Length\{C{I_{ij}}(95\% )\} } \,, \\
{P_{CI}(95\%)} =& \frac{1}{N{n^{*}}}\sum\limits_{j = 1}^N {\sum\limits_{i = 1}^{n^{*}} 1\{y^R_{j}(\mathbf x^{*}_i)\in C{I_{ij}}(95\% )\}}\,, 
\end{align*}
where $C{I_{ij}}(95\% )$ is the $95\%$ predictive interval; $N$ and $n^*$ are the total number of experiments and the number of held-out test data in each experiment, respectively.  An efficient method should have small  $\mbox{AvgRMSE}_{f^M+\delta}$ and $\mbox{AvgRMSE}_{f^M}$, short predictive interval and ${P_{CI}(95\%)}$ close to the nominal $95\%$. For the real examples, we replace the test reality output by the held-out observations for out-of-sample predictions.


We numerically compare 5 different methods. The first and second methods are the GaSP calibration and S-GaSP calibration, respectively, both implemented in the full Bayesian framework. The scaling parameter of the S-GaSP is fixed to be $\lambda_z=( \lambda ||\tilde {\bm \gamma} || )^{-1/2} $, where $\lambda=\sigma^2_0/(\sigma^2 n)$ and $\tilde {\bm \gamma}=(\tilde { \gamma}_1,...,\tilde { \gamma}_{p})^T$, with $\tilde { \gamma}_i$  being the normalized range parameter (normalized by the length of each coordinate of the input variable). The rest of model parameters and calibration parameters are sampled from  posterior distributions. The third  and fourth approaches are the $L_2$ calibration and LS calibration, respectively. The GaSP regression using the {\tt RobustGaSP} {\sf R} package is used to estimate the reality in the first step of the $L_2$ calibration, and the residual (between the calibrated mathematical model and reality) in the second step of the LS calibration.  The kernel function $K(\cdot,\,\cdot)$ is assumed to follow (\ref{equ:matern_5_2}) in all methods for demonstration purposes.   We also include the method with no discrepancy function computed under the  Bayesian framework. The GaSP,  S-GaSP and no-discrepancy calibration approaches are implemented in the {\tt RobustCalibration} {\sf R} package.
\subsection{Simulated exanple}
\label{subsec:simulated_eg}


\begin{example}
Let $y^F(\mathbf  x)=y^R(\mathbf x)+\epsilon$, where $\mathbf x=(x_1, x_2) \in \mathcal [0,1]^2$, $y^R(\mathbf x)=\sin(0.2\pi x_1)x_2 +\sin(2\pi x_1)x_2+1$ and $\epsilon \overset{i.i.d.}{\sim} N(0,0.1^2)$. The mathematical model is $f^M(\mathbf x, \bm \theta)=\sin(\theta_1 x_1)x_2+\theta_2$ with $\theta_1 \in [0,\, 10]$ and $\theta_2\in \mathbb R$.  
\label{eg:sin_2d}
\end{example}



We first consider a simulated study in Example \ref{eg:sin_2d} and test two configurations, where with sample sizes of the observations are taken to be $n=25$ and $n=50$, respectively.  For each configuration, we test $N=100$ experiments with $n^*=2500$ reality output equally spaced in each interval held-out for testing. The input variable in each experiment is generated from the maximin Latin hypercube design (\cite{santner2003design}).



\begin{table}[t]
\begin{center}
\begin{tabular}{llllll}
  \hline
     n=25                & GaSP   &  S-GaSP &  $L_2$   & LS & No-discrepancy   \\
  \hline
    $\mbox{AvgRMSE}_{f^M+\delta}$                   & {0.0556} &{0.0558}  &0.103  &0.0702  & / \\
    $\mbox{AvgRMSE}_{f^M}$                      &  0.143& {0.131} & {0.131} & {0.131} &{0.130} \\
    $L_{CI}(95\%) $                                          & {0.206} &  0.209 &0.434 &0.389 &0.654 \\
  $P_{CI}(95\%) $                                          & 0.931 &  0.932 &0.921 & {0.958} &0.987 \\
    \hline
     n=50                & GaSP  &  S-GaSP & $L_2$  & LS  &No-discrepancy  \\
  \hline
    $\mbox{AvgRMSE}_{f^M+\delta}$               & ${ 0.0405} $&${0.0403}$&$0.0655$ &${0.0437}$ &$/$ \\
    $\mbox{AvgRMSE}_{f^M}$                & $0.144 $& $0.130$ &${0.128}$ &${ 0.128}$ &${ 0.128}$ \\
      $L_{CI}(95\%) $                                          & 0.153 &  {0.152} &0.439 &0.402 &0.632 \\
  $P_{CI}(95\%) $                                          & {0.938} &  0.935 &0.992 &0.997 &0.991 \\

    \hline

\end{tabular}
\end{center}
\caption{Predictive performance by different methods for Example~\ref{eg:sin_2d}.}
\label{tab:AvgRMSE_eg3}
\end{table}

The predictive errors by different approaches are given in Table \ref{tab:AvgRMSE_eg3}. First  the  $\mbox{AvgRMSE}_{f^M+\delta}$  of all methods are better than  $\mbox{AvgRMSE}_{f^M}$ for almost all methods, indicating that a nonparametric model can improve the predictive performance. Second, the GaSP and S-GaSP calibration approaches performs better than the two-step LS and $L_2$ calibration methods in terms of    $\mbox{AvgRMSE}_{f^M+\delta}$. Though around $95\%$ of the held out test data are covered by the $95\%$ predictive interval in all methods, the average lengths of the predictive interval by the GaSP and S-GaSP calibration are shorter than the two-step approaches.

 For the $L_2$ calibration approach, the  mathematical model is not used in the first step,  as the parameters in the mathematical model are estimated in the second step to minimize the $L_2$ loss.  The high frequency term ($\sin(2\pi x_1)x_2$) makes the predictions by the nonparametric regression less accurate than jointly estimate of the calibration parameters and discrepancy function, as the high frequency term can be explained by the mathematical model. The comparison between predictions by using the nonparameteric regression alone and S-GaSP calibration in the first experiment with $n=60$, for example, is shown in Figure \ref{fig:reality_diff_sin_2d}. Indeed the S-GaSP calibration seems to have a smaller predictive error.


\begin{figure}[t]
	\centering
			\includegraphics[height=.35\textwidth,width=1\textwidth ]{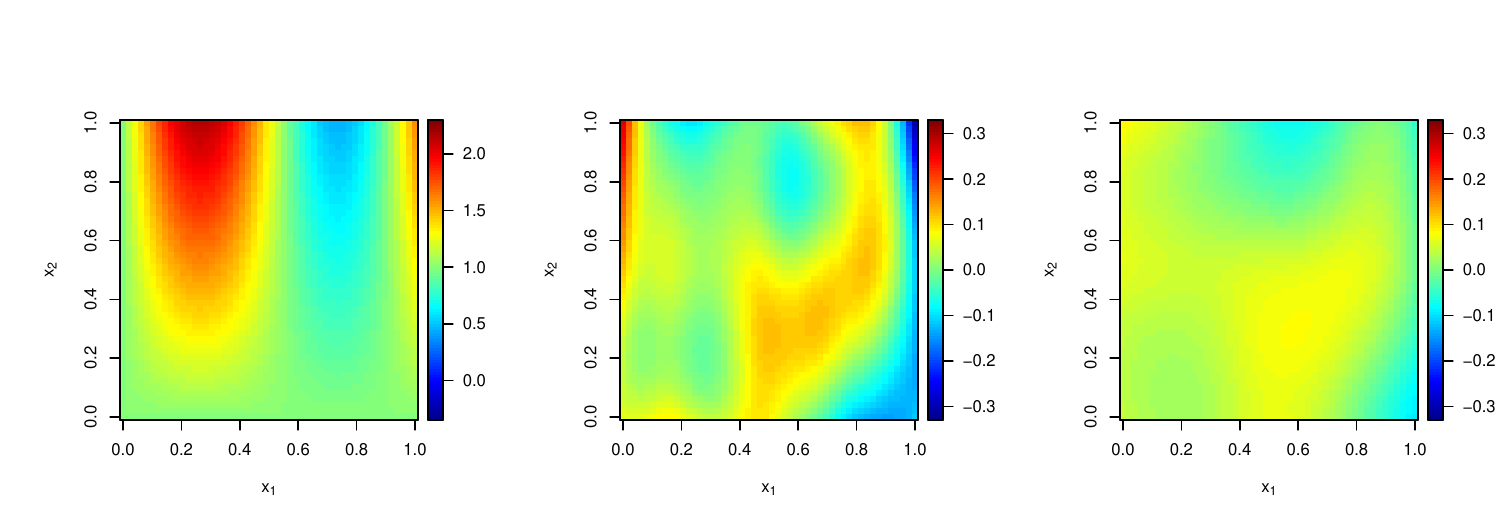}  \vspace{-.3in}
	\caption{The reality in an experiment of Example~\ref{eg:sin_2d} is graphed in the left panel.  The residuals between the reality and prediction by a  GaSP regression (without the mathematical model) are graphed in the middle panel, whereas those by the S-GaSP calibration are graphed in right panel. The middle and right panels have the same scale of color. }
	\label{fig:reality_diff_sin_2d}
\end{figure}

The estimated calibration parameters of Example~\ref{eg:sin_2d} are graphed in  Figure~\ref{fig:sin_2d_est}. In the left panel, the estimation of $\theta_1$ using the LS, $L_2$ and no-discrepancy methods is close to the $L_2$ minimizer (graphed as the solid line), whereas the estimation of $\theta_1$ using the GaSP and the S-GaSP is, in fact, closer to $2\pi$. This is because the model complexity is naturally built into the calibration: an estimated $\theta_1$ that is close to $2\pi$ makes the prediction better, since the high frequency term can be approximately explained by the mathematical model.

\begin{figure}[t]
	\centering
			\includegraphics[height=.4\textwidth,width=1\textwidth ]{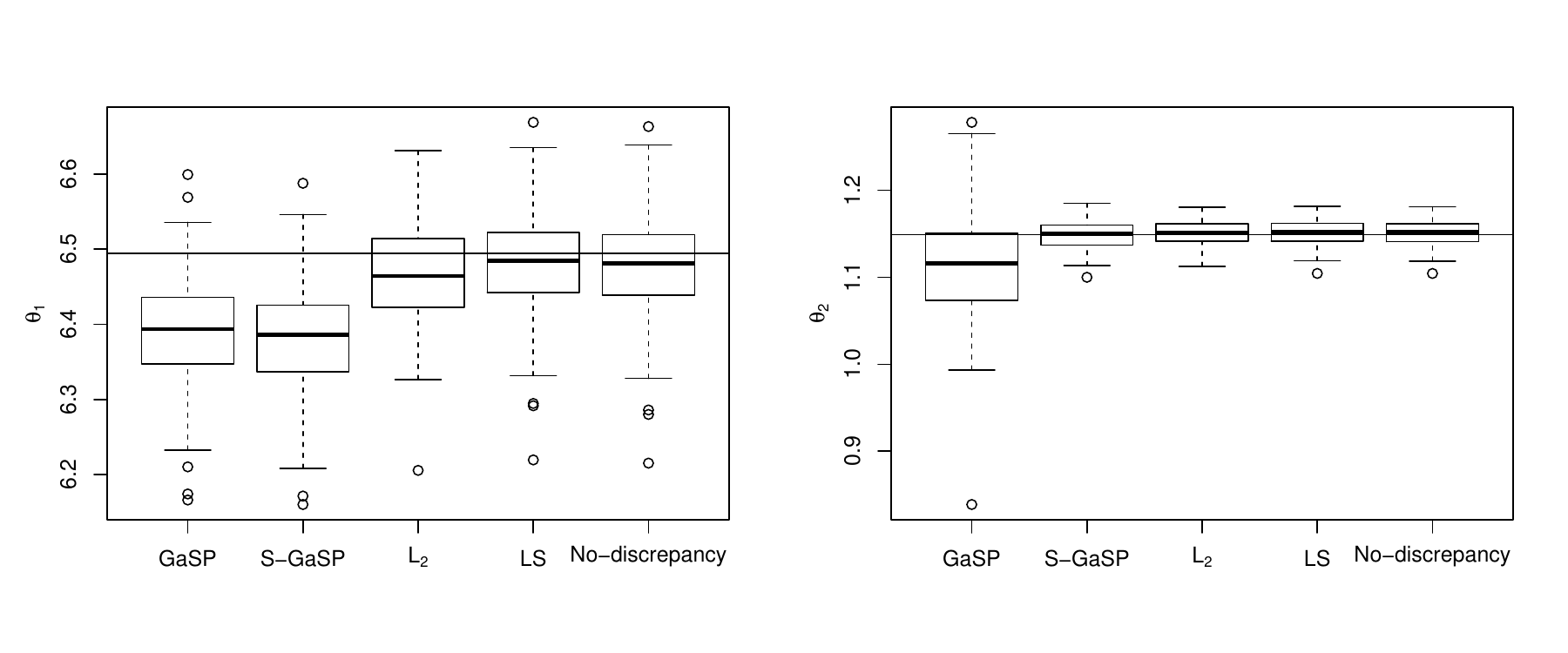}   \vspace{-.3in}
			
	\caption{Boxplots of the estimated calibration parameters from the GaSP, the S-GaSP, $L_2$ and LS calibration approaches for each experiment in Example~\ref{eg:sin_2d}.  The solid lines are the $L_2$ minimizer, which is around $6.48$ and $1.15$ for $\theta_1$ and $\theta_2$, respectively. }
	\label{fig:sin_2d_est}
\end{figure}



The GaSP calibration approach produces the largest $\mbox{AvgRMSE}_{f^M}$ compared to  other calibration methods $\mbox{AvgRMSE}_{f^M}$, as illustrated in Table \ref{tab:AvgRMSE_eg3}.  This is probably caused by  the estimation of $\theta_2$ (a mean parameter) shown in the right panel in Figure \ref{fig:sin_2d_est}. The estimates of $\theta_2$ by the S-GaSP calibration, LS, $L_2$ and no-discrepancy methods are all close to the $L_2$ minimizer, whereas the one in the GaSP calibration seems to be smaller the other ones.

Example \ref{eg:sin_2d} indicates that the calibrated mathematical model that minimizes the $L_2$ loss between the reality and mathematical model might not always be the optimal choice for predictions. In this example, when $\theta_1$ is estimated to be close to $2\pi$, the predictive accuracy can be improved significantly, and the calibrated computer model only produces slightly larger error than the one by the $L_2$ minimizer. The other parameter, $\theta_2$, is a mean parameter. The two types of errors are both small when $\theta_2$ is estimated to be  close to the $L_2$ minimizer.
The S-GaSP calibration model seems to do well in both sides. It predicts the reality as accurately as the GaSP calibration model with the assistance of the calibrated mathematical model. The calibrated mathematical model using the S-GaSP is also closer to the reality than the one using the GaSP calibration. 

\subsection{Chemical system interaction}
\begin{example}
Consider the system interaction between  two chemical substances $y_1$ and $y_2$:
\begin{align*}
\dot{y}_1(t)&=10^{\theta_1-3} y_1(t) \\
\dot{y}_2(t)&=10^{\theta_1-3} y_1(t) - 10^{\theta_2-3}{y}_2(t)
\end{align*}
where $2$ repeated observations of the second chemical substance are measured at each of the 6 time points $t=10,20,40,80,160, 320$ in \cite{box1956application}. The goal is to estimate $\theta_1$ and $\theta_2$, and to predict the values of the chemical substance across time. 
\label{eg:box}
\end{example}

We consider $(\theta_1, \theta_2) \in [0.5, 1.5]^2$. As the number of observations are limited, we first perform a leave-one-out comparison by holding out two repeated observations for prediction at each time point. We replace the reality in each criterion by the held-out observations to test each approach. The predictive performance of the leave-one-out comparison for Example \ref{eg:box} is given in Table \ref{tab:AvgRMSE_eg_box}.

\begin{table}[h]
\begin{center}
\begin{tabular}{llllll}
  \hline
                     & GaSP   &  S-GaSP &  $L_2$   & LS & No-discrepancy   \\
  \hline
    $\mbox{AvgRMSE}_{f^M+\delta}$                   & 4.98 & {5.05}  &{10.2}  &7.80  & / \\
    $\mbox{AvgRMSE}_{f^M}$                      &  11.8 & {10.4}  &{9.34}   & { 10.9} &{9.63} \\
    $L_{CI}(95\%) $                                          & {32.8} &  44.6 &52.5 &42.6 &47.0 \\
  $P_{CI}(95\%) $                                          & $100\%$ &  $100\%$ &$91.7\%$ &$100\%$  &$100\%$  \\
    \hline

\end{tabular}
\end{center}
\caption{Predictive performance by different methods for Example~\ref{eg:box}.}
\label{tab:AvgRMSE_eg_box}
\end{table}

First the predictive error $\mbox{AvgRMSE}_{f^M+\delta}$ is smaller than $\mbox{AvgRMSE}_{f^M}$ in both GaSP and S-GaSP calibration, indicating that jointly estimating the mathematical model and discrepancy function improves the predictive accuracy than the two-step approaches. On the other hand, the GaSP calibration has the largest predictive error using the calibrated mathematical model alone. The $L_2$ calibration and no-discrepancy calibration have a smaller  error $\mbox{AvgRMSE}_{f^M}$ compared to other methods. The S-GaSP calibration has relatively small predictive error under both predictive criteria. 


\begin{figure}[t]
\centering
\begin{tabular}{ccc}
\includegraphics[height=.33\textwidth,width=.333\textwidth ]{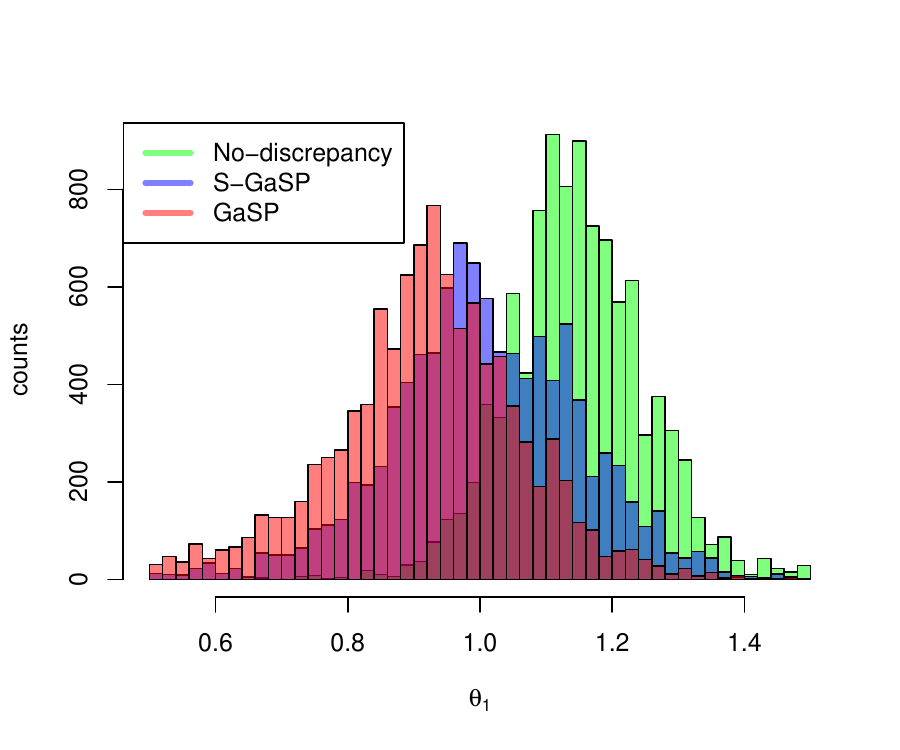}
\includegraphics[height=.33\textwidth,width=.333\textwidth ]{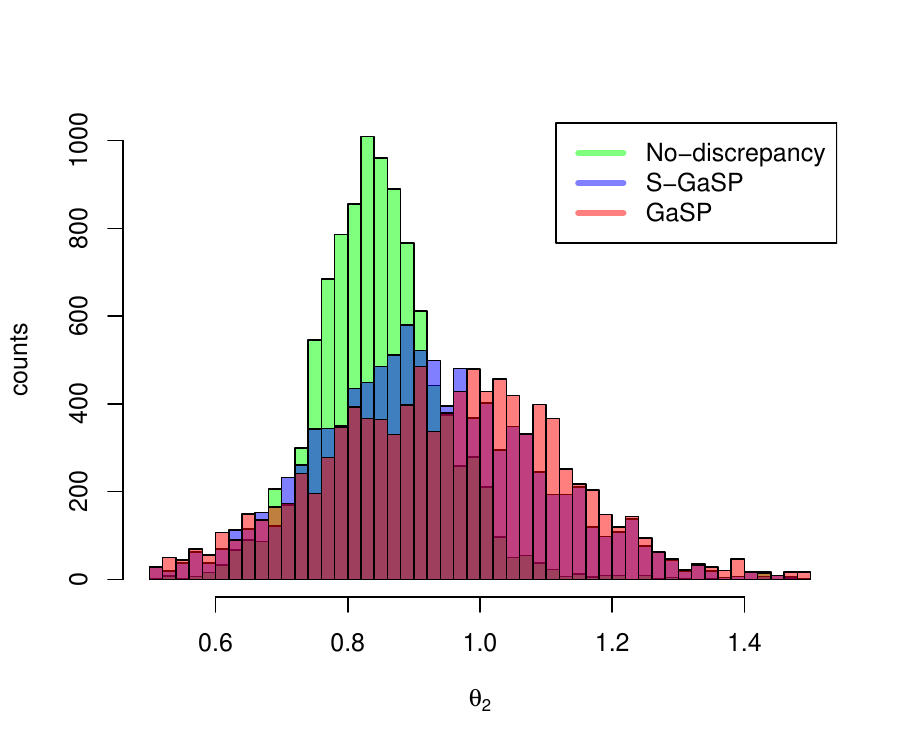}
\includegraphics[height=.33\textwidth,width=.333\textwidth ]{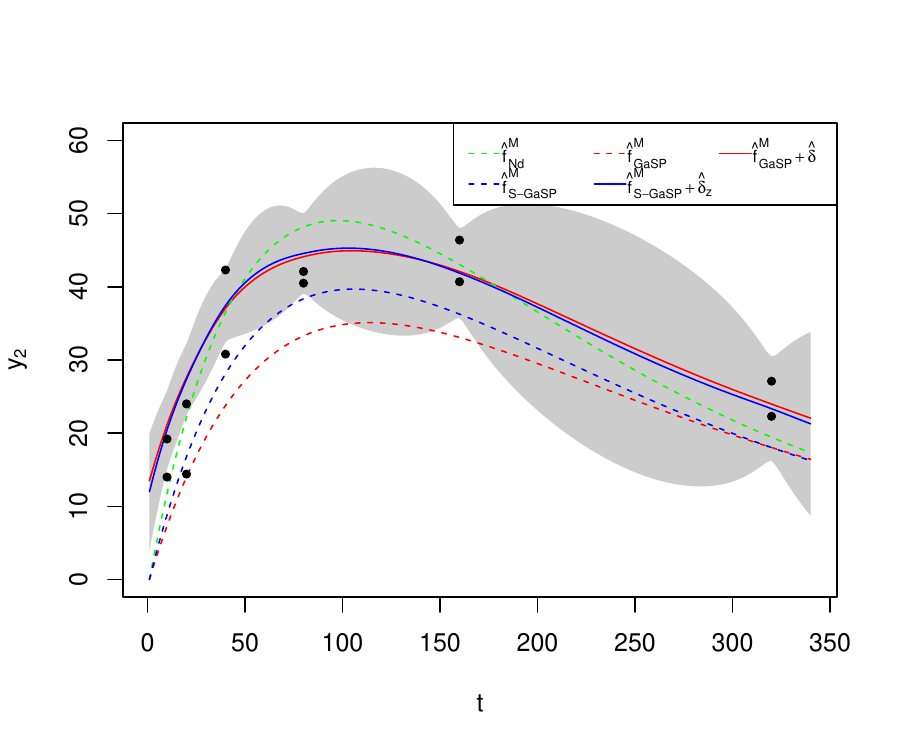}

\end{tabular}
\caption{Posteriors samples of calibration parameters and predictions from the GaSP, S-GaSP and no-discrepancy calibration approaches. The shaded area in the right panel is the predictive interval of the reality from the S-GaSP calibration approach.}
\label{fig:post_samples_pred_box}
\end{figure}

The left panel in Figure \ref{fig:post_samples_pred_box} gives posterior samples of calibration parameters by different methods in Example \ref{eg:box}. When there is no discrepancy function, the posterior samples of the calibration parameters are close to $(1.1, 0.8)$, which is inline with the result in \cite{box1956application}. The estimated calibration parameters of the no-discrepancy method is close to the $L_2$ and LS methods, whereas the posterior samples of the calibration parameters from the GaSP calibration and S-GaSP calibration reflects more uncertainty. 

The predictions in the GaSP, S-GaSP and no-discrepancy calibration  are graphed in right panel of Figure \ref{fig:post_samples_pred_box}. The calibrated mathematical model in the S-GaSP calibration seems to fit the reality better than the one by the GaSP calibration. The calibrated mathematical model in the no-discrepancy calibration fits the observations the best, as the residuals between the mathematical model and observations are assumed to be random noises.  Although we do not know the reality of this  example, the predictive performance of GaSP and S-GaSP calibration is better than the other approaches as shown in the first row in Table \ref{tab:AvgRMSE_eg_box}, suggesting that there might be a systematic discrepancy between the  observations and the mathematical model. 


\subsection{Ion Channel experiments}
In the last example, we consider calibrating the mathematical model for sodium ion channels using real observations from whole cell voltage clamp experiments \citep{plumlee2016calibrating}. 
\begin{example} 
The data sets consists of 19 observations of normalized current needed to maintain the
membrane potential at $-35mV$ over time \citep{plumlee2017bayesian}. Denote the input variable $x$ the natural logarithm of time. The mathematical model   has the following expression:
\[f^M(x, \bm \theta)=\mathbf e^T_1 \exp[\exp(x) \mathbf A(\bm \theta)]\mathbf e_4, \]
where $\bm e_i$ is a column vector with $1$ at the $i$th element and 0 for the rest of components,  the first exp is the matrix exponential,  $\bm \theta=(\theta_1, \theta_2,\theta_3)^T$,  and 
\[\mathbf A(\bm \theta)=\begin{pmatrix}
-\theta_2 -\theta_3 & \theta_1 & 0 & 0 \\
\theta_2 & -\theta_1-\theta_2 & \theta_1 & 0 \\
0 & \theta_2 & -\theta_1-\theta_2 & \theta_1 \\
0 & 0 & \theta_2 & -\theta_1
\end{pmatrix} \]
\label{eg:ion_channel}
The range of calibration parameter considered herein is $\theta_i \in [0,10]$ for $i=1,2,3$. 
\end{example}

 
We first consider a leave-one-out comparison for the ion channel experiment. The predictive errors by different approaches are given in Table \ref{tab:AvgRMSE_eg_ion_channel}. First the average root of mean squared errors under two criteria are all very small. Here since the input variable is only one dimensional and the real observations (graphed in the bottom left panel in Figure \ref{fig:eg_5_ion_channel}) seems to be very smooth, using the GaSP regression by the {\tt RobustGaSP} {\sf R} package  gives very high predictive accuracy  already. Including the mathematical model does not improves the prediction accuracy in this example. Therefore, the predictive error by the $L_2$ calibration is the smallest, as the GaSP regression without the mathematical model is used for predicting the reality in the first step. 

\begin{table}[h]
\begin{center}
\begin{tabular}{llllll}
  \hline
                         & GaSP   &  S-GaSP &  $L_2$   & LS & No-discrepancy   \\
  \hline
    $\mbox{AvgRMSE}_{f^M+\delta}$                   &$1.77 \times 10^{-3}$ & $1.52 \times 10^{-3}$  &$8.26 \times 10^{-4}$  &$2.17 \times 10^{-3}$  & / \\
    $\mbox{AvgRMSE}_{f^M}$                      &  $1.13 \times 10^{-2}$ &  $5.62 \times 10^{-3}$  &$4.83 \times 10^{-3}$   & $6.49 \times 10^{-3}$ &$6.45 \times 10^{-3}$ \\
    $L_{CI}(95\%) $                                          & $5.64 \times 10^{-3}$ &  $9.58 \times 10^{-3}$ &$9.45 \times 10^{-3}$ &$5.80 \times 10^{-3}$ &$2.58 \times 10^{-2}$ \\
  $P_{CI}(95\%) $                                          & $94.7\%$ &  $100\%$ &$84.2\%$ & $78.9\%$ &$84.2\%$ \\
    \hline

\end{tabular}
\end{center}

\caption{Predictive performance by different methods for Example~\ref{eg:ion_channel}.}
\label{tab:AvgRMSE_eg_ion_channel}
\end{table}

In the second place, the GaSP calibration approach has a substantial larger predictive error by using the calibrated computer model alone, shown in the second row  in Table \ref{tab:AvgRMSE_eg_ion_channel}. Besides, the lengths of predictive intervals are all very small in all approaches. The predictive  interval of the GaSP calibration seems to be most faithful for uncertainty assessment, as the proportion of the observations covered in the  interval are closest to the nominal level. 

\begin{figure}[t]
\centering
\begin{tabular}{cc}
\includegraphics[height=.35\textwidth,width=.5\textwidth ]{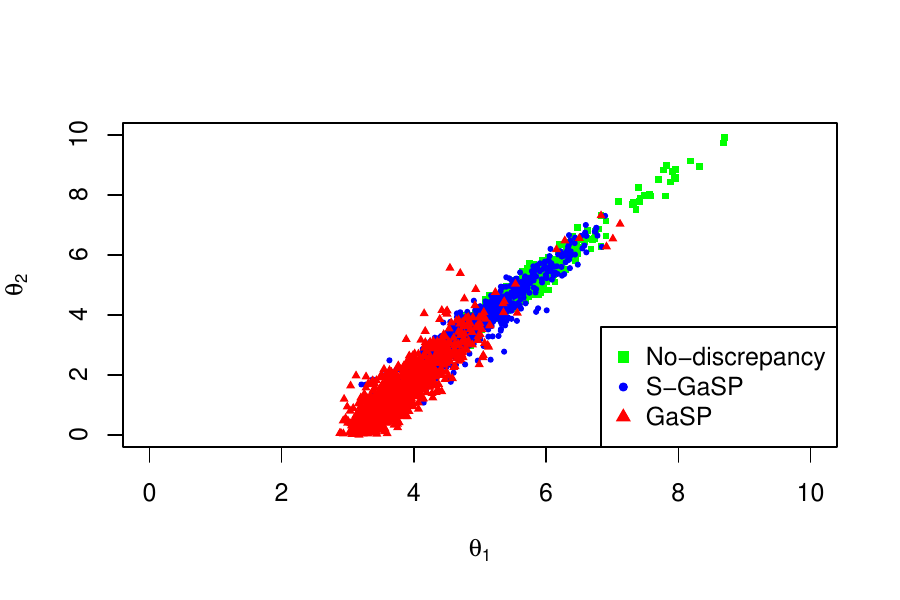}
\includegraphics[height=.35\textwidth,width=.5\textwidth ]{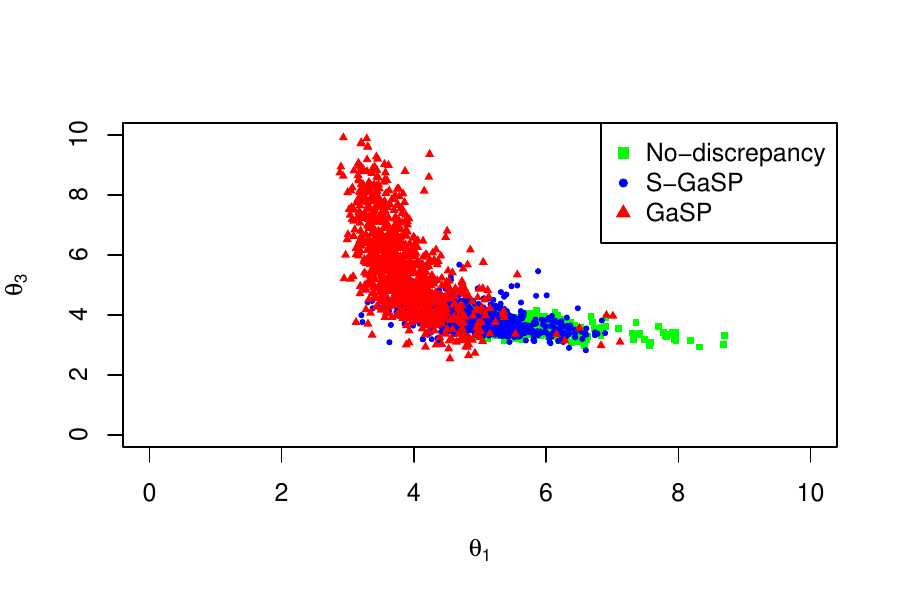} \vspace{-.5in} \\
\includegraphics[height=.35\textwidth,width=.5\textwidth ]{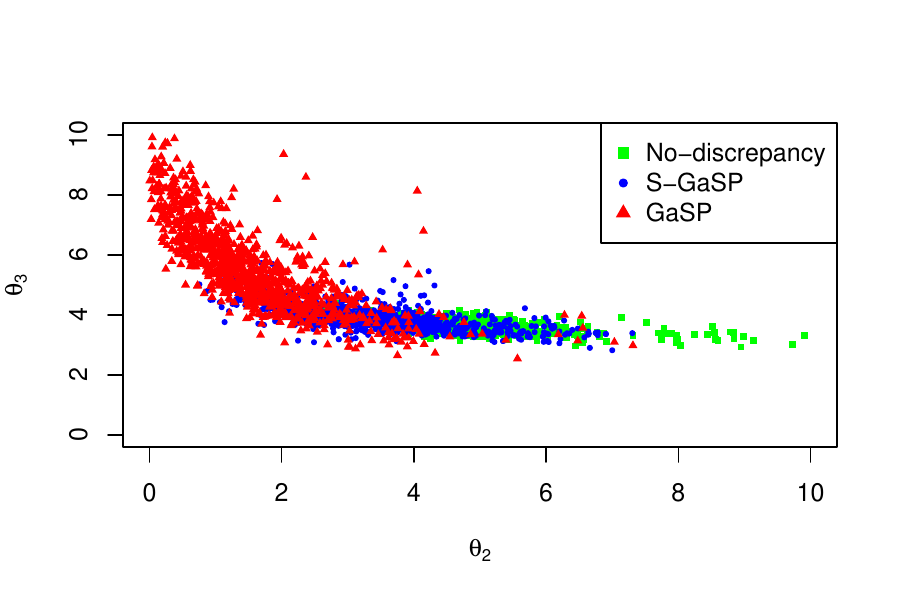}
\includegraphics[height=.35\textwidth,width=.5\textwidth ]{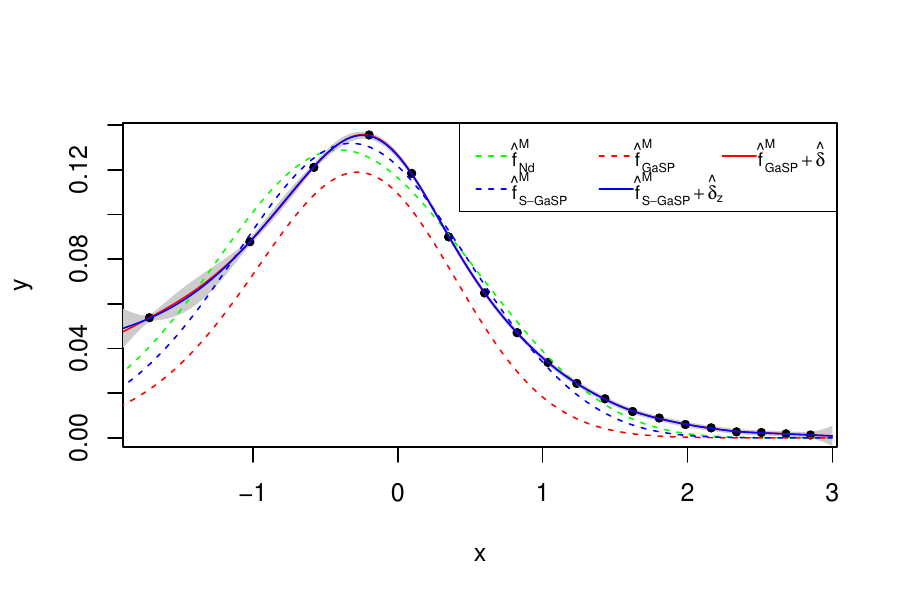} 

\end{tabular}
\caption{The posterior samples and predictions from no-discrepancy calibration, S-GaSP and GaSP calibration.   The shaded area in the bottom right panel is the 95\% predictive credible interval of the mean of the observations. }
\label{fig:eg_5_ion_channel}
\end{figure}

The posterior samples and estimation of the reality by different calibration approaches are graphed in Figure \ref{fig:eg_5_ion_channel}. For better visualization, we thin the posterior samples by 10 and only graph one tenth of the total posterior samples.  In this example, the posterior samples of the S-GaSP are closer to the ones with the no-discrepancy calibration. Due to this reason, the calibrated computer model by the S-GaSP calibration (graphed as the dashed blue curve in the lower right panel) fits the observations well, whereas the calibrated computer model by the GaSP calibration  (graphed as the dashed red curve in the lower right panel) underestimates the values of the observations. This identifiability issue of the GaSP calibration in this example was also reported in \citep{plumlee2017bayesian}, whereas the calibrated mathematical model fits the observations reasonably well in the S-GaSP calibration.


 


 




\section{Concluding remarks}
\label{sec:conclusion}
We have introduced the scaled Gaussian stochastic process (S-GaSP) for the calibration and prediction. We showed that under certain some routinely used assumptions, the predictive mean of the S-GaSP calibration model converges to the reality as fast as the GaSP calibration with some suitable choice of the regularization parameter and scaling parameter. The estimated calibration parameters in the S-GaSP calibration  converges to the $L_2$ minimizer  with the same choice of the regularization parameter and scaling parameter, whereas those in the GaSP calibration typically do not converge to the $L_2$ minimizer. The results rely on the  the orthogonal series representation of the processes studied in this work. The computational complexity of the discretized S-GaSP calibration is the same as the GaSP calibration. Both GaSP, S-GaSP calibration and calibration without a discrepancy function were implemented in the {\tt RobustCalibration} {\sf R} package under the Bayesian framework.

The numerical studies indicate that jointly estimating the calibration parameters and discrepancy function in GaSP and S-GaSP calibration can improve the predictive accuracy, if the reality are too complicated to be predicted precisely by nonparametric regression alone. The mathematical model, which typically contains some information of the trend and shape of the reality function, can be helpful in predictions. We also empirically found that the calibrated mathematical model that minimizes the $L_2$ distance between the reality and mathematical model may not always be the best for reducing the predictive error, as the residuals may contain information that could be hard to be modeled by the nonparametric regression. The S-GaSP calibration give as least as accurate predictions  as the GaSP calibration, and the calibrated mathematical model by the S-GaSP calibration fits the real observations much closer than the one by the  S-GaSP calibration in almost all examples. 




We outline a few extensions of the S-GaSP model below. From the theoretical perspective, we did not study the convergence of the discretized S-GaSP, whereas the numerical studies indicate the convergence rate from the discretized S-GaSP is the same as the S-GaSP. Secondly we illustrated that the S-GaSP calibration can be implemented in both Frequentist and Bayesian ways. The contraction rate of the S-GaSP under the Bayesian framework, however, was not studied. The studies of the contraction rate of the GaSP regression may be extended to achieve this goal \citep{van2009adaptive,bhattacharya2014anisotropic}. Thirdly, we fix the scaling parameter in the S-GaSP calibration as a function of the sample size, whereas historical information may be used to develop a reasonable prior for this parameter \citep{salter2019uncertainty}. Furthermore, the S-GaSP calibration framework may be extended to include a model of correlated noises, as the field observations may contain time series and images  \citep{anderson2017abundant}.





	
\bibliographystyle{apalike}
\bibliography{References_2020}

\pagebreak
	
\beginsupplement
{\center \Large {Supplementary materials for a theoretical framework of the scaled Gaussian stochastic process in prediction and calibration} }

\quad

All the formulas in this supplementary materials are cross-referenced in the main body of the article.  We first give a brief introduction of the Gaussian stochastic process model and reproducing kernel Hilbert space in Section~\ref{sec:GP}. The proof for Section~\ref{sec:SGP} is given in Section~\ref{sec:proof_SGP}. The proof for Theorem~\ref{thm:nonparametric} and two auxiliary lemmas are provided in Section~\ref{sec:proof_nonparametric_SGP}. Section~\ref{proof:theta_convergence} encloses the proof for Theorem~\ref{thm:L_2_convergence} and provides additional results regarding the convergence of S-GaSP calibration when kernel parameters are estimated.

\section{Background: Gaussian stochastic process}
\label{sec:GP}
Assume the mean and trend of the reality are properly modeled in the mathematical model. Consider to model the unknown discrepancy function in the calibration model (\ref{equ:gp_calibration}) via a real-valued zero-mean Gaussian stochastic process $\delta(\cdot)$ on a $p$-dimensional input domain $\mathcal X$,
\begin{equation}
\delta(\cdot)\sim \mbox{GaSP}(0, \sigma^2 K(\cdot, \cdot)),
\label{equ:GP}
\end{equation}
where $\sigma^2$ is a variance parameter and $K(\mathbf x_a, \mathbf x_b)$ is the correlation for any $\mathbf x_a, \mathbf x_b \in \mathcal X$, parameterized by a kernel function. For simplicity, we assume $\mathcal X=[0,\,1]^{p}$ in this work.
 
For any $\{\mathbf x_1,...,\mathbf x_n \}$, the outputs $\left(\delta(\mathbf x_1),...,\delta(\mathbf x_n)\right)^T$ follow a multivariate normal distribution
\begin{equation}
[\delta(\mathbf x_1),..., \delta(\mathbf x_n)\mid \sigma^2, \mathbf R ] \sim \mbox{MN}( \mathbf 0, \sigma^2 \mathbf R),
\end{equation}
where the $(i,\, j)$ entry of $\mathbf R$ is $K(\mathbf x_i, \mathbf x_j)$. Some frequently used kernel functions include the power exponential kernel and the Mat{\'e}rn kernel. We defer the issue of estimating the parameters in the kernel function in Section~\ref{sec:discretized_sgasp} and assume $K(\cdot, \, \cdot)$ is known for now.

The reproducing kernel Hilbert space (RKHS), denoted as $\mathcal H$, attached to the Gaussian stochastic process $\mbox{GaSP}(0, \sigma^2 K(\cdot, \cdot))$, is the completion of the space of all functions
\[\mathbf x \to \sum^k_{i=1} w_i K(\mathbf x_i, \mathbf x), \quad w_1,...,w_k \in \mathbb R,\, \mathbf x_1,..., \mathbf x_k, \mathbf x \in \mathcal X, \, k\in \mathbb N, \]
with the inner product 
\[ \left\langle\sum^k_{i=1} w_i K(\mathbf x_i,  \cdot) , \sum^m_{j=1} w_j K(\mathbf x_j,  \cdot) \right\rangle_{\mathcal H}=\sum^k_{i=1}\sum^m_{j=1} w_i w_j K(\mathbf x_i, \mathbf x_j). \]
  
For any function $f(\cdot)\in \mathcal H$,  denote $\| f\|_{\mathcal H}=\sqrt{\langle f, f \rangle_{\mathcal H} }$ the RKHS norm or the native norm. Because the evaluation maps in RKHS are bounded linear, it follows from the Riesz representation theorem that for each $\mathbf x \in \mathcal X$ and $f(\cdot) \in \mathcal H$, one has $f(\mathbf x)= \langle f(\cdot), \, K(\cdot, \mathbf x) \rangle_{\mathcal H}$. 
 
 
Denote $L_2(\mathcal X)$ the space of square-integrable functions $f: \mathcal X \to \mathbb R$ with $\int_{\mathbf x \in \mathcal X} f^2(\mathbf x)d \mathbf x<\infty$. We denote $\langle f, g \rangle_{L_2(\mathcal X)}:= \int_{\mathbf x \in \mathcal X} f(\mathbf x)g(\mathbf x)d \mathbf x$ the usual inner product in $L_2(\mathcal X)$. By Mercer's theorem, there exists an orthonormal sequence of continuous eigenfunctions $\{\phi_k\}^{\infty}_{k=1}$ with a sequence of non-increasing and non-negative eigenvalues $\{\rho_k\}^{\infty}_{k=1}$ such that 
\begin{align}
K( \mathbf x_a, \mathbf x_b )=  \sum^{\infty}_{k=1} \rho_k \phi_k(\mathbf x_a) \phi_k(\mathbf x_b),
\label{equ:mercer_gp}
\end{align}
for any $\mathbf x_a, \mathbf x_b \in \mathcal X$.
 
The RKHS $\mathcal H$ contains all functions $f(\cdot)= \sum^\infty_{k=1} f_k \phi_k(\cdot) \in L_2(\mathcal X)$ with $f_k=\langle f, \, \phi_k \rangle_{L_2(\mathcal X)}$  and $\sum^\infty_{k=1}f^2_k/\rho_k<\infty$. For any $g(\cdot)=\sum^\infty_{k=1} g_k \phi_k(\cdot) \in \mathcal H$ and $f(\cdot)$, the inner product can be represented as $ \langle f, g\rangle_{\mathcal H}=\sum^\infty_{k=1} f_k g_k/\rho_k $. For more properties of the RKHS, we refer to Chapter 1 of \cite{wahba1990spline} and Chapter 11 of  \cite{ghosal2017fundamentals}.



\subsection{The equivalence between the maximum likelihood estimator and the kernel ridge regression estimator in calibration}

Assume one has a set of observations $\mathbf y^F:=\left(y^F(\mathbf x_1),...,y^F(\mathbf x_n)\right)^T$ and mathematical model outputs $\mathbf f^M_{\bm \theta}:=( f^M(\mathbf x_1, \bm \theta),...,f^M(\mathbf x_n, \bm \theta))^T$, where $\bm \theta =(\theta_1,...,\theta_q)^T \in \bm \Theta\subset\mathbb{R}^q$ is a $q$-dimensional vector of the calibration parameters.

Denote the regularization parameter $\lambda:= \sigma^2_0/ (n\sigma^2)$. For the calibration model (\ref{equ:gp_calibration}) with $\delta$ modeled as a GaSP in (\ref{equ:GP}), the marginal distribution of $\mathbf y^F$ follows a multivariate normal after marginalizing out $\delta$,
\begin{equation}
[\mathbf y^F \mid \bm \theta,   \, \sigma^2_0, \, \lambda] \sim \mbox{MN}(  \mathbf f^M_{\bm \theta},\,  \sigma^2_0((n \lambda)^{-1} \mathbf R+ \mathbf I_n) ).
\label{equ:y_F_GP}
\end{equation}
Let $ \mathcal L( \bm \theta)$ be the likelihood for $\bm \theta$  in (\ref{equ:y_F_GP}) given the other parameters in the model. 
For any given $\lambda$, the maximum likelihood estimator (MLE) of $\bm \theta$ is denoted as
\begin{equation}
\hat {\bm \theta}_{\lambda,n}:= \argmax_{\bm \theta \in \bm \Theta} \mathcal L(\bm \theta).
\label{equ:MLE_theta_gp}
\end{equation}
Conditioning on the observations, $\hat {\bm \theta}_{\lambda,n}$ and $\lambda$, the predictive mean of the discrepancy function at any $\mathbf x \in \mathcal X$ has the following expression  
\begin{align}
\hat \delta_{\lambda,n}(\mathbf x) := \E[ \delta(\mathbf x)  \mid \mathbf y^F, \hat {\bm \theta}_{\lambda, n},  \, \lambda]= \mathbf r^T(\mathbf x) (\mathbf R + n\lambda \mathbf I_n)^{-1} \left(\mathbf y^F- \mathbf f^M_{\hat{\bm \theta}_{\lambda, n}}\right)  
\label{equ:pred_gasp}
\end{align}
with $\mathbf r(\mathbf x)=(K(\mathbf x_1, \mathbf x),...,K(\mathbf x_n, \mathbf x) )^T$ and $\mathbf I_n$ being the $n$-dimensional identity matrix.

It is well-known that the predictive mean in (\ref{equ:pred_gasp}) can be written as the estimator for the kernel ridge regression (KRR).  In the following lemma, we show that $(\hat {\bm \theta}_{\lambda,n}, \hat \delta_{\lambda,n}(\cdot))$ is equivalent to the KRR estimator.


\begin{lemma}
The maximum likelihood estimator $\hat {\bm \theta}_{\lambda,n}$ defined in (\ref{equ:MLE_theta_gp}) and predictive mean estimator $\hat \delta_{\lambda,n}(\cdot) $ defined in (\ref{equ:pred_gasp}) can be expressed as the estimator of the kernel ridge regression as follows
\begin{align*}
(\hat {\bm \theta}_{\lambda, n}, \hat \delta_{\lambda, n}(\cdot ) )=\underset{ \bm \theta \in \bm \Theta, \delta(\cdot) \in \mathcal H}{\argmin} \ell_{\lambda, n}(\bm \theta, \delta),
\end{align*}
where
 \begin{align}
 \ell_{\lambda, n}(\bm \theta, \delta)&=
 \frac{1}{n}\sum^n_{i=1}(y^F(\mathbf x_i) -f^M(\mathbf x_i, \bm \theta) -\delta(\mathbf x_i))^2 +\lambda \|\delta\|^2_{\mathcal H}  
 .
 \label{equ:KRR_gp}
 \end{align}

 \label{lemma:gp_est}
 \end{lemma}
 
 \begin{proof}[Proof of Lemma~\ref{lemma:gp_est}]
 
By the representer lemma \citep{rasmussen2006gaussian,wahba1990spline}, for any $\bm \theta \in \bm \Theta$ and $\mathbf x \in \mathcal X$, one has 
\begin{equation}
\hat \delta_{\lambda, n, \bm \theta}(\mathbf x)=\sum^n_{i=1}w_i(\bm \theta) K(\mathbf x_i, \mathbf x).
\label{equ:hat_delta_KRR}
\end{equation}
Denote $\mathbf w_{\bm \theta}=(w_1(\bm \theta),...,w_n(\bm \theta))^T$. Since $\langle K(\mathbf x_i,\cdot), K(\mathbf x_j,\cdot)\rangle_{\mathcal H}= K(\mathbf x_i, \mathbf x_j)$, (\ref{equ:KRR_gp}) becomes to find $\bm \theta$ and $\mathbf w_{\bm \theta} $ that minimize 
\begin{equation}
\frac{1}{n} (\mathbf y^F-\mathbf f^M_{\bm \theta} -\mathbf R \mathbf w_{\bm \theta} )^T(\mathbf y^F-\mathbf f^M_{\bm \theta} -\mathbf R \mathbf w_{\bm \theta} )+ \lambda \mathbf w^T_{\bm \theta} \mathbf R \mathbf w_{\bm \theta}. \label{equ:quadratic_w}
\end{equation}
For any $\bm \theta$, solving the minimization for (\ref{equ:quadratic_w}) with regard to $\mathbf w_{\bm \theta}$ gives
\begin{equation}
\mathbf {\hat w}_{\bm \theta}= (\mathbf R+n\lambda \mathbf I_n)^{-1}(\mathbf y^F - \mathbf f^M_{\bm \theta}). 
\label{equ:w_form}
\end{equation}
Then plugging $ \mathbf {\hat w}_{\bm \theta}$ into (\ref{equ:quadratic_w}), based on the Woodbury matrix identity, one has
\begin{align}
\nonumber &\frac{1}{n} (\mathbf y^F-\mathbf f^M_{\bm \theta} -\mathbf R \mathbf {\hat w}_{\bm \theta})^T(\mathbf y^F-\mathbf f^M_{\bm \theta} -\mathbf R \mathbf {\hat w}_{\bm \theta} )+ \lambda\mathbf {\hat w}_{\bm \theta}^T \mathbf R \mathbf {\hat w}_{\bm \theta}\\
\nonumber &\quad=\frac{1}{n}(\mathbf y^F-\mathbf f^M_{\bm \theta} )^T[ (\mathbf I_n - \mathbf R(\mathbf R+n\lambda \mathbf I_n)^{-1} )^T(\mathbf I_n - \mathbf R(\mathbf R+n\lambda \mathbf I_n)^{-1})] \\
\nonumber &\quad\quad+ \lambda (\mathbf y^F-\mathbf f^M_{\bm \theta} )^T (\mathbf R+n\lambda \mathbf I_n)^{-1} \mathbf R (\mathbf R+n\lambda \mathbf I_n)^{-1} (\mathbf y^F-\mathbf f^M_{\bm \theta} ) \\
&\quad= \lambda (\mathbf y^F-\mathbf f^M_{\bm \theta} )^T  (\mathbf R+n\lambda \mathbf I_n)^{-1}  (\mathbf y^F-\mathbf f^M_{\bm \theta}),
\label{equ:KKR_theta}
\end{align}
which shows that the minimizer of $\bm \theta$ on right-hand side of (\ref{equ:KKR_theta}) is the same as the MLE of $\bm \theta$ in (\ref{equ:MLE_theta_gp}). Finally, plugging the estimator $\bm {\hat \theta}_{\lambda, n}$ into (\ref{equ:quadratic_w}), the result follows from the KRR estimator of $\delta(\cdot)$ in  (\ref{equ:hat_delta_KRR}) with the weights in (\ref{equ:w_form}). 
\end{proof}

  



Although modeling the discrepancy function by the GaSP typically improves the prediction accuracy of the reality, the penalty term of (\ref{equ:KRR_gp}) only contains $\|\delta\|_{\mathcal H}$ to control the complexity of the discrepancy. As the RKHS norm is not equivalent to the $L_2$ norm, the calibrated computer model could deviate a lot from the best performed mathematical model in terms of the $L_2$ loss \citep{tuo2015efficient}. In Section~\ref{sec:SGP}, we introduce the scaled Gaussian stochastic process that predicts the reality as accurately as the GaSP with the aid of the mathematical model, but has more prior mass on the small $L_2$ distance between the reality and mathematical model. As a consequence,  the KRR estimator of the new model penalizes both $\|\delta\|_{\mathcal H}$ and $\|\delta\|_{L_2(\mathcal X)}$ simultaneously.

\section{Proof for Section \ref{sec:SGP}}
\label{sec:proof_SGP}

\begin{proof}[Proof of Lemma~\ref{lemma:KL_sgp}]
By Karhunen-Lo{\`e}ve expansion, we have
\begin{equation*}
\delta(\xbf) = \sigma \sum_{i=1}^{\infty}\sqrt{\rho_i}Z_i\phi_i(\xbf)
\end{equation*}
with $Z_i \overset{i.i.d}{\sim}  \mbox{N}(0,1) $. Denote $ W_k = \sum_{i=k+1}^{\infty} \rho_iZ_i^2 $ for any $k \in \mathbb N^{+}$. From the definition of $Z=\int_{\mathbf{x}\in\mathcal{X}} \delta^2(\mathbf x)d \mathbf{x}$ and $\int_{\mathbf{x}\in\mathcal{X}} \phi^2_i(\mathbf x)d\mathbf x=1$ for any $i \in \mathbb N^{+}$, it is straightforward to see that
\begin{equation}
Z = \sigma^2(\rho_1Z_1^2 + \cdots + \rho_kZ_k^2 + W_k).
\label{equ:Z_GP_expansion}
\end{equation}
	 

In the following expressions, we are conditioning on all parameters and they are dropped for simplicity. From the construction of 
\[
\delta_z(\xbf)= \left\{\delta(\xbf) \mid \int_{\mathbf{x}\in\mathcal{X}} \delta^2(\xbf) d\xbf=Z\right\}\]
 and $Z\sim p_Z(\cdot)$, the joint density of $(Z_1,...,Z_k, W_k)$ in the S-GaSP can be expressed as 
\begin{align*}
&p_{\delta_z}(Z_1=z_1, ..., Z_k=z_k, W_k=w_k)\\
&=\int_0^\infty p_{\delta} \parenth{Z_1=z_1, ..., Z_k=z_k, W_k=w_k \mid Z=z} p_Z(Z=z)dz\\
&\propto \int_0^\infty \frac{p_{\delta}\parenth{ Z=z, Z_1=z_1, ..., Z_k=z_k, W_k=w_k }}{p_{\delta}(Z=z)}  g_Z(Z=z) p_{\delta}(Z=z) dz\\
&\propto p_{\delta}(Z_1=z_1, ..., Z_k=z_k) p_{\delta}(W_k=w_k) \times \\
& \quad \quad    \int_0^\infty p_{\delta}(Z=z\mid Z_1=z_1, ..., Z_k=z_k, W_k=w_k)  \exp\parenth{-\frac{\lambda_zz}{2\sigma^2}} dz\\
&\propto p_{\delta}(Z_1=z_1, ..., Z_k=z_k) p_{\delta}(W_k=w_k) \times \\
& \quad \quad \int_0^\infty  \mathbbm 1\left\{{z = \sigma^2(\rho_1z_1^2 + \cdots + \rho_kz_k^2 + w_k)}\right\} \exp\parenth{-\frac{\lambda_zz}{2\sigma^2}}dz\\
&\propto\exp\parenth{-\frac{1}{2}\sum_{i=1}^{k}z_i^2} p_{\delta}(W_k=w_k) \exp\sqbracket{-\frac{\lambda_z}{2}\parenth{\sum_{i=1}^{k}\rho_iz_i^2 + w_k}}\\
&=\left\{ \prod_{i=1}^{k}\exp\sqbracket{-\frac{1}{2}(1 + \lambda_z\rho_i)z_i^2} \right\} p_{\delta}(W_k=w_k) \exp(-\lambda_zw_k/2) \,,
\end{align*}
where $\mathbbm 1$ in the fourth step is a Dirac delta function.  

After integrating out $ W_k $, it is clear that $ Z_i $'s are independently distributed as $ \mbox{N}(0, 1/(1+\lambda_z\rho_i)) $ under the measure induced by the S-GaSP. Since $k$ is arbitrary, we have
\begin{equation*}
\delta_z(\xbf) = \sigma \sum_{i=1}^{\infty}\sqrt{\frac{\rho_i}{1 + \lambda_z\rho_i}}Z_i\phi_i(\xbf)\,
\end{equation*}
with $Z_i \overset{i.i.d.}{\sim} \mbox{N}(0, 1)$, from which the proof is complete.
\end{proof}

\begin{proof}[Proof of Lemma \ref{lemma:RKHS_comparison}]
First note that for any $\mathbf x_a, \mathbf x_b \in \mathcal X$, we have $K(\mathbf x_a,\mathbf x_b)=\sum^{\infty}_{i=1} \rho_{i} \phi_i(\mathbf x_a) \phi_i(\mathbf x_b) $ and $K_z(\mathbf x_a,\mathbf x_b)=\sum^{\infty}_{i=1}\rho_{z,i} \phi_i(\mathbf x_a) \phi_i(\mathbf x_b) $ with $\rho_{z,i}=\rho_i / (1+\lambda_z \rho_i)$. For $h(\cdot)=\sum^{\infty}_{i=1}h_i \phi_i(\cdot) \in \mathcal H$ and $g(\cdot)=\sum^{\infty}_{i=1}g_i \phi_i(\cdot) \in \mathcal H$, one has
\begin{align*}
\langle h, g \rangle_{\mathcal H_z}=\sum^{\infty}_{i=1} \frac{1}{\rho_{z,i}} h_i g_i
=\sum^{\infty}_{i=1} \frac{1}{\rho_i} h_i g_i+ \lambda_z \sum^{\infty}_{i=1} h_i g_i
=\langle h, g \rangle_{\mathcal H} + \lambda_z \langle  h, g  \rangle_{L_2}.
\end{align*}
\end{proof}

\begin{proof}[Proof of Lemma~\ref{lemma:sgp_est}]
We show below that for any $\bm \theta \in \bm \Theta$ and any $\mathbf x \in \mathcal X$, one has 
\begin{equation}
\hat \delta_{\lambda, \lambda_z, n}(\mathbf x)=\sum^n_{i=1} w_{z,i}(\bm \theta) K_z(\mathbf x_i ,\mathbf x).
\label{equ:delta_z_regulariation}
\end{equation}

For any $\delta(\cdot) \in \mathcal H$, decomposing it into the linear combination of the basis $ \{ \lambda_z K_z(\mathbf x_i, \cdot) \}^{n}_{i=1}$ and the orthogonal complement $v(\cdot)$ gives
\[\delta(\cdot) =\sum^n_{i=1} \tilde w_{z,i}(\bm \theta) \lambda_z  K_z(\mathbf x_i , \cdot) +v(\cdot),\] 
where $\langle v(\cdot), \lambda_z K_z(\cdot, \mathbf x_i) \rangle_{\mathcal H_z}=0$ for $i=1,...,n$.

To evaluate $\delta(\cdot)$ at $\mathbf x_j$ for any $j=1,...,n$, we have
\begin{align*}
\delta(\mathbf x_j)&=\left\langle\sum^n_{i=1}  \tilde w_{z,i}(\bm \theta) \lambda_z K_z(\mathbf x_i , \cdot) +v(\cdot), K_z( \mathbf x_j,\cdot)  \right\rangle_{\mathcal H_z} \\
&= \sum^n_{i=1} \tilde w_{z,i}(\bm \theta) \lambda_z K_z(\mathbf x_i, \mathbf x_j),
\end{align*}
which is independent from $v(\cdot)$.  Hence the first term on right-hand side of (\ref{equ:KRR_sgp}) is also independent from $v(\cdot)$. For the second term on right-hand side of (\ref{equ:KRR_sgp}), since $v(\cdot)$ is orthogonal to $\{K_z(\mathbf x_i, \cdot) \}^{n}_{i=1}$, plugging in the decomposition of $\delta(\cdot)$, we have
\begin{align*}
\lambda \| \delta\|^2_{\mathcal H_z}&= \lambda (\| \sum^n_{i=1} \tilde w_{z,i}(\bm \theta) \lambda_z K_z(\mathbf x_i , \cdot)\|^2_{\mathcal H_z}+ \|v \|^2_{\mathcal H_z}) \\
& \geq \lambda \| \sum^n_{i=1} \tilde w_{z,i}(\bm \theta) \lambda_z K_z(\mathbf x_i , \cdot)\|^2_{\mathcal H_z}.
\end{align*}
Thus choosing $v(\cdot)=0$ does not change the first term on right-hand side of (\ref{equ:KRR_sgp}), but also minimizes the second term on right-hand side of (\ref{equ:KRR_sgp}). Letting $w_{z,i}(\bm \theta)= \tilde w_{z,i}(\bm \theta) \lambda_z$, we have proved (\ref{equ:delta_z_regulariation}). The rest of the proof can be derived similarly as the proof for Lemma~\ref{lemma:gp_est}, so it is omitted here.
\end{proof}

\section{Proof for Section~\ref{subsec:nonparametric_SGP}}
\label{sec:proof_nonparametric_SGP}

We prove Theorem \ref{thm:nonparametric} in this Section. Two auxiliary lemmas used for the proof of Theorem \ref{thm:nonparametric} are given after the proof.

\begin{proof}[\bf{Proof} for Theorem~\ref{thm:nonparametric}]

\noindent Define a new inner product on $\calH$ as
\begin{align}
\langle f, g\rangle_\lambda = (1 + \sqrt{\lambda})\langle f, g\rangle_{L_2(\mathcal X)} + \lambda\langle f, g\rangle_\calH
\end{align}

\noindent Let $f = \sum_{k=1}^\infty f_k\phi_k$ and $g = \sum_{k=1}^\infty g_k\phi_k$ be elements in $\calH$. Then
\[
\langle f,g\rangle_\lambda = (1+ \sqrt{\lambda})\sum_{k=1}^\infty f_kg_k + \lambda \sum_{k=1}^\infty\frac{f_kg_k}{\rho_k} = \sum_{k=1}^\infty \left(1 + \sqrt{\lambda} + \frac{\lambda}{\rho_k}\right)f_kg_k.
\]
By letting $\mu_k$ by $\mu_k^{-1} = 1 + \sqrt{\lambda} + \lambda/\rho_k$, we can define a new reproducing kernel
\begin{equation}
K_\lambda(\bx,\bx') = \sum_{k=1}^\infty \mu_k\phi_k(\bx)\phi_k(\bx')
\label{equ:K_lambda}
\end{equation}
Since $c_{\rho}^{-1} k^{-2m/p}\leq\rho_k\leq C_\rho^{-1} k^{-2m/p}$ and $|\phi_i(\cdot)| < C_{\phi}$  for some positive constants $c_\rho$, $C_{\rho}$ and $C_{\phi}$, bounding the sums by integrals, we have
\begin{align}
\sup_{\bx,\bx'}K_\lambda(\bx, \bx') &\leq C_\phi^2\sum_{k = 1}^\infty\frac{1}{1 + \lambda c_\rho k^{2m/p}} \leq C_\phi^2\sum_{k = 1}^\infty \int_{k - 1}^k\frac{1}{1 + \lambda c_\rho x^{2m/p}}dx\nonumber\\
& = C_\phi^2 c_\rho^{-p/2m}\lambda^{-p/2m}\int_0^\infty\frac{(\lambda c_\rho)^{p/2m}}{1 + \{(\lambda c_\rho)^{p/2m}x\}^{2m/p}}dx\nonumber\\
& = 
C_\phi^{2} c_{\rho}^{-p/2m}\lambda^{-p/2m}\int_0^\infty \frac{1}{1 + x^{2m/p}}dx\nonumber.
\end{align}
Thus
\begin{equation}
\sup_{\bx,\bx'}K_\lambda(\bx,\bx')\leq C^2_K\lambda^{-p/(2m)},
\label{equ:sup_K_bound}
\end{equation}
for some constant $C_K$ depending on $K$. Define the following linear operators $F_\lambda:\calH\to\calH$ and $P_\lambda:\calH\to\calH$ via
\[
(F_\lambda g)(\bx) = \int_\calX g(\bx')K_\lambda(\bx,\bx')\mathrm{d}\bx',\quad\text{and}\quad
(P_\lambda g)(\bx) = g(\bx) - (F_\lambda g)(\bx),
\]
Clearly, we have
\begin{align}
&\langle f, F_\lambda g\rangle_\lambda = \sum_{k=1}^\infty \langle f, \phi_k\rangle_{L_2(\mathcal X)}\langle g, \phi_k\rangle_{L_2(\mathcal X)} = \langle f, g\rangle_{L_2(\mathcal X)},\nonumber\\
&\langle f, P_\lambda g\rangle_\lambda = \langle f, g\rangle_\lambda - \langle f, F_\lambda g\rangle_\lambda = \sqrt{\lambda}\langle f, g\rangle_{L_2(\mathcal X)} + \lambda \langle f, g\rangle_{\calH}.
\end{align}

\noindent Denote the loss function 
\[
\ell_{n\lambda}(f) = \frac{1}{n}\sum_{i=1}^n(y_i-f(\bx_i))^2 + \sqrt{\lambda}\|f\|_{L_2(\mathcal X)}^2 + \lambda\|f\|_{\calH}^2,
\]
and the estimator $\hat f_{n\lambda}:= \argmin_{f \in \mathcal H} \ell_{n\lambda}(f)$. Let $D\ell_{n\lambda}(f):\calH\to\calH$ be the Frech\'et derivative of $\ell_{n\lambda}$ evaluated at $f$. Clearly, for any $g\in\calH$,
\begin{align}
D\ell_{n\lambda}(f)g &= \frac{2}{n}\sum_{i=1}^n(f(\bx_i) - y_i)\langle K_\lambda (\bx_i, \cdot), g(\cdot)\rangle_\lambda + 2\langle P_\lambda f, g\rangle_\lambda\nonumber\\
& = \left\langle
\frac{2}{n}\sum_{i=1}^n(f(\bx_i) - y_i)K_\lambda(\bx_i, \cdot) + 2(P_\lambda f)(\cdot), g(\cdot)
\right\rangle_\lambda.
\end{align}
It follows that $D\ell_{n\lambda}(\hat f_{n\lambda})g = 0$ for all $g\in\calH$, and hence,
 $S_{n\lambda}(\hat f_{n\lambda})(\cdot) = 0$, where
\[
S_{n\lambda}(f)(\cdot) = \frac{1}{n}\sum_{i=1}^n(y_i - f(\bx_i))K_\lambda(\bx_i, \cdot) - (P_\lambda f)(\cdot).
\]
Define $S_\lambda(f)(\cdot) = \E_{y, \mathbf x}(S_{n\lambda}(f) (\cdot))$. Then
\begin{align*}
S_\lambda (f)(\cdot) &= \int_{\mathbf{x}\in\mathcal{X}} (f_0(\bx) - f(\bx))K_\lambda(\bx, \cdot)\mathrm{d}\bx - (P_\lambda f)(\cdot)  \\
&= (F_\lambda(f_0 - f))(\cdot) - (P_\lambda f)(\cdot) = (F_\lambda f_0)(\cdot) - f(\cdot),
\end{align*}
and therefore, $S_\lambda (F_\lambda f_0)(\cdot) = 0$. 
\noindent
Let $\Delta f = \hat f_{n\lambda} - F_\lambda f_0$. By the definitions of $S_{n\lambda}$ and $S_\lambda$, we have
\begin{align*}
&\left\{S_{n\lambda}(\hat f_{n\lambda}) - S_\lambda(\hat f_{n\lambda})\right\}(\cdot) - 
\left\{S_{n\lambda}(F_\lambda f_0) - S_\lambda(F_\lambda f_0)\right\}(\cdot)\nonumber\\
 &\quad= \left\{S_{n\lambda}(\hat f_{n\lambda}) - S_{n\lambda}(F_\lambda f_0)\right\}(\cdot) - 
\left\{S_\lambda(\hat f_{n\lambda}) - S_\lambda(F_\lambda f_0)\right\}(\cdot)\nonumber\\
 &\quad= \frac{1}{n}\sum_{i=1}^n\left\{F_\lambda f_0(\bx_i) - \hat{f}_{n\lambda}(\bx_i)\right\}K_\lambda(\bx_i,\cdot) + P_\lambda\left\{(F_\lambda f_0)(\cdot) - \hat{f}_{n\lambda}(\cdot)\right\} \\
 &\quad\quad+ \hat{f}_{n\lambda}(\cdot) - (F_\lambda f_0)(\cdot)\nonumber\\
 &\quad= -\frac{1}{n}\sum_{i=1}^n\Delta f(\bx_i)K_\lambda(\bx_i, \cdot) - (P_\lambda\Delta f)(\cdot) + (\Delta f)(\cdot)\nonumber\\
 &\quad= -\frac{1}{n}\sum_{i=1}^n\Delta f(\bx_i)K_\lambda(\bx_i, \cdot) + \E_\bx\left\{\Delta f(\bx) K_\lambda(\bx, \cdot)\right\}\nonumber.
\end{align*}
On the other hand, 
$ S_{n\lambda}(\hat{f}_{n\lambda})(\cdot) = S_{\lambda}(F_\lambda f_0)(\cdot) = 0 $ and 
$ S_\lambda(\hat{f}_{n\lambda})(\cdot) = (F_\lambda f_0)(\cdot) - \hat{f}_{n\lambda}(\cdot) = -\Delta f(\cdot)$.
Therefore,
\[
\left\{S_{n\lambda}(\hat f_{n\lambda}) - S_\lambda(\hat f_{n\lambda})\right\}(\cdot) - 
\left\{S_{n\lambda}(F_\lambda f_0) - S_\lambda(F_\lambda f_0)\right\}(\cdot) = \Delta f(\cdot) - S_{n\lambda}(F_\lambda f_0)(\cdot).
\]
Define the event
\[
A_n(t) = \left\{
\left\|\frac{1}{n}\sum_{i = 1}^ng(\bx_i)K_\lambda(\bx_i,\cdot) - \E_\bx\{g(\bx)K_\lambda(\bx,\cdot)\}\right\|_\lambda<t\|g\|_{\lambda}
\text{ for all }g \in \calH
\right\}.
\]
Applying Lemma \ref{lemma:maximum_inequality} on $ g(\cdot) / \|g\|_\lambda $, 
\[
\prob_{\mathbf x}\{A_n(t)\}\geq 1- 2\exp\left\{-\lambda^{p(6m - p)/(4m^2)}\left(\frac{nt^2}{\kappa_K}\right)\right\}
\]
for some constant $\kappa_K > 0$. The deviation threshold $t$ will be specified later, and from now we consider data points $(\bx_i, y_i)_{i = 1}^n$ over the event $A_n(t)$. 

\noindent Over the event $A_n(t)$, we have
\begin{align}
&\left\{S_{n\lambda}(\hat f_{n\lambda}) - S_\lambda(\hat f_{n\lambda})\right\}(\cdot) - 
\left\{S_{n\lambda}(F_\lambda f_0) - S_\lambda(F_\lambda f_0)\right\}(\cdot)\nonumber\\
&\quad = -\frac{1}{n}\sum_{i=1}^n\Delta f(\bx_i)K_\lambda(\bx_i, \cdot) + \E_\bx\left\{\Delta f(\bx) K_\lambda(\bx, \cdot)\right\}\nonumber\\
&\quad = \Delta f(\cdot) - S_{n\lambda}(F_\lambda f_0)(\cdot)\nonumber,
\end{align}
implying that
\begin{align*}
\|\Delta f - S_{n\lambda}(F_\lambda f_0)\|_{\lambda} &= \left\|
\frac{1}{n}\sum_{i=1}^n\Delta f(\bx_i)K_\lambda(\bx_i, \cdot) - \E_\bx\left\{\Delta f(\bx) K_\lambda(\bx, \cdot)\right\}
\right\|_{\lambda}\\
&\leq t\|\Delta f\|_\lambda\nonumber.
\end{align*}

\noindent
Now we proceed to bound $\|S_{n\lambda}(F_\lambda f_0)\|_{\lambda}$. Write
\begin{align*}
&\|S_{n\lambda}(F_\lambda f_0)\|_{\lambda} \\
&\quad= \left\|
\frac{1}{n}\sum_{i = 1}^n\{y_i - F_\lambda f_0(\bx_i)\}K_\lambda(\bx_i,\cdot) - (P_\lambda F_\lambda f_0)(\cdot)
\right\|_{\lambda}\nonumber\\
&\quad\leq \left\|
\frac{1}{n}\sum_{i = 1}^n\{f_0(\bx_i) - F_\lambda f_0(\bx_i)\}K_\lambda(\bx_i,\cdot) - \{F_\lambda (f_0 -  F_\lambda f_0)\}(\cdot)
\right\|_{\lambda}\nonumber\\
&\quad\quad + \left\|\frac{1}{n}\sum_{i = 1}^n{\epsilon}_i K_\lambda(\bx_i,\cdot)\right\|_{\lambda}\nonumber\\
&\quad= \left\|
\frac{1}{n}\sum_{i = 1}^n\{f_0(\bx_i) - F_\lambda f_0(\bx_i)\}K_\lambda(\bx_i,\cdot) - \E_\bx [\{f_0(\bx) - F_\lambda f_0(\bx)\}K_\lambda(\bx,\cdot)]
\right\|_\lambda\nonumber\\
&\quad\quad + \left\|\frac{1}{n}\sum_{i = 1}^n{\epsilon}_iK_\lambda(\bx_i,\cdot)\right\|_{\lambda}\nonumber\\
&\quad\leq t\|f_0 - F_\lambda f_0\|_{\lambda} + \left\|\frac{1}{n}\sum_{i = 1}^n{\epsilon}_iK_\lambda(\bx_i,\cdot)\right\|_{\lambda},\nonumber
\end{align*}
where the last inequality is due to the construction of the event $A_n(t)$. To bound the second term of the preceding display, we let ${\bm \sigma}_\lambda = [K_\lambda(\bx_i, \bx_j)]_{n\times n}$ and ${ \bm {\epsilon}} = [{\epsilon}_1,\ldots,{\epsilon}_n]^T$. By the Hanson-Wright inequality \citep{rudelson2013hanson}, for all $x>0$, we have
\[
\prob_\bx\left[\bm {\epsilon}^T{\bm \sigma}_\lambda\bm {\epsilon}\geq \sigma_0^2\left\{\mathrm{tr}({\bm \sigma}_\lambda) + 2\sqrt{\mathrm{tr}({\bm \sigma}^2_\lambda)}x + 2\|{\bm \sigma}_\lambda\|_{\mathrm{F}}x^2\right\}\right]\leq\mathrm{e}^{-x^2}.
\]
Since by the Cauchy-Schwarz inequality,
\begin{align}
\mathrm{tr}({\bm \Sigma}_\lambda)& =  \sum_{i = 1}^n \|K_\lambda(\bx_i, \cdot)\|_{\lambda}^2= \sum_{i = 1}^n K_\lambda(\bx_i, \bx_i) \leq C^2_Kn\lambda^{-p/(2m)}\nonumber,\\
\mathrm{tr}({\bm \Sigma}_\lambda^2)& \leq \sum_{i = 1}^n\sum_{j = 1}^n \|K_\lambda(\bx_i, \cdot)\|_{L_2(\mathcal X)}\|K_\lambda(\bx_j, \cdot)\|_{L_2(\mathcal X)}\leq C_K^4n^2\lambda^{-p/m}\nonumber,\\
\|{\bm \sigma}_\lambda\|_{\mathrm{F}}& = \sqrt{\mathrm{tr}({\bm \Sigma}_\lambda^2)}\leq C^2_K n\lambda^{-p/(2m)},\nonumber
\end{align}
it follows that
\begin{align}
\mathrm{tr}({\bm \Sigma}_\lambda) + 2\sqrt{\mathrm{tr}({\bm \Sigma}^2_\lambda)}x + 2\|{\bm \Sigma}_\lambda\|_{\mathrm{F}}x^2\leq  C^2_K n\lambda^{-p/(2m)}(1 + 2x + 2x^2).
\end{align}
Set the event $B_n$ to be
\[
B_n = \left\{
\left\|\frac{1}{n}\sum_{i = 1}^ne_i K_\lambda(\bx_i, \cdot)\right\|_{\lambda}< \sigma_0{C_K} n^{-1/2}\lambda^{-p/(4m)}\alpha^{1/2}
\right\},
\]
where $\alpha = 2+3x^2$. Since $1+2x + 2x^2\leq 2+3x^2 = \alpha$, by taking $x=\sqrt{ (\alpha-2)/3}$, we have $\prob(B_n) \geq 1-\exp(-(\alpha-2)/3)  $ for any $\alpha>{2}$.
Putting all pieces obtained above together, we have
\begin{align}
\|\Delta f\|_{\lambda}&\leq \|\Delta f - S_{n\lambda}(F_\lambda f_0)\|_{\lambda} + \|S_{n\lambda}(F_\lambda f_0)\|_{\lambda}\nonumber\\
&\leq t\|\Delta f\|_{\lambda} + t\|f_0 - F_\lambda f_0\|_\lambda + \sigma_0{C_K}n^{-1/2}\lambda^{-p/(4m)}\alpha^{1/2}\nonumber\\
& = t\|\Delta f\|_{\lambda} + t\|P_\lambda f_0\|_\lambda + \sigma_0{C_K}n^{-1/2}\lambda^{-p/(4m)}\alpha^{1/2},
\end{align}
over the event $A_n(t)\cap B_n$. Now take $\lambda = n^{-2m/(2m + p)}$. Choose any $C_{\beta}\in (0, 1)$ and let $t=\sqrt{\kappa_K/(n^{(1-C_{\beta})\beta}\log(2))}$. Then, for sufficiently large $n$,
\[
\prob_{\mathbf x}\{A_n(t)\} \geq 1-2\exp\left\{-\frac{n^\beta t^2}{\kappa_K }\right\}\geq 1- \exp\left\{-{n^{C_{\beta}\beta}}\right\},
\]
where $\beta = (2m-p)^2/(2m(2m+p))$, and therefore,
\[
\|\Delta f\|_\lambda\leq \|P_\lambda f_0\|_\lambda + 2\sigma_0{C_K}n^{-m/(2m + p)}\alpha^{1/2},
\]
with probability at least  
\begin{align*}
&\prob\{A_n(t) \cap B_n)\}
= 1-\prob\{A^c_n(t) \cup B^c_n\} \\
&\quad\geq 1-\prob\{A^c_n(t)\}-\prob(B^c_n)
=1-\exp\{-(\alpha-2)/3\}-\exp\{-n^{C_{\beta}\beta}\}
\end{align*}
for sufficiently large $n$. 
Observe that
\begin{align}
\|P_\lambda f_0\|_\lambda^2 &= \left\|\sum_{k = 1}^\infty(1 - \mu_k)\langle f_0, \phi_k\rangle_{L_2(\mathcal X)}\phi_k(\cdot)
\right\|_{\lambda}^2 = \sum_{k = 1}^\infty \frac{(1 - \mu_k)^2}{\mu_k}\langle f_0,\phi_k\rangle_{L_2(\mathcal X)}^2\nonumber\\
& = \sum_{k = 1}^\infty \frac{(\sqrt{\lambda} + \lambda/\rho_k)^2}{1 + \sqrt{\lambda} + \lambda/\rho_k}\langle f_0,\phi_k\rangle_{L_2(\mathcal X)}^2
 \leq\sum_{k = 1}^\infty\frac{2\lambda +2 (\lambda/\rho_k)^2}{1 + \lambda/\rho_k}\langle f_0,\phi_k\rangle_{L_2(\mathcal X)}^2\nonumber\\
 &\leq 2\lambda\sum_{k = 1}^\infty \langle f_0,\phi_k\rangle_{L_2(\mathcal X)}^2 +2 \lambda\sum_{k = 1}^\infty \frac{1}{\rho_k}\langle f_0,\phi_k\rangle_{L_2(\mathcal X)}^2\nonumber\\
 &= 2\lambda\|f_0\|_{L_2(\mathcal X)}^2 + 2\lambda\|f_0\|_\calH^2\nonumber\\
 & \leq 2n^{-2m/(2m + p)}\left(\|f_0\|_{L_2(\mathcal X)} + \|f_0\|_\calH\right)^2\nonumber.
\end{align}
Hence, we proceed to compute
\begin{align}
\|\hat{f}_{n\lambda} - f_0\|_{L_2(\mathcal X)}&\leq 
\|\hat{f}_{n\lambda} - f_0\|_\lambda\nonumber\\
&\leq \|\hat{f}_{n\lambda} - F_\lambda f_0\|_\lambda + \|F_\lambda f_0 - f_0\|_\lambda\nonumber\\
&=\|\Delta f\|_\lambda + \|P_\lambda f_0\|_\lambda\nonumber\\
&\leq \left(2\sqrt{2}(\|f_0\|_{L_2(\mathcal X)} + \|f_0\|_\calH) + 2\sigma_0{C_K}\alpha^{1/2}\right)n^{-m/(2m + p)} \nonumber
\end{align}
with probability at least $1-\exp\{-(\alpha-2)/3\}-\exp\left(-n^{C_{\beta}\beta}\right)$ for sufficiently large $n$. The bound for $\|\hat f_{\lambda, \lambda_z, n}- f_0 \|_{\calH}$ follows immediately by the definition of $|| \cdot||_{\lambda}$, completing the proof.
\end{proof}


The following the Lemma~\ref{lemma:bounded_difference_inequality} is  Theorem 3.6 in \cite{pinelis1994optimum}, which is needed for the proof of Lemma~\ref{lemma:maximum_inequality}. 

\begin{lemma}
\label{lemma:bounded_difference_inequality}
Let $(X_j)_{j=0}^\infty$ be a sequence of random elements in a Hilbert space $\calH$ with norm $\|\cdot\|_\calH$. Suppose that $(X_j)_{j=0}^\infty$ forms a martingale in the sense that $\E(X_j\mid X_0,\ldots,X_{j-1}) = X_j$ a.s., and that the difference sequence $(D_j)_{j = 1}^\infty = (X_j - X_{j - 1})_{j = 1}^\infty$ satisfies $\|D_j\|_{\calH}^2\leq b_j^2$ a.s. and $\sum_{j = 1}^\infty b_j^2 \leq b_*^2$. Then for any $t\geq 0$,
\[
\prob\left(\sup_{j\geq 1}\|X_j\|_{\calH} \geq t\right)\leq 2\exp\left(-\frac{t^2}{2b_*^2}\right).
\]
\end{lemma}


The following maximum inequality for functional empirical processes in the Sobolev space $\calW_2^m(\calX, 1)$, which generalizes Lemma 5.1 in \cite{yang2017frequentist} to multivariate functions, is of fundamental importance to the proof of Theorem \ref{thm:nonparametric}. 

\begin{lemma}
\label{lemma:maximum_inequality}
Denote $ \calW_2^m(\calX, 1) = \left\{f\in\calW_2^m(\mathcal X):\|f\|_\lambda\leq 1\right\}$. Suppose $\bx_1,\ldots,\bx_n$ are independently and uniformly drawn from $\calX$. Then there exists some constant $\kappa_K$ depending on the kernel $K$, such that for any $t>0$,
\begin{align*}
&\prob_\bx\left(
\sup_{g\in \calW_2^m(\calX, 1)}\left\|\frac{1}{n}\sum_{i=1}^n\left[g(\bx_i)K_\lambda(\bx_i,\cdot) - \E_\bx \left\{g(\bx)K_\lambda(\bx,\cdot)\right\}\right]\right\|_\lambda\geq t
\right)\nonumber\\
&\quad\leq 2\exp\left\{-\frac{\lambda^{d(6m - d)/(4m^2)}nt^2}{\kappa_K}\right\}.
\end{align*}
\end{lemma}

\begin{proof}[\bf{Proof of Lemma \ref{lemma:maximum_inequality}}]
We follow the argument used in the proof of Lemma 6.1 in \cite{yang2017non}. 
Denote
\[
\{Z_{n\lambda}(g)\}(\cdot) = \frac{1}{n}\sum_{i=1}^n[g(\bx_i)K_\lambda(\bx_i, \cdot) - \E_\bx\{g(\bx)K_\lambda(\bx, \cdot)\}]. 
\]
Fix $g$, $h\in \calH$, $n$, and $\lambda$, consider the following sequence of martingale $(X_j)_{j = 0}^\infty$ in $\calH$:
\[
X_j = \left\{
\begin{array}{ll}
0,&\quad\text{if }j = 0,\\
j\{Z_{j\lambda}(g) - Z_{j\lambda}(h)\},&\quad\text{if }j = 1,\ldots,n\\
X_n,&\quad\text{if }j \geq n + 1.
\end{array}
\right.
\]
Clearly, for $j = 1,\ldots, n$,
\[
(X_j - X_{j - 1})(\cdot) = \{g(\bx_j) - h(\bx_j)\}K_\lambda(\bx_j, \cdot) - \E_\bx[\{g(\bx_j) - h(\bx_j)\}K_\lambda(\bx_j, \cdot)]
\]
and $X_j - X_{j - 1} = 0$ for $j\geq n + 1$. 
Observe that 
\[
\|K_\lambda(\bx_j, \cdot)\|_\lambda = \sqrt{\langle K_\lambda(\bx_j, \cdot),K_\lambda(\bx_j, \cdot)\rangle_\lambda} = \sqrt{K_\lambda(\bx_j, \bx_j)}\leq {C_K}\lambda^{-d/(4m)}
\]
with probability one.
Therefore, with probability one, we have
\[
\|X_j - X_{j - 1}\|_{\lambda}^2 \leq 4 C^2_K\lambda^{-d/(2m)}\|g - h\|_{L_\infty}^2
\]
for $j = 1,\ldots,n$, and hence, we invoke the bounded difference inequality for martingales in Banach space (Lemma \ref{lemma:bounded_difference_inequality}) to derive
\begin{align}
\prob\left(\|Z_{n\lambda}(g) - Z_{n\lambda}(h)\|_\lambda\geq t\right)
& = \prob\left(\|n\{Z_{n\lambda}(g) - Z_{n\lambda}(h)\}\|_\lambda\geq nt\right)\nonumber\\
&\leq \prob\left(\sup_{j\geq 1}\|j\{Z_{j\lambda}(g) - Z_{j\lambda}(h)\}\|_\lambda\geq nt\right)\nonumber\\
&\leq 2\exp\left\{-\frac{nt^2}{8C^2_K\lambda^{-d/(2m)}\|g - h\|_{L_\infty}^2} \right\}\nonumber.
\end{align}
Applying Lemma 8.1 in \cite{kosorok2008introduction}, we obtain the following bound
\begin{align}
\left\|
\|Z_{n\lambda}(g) - Z_{n\lambda}(h)\|_\lambda
\right\|_{\psi_2}
\label{eqn:Z_nlambda_g_Orlicz_norm_bound}
&\leq \frac{\sqrt{24}C^2_K}{\sqrt{n}}\lambda^{-d/(4m)}\|g-h\|_{L_\infty},
\end{align}
where $\|\cdot\|_{\psi_2}$ is the Orlicz norm associated with $\psi_2(s) = \exp(s^2) - 1$. 

Now let $\tau = \{\log(3/2)\}^{1/2}$ and set $\phi(x) = \psi_2(\tau x)$. Clearly, $\phi(1) = 1/2$, and $\phi(x)\phi(y)\leq \phi(xy)$ for any $x,y\geq1$. Applying Lemma 8.2 in \cite{kosorok2008introduction}, the Orlicz norm of the maximum of finitely many random variables can be bounded by the maximum of these Orlicz norms as follows:
\begin{align}
\left\|\max_{1\leq i\leq k}(\tau\xi_i)\right\|_{\psi_2}=
\left\|\max_{1\leq i\leq k}\xi_i\right\|_\phi \leq 2\phi^{-1}(k)\max_{1\leq i\leq k}\|\xi_i\|_{\phi} = \frac{2}{\tau}\psi_2^{-1}(k)\max_{1\leq i\leq k}\|\tau\xi_i\|_{\psi_2}\nonumber,
\end{align}
namely,
\begin{align}
\label{eqn:finite_maximal_inequality}
\left\|\max_{1\leq i\leq k}\xi_i\right\|_{\psi_2} \leq  \frac{2}{\tau}\psi_2^{-1}(k)\max_{1\leq i\leq k}\|\xi_i\|_{\psi_2},
\end{align}
where $\{\xi_i\}_{i=1}^k$ are finitely many random variables. 

Next we apply the ``chaining'' argument. Let $\varepsilon>0$ be some constant to be determined later. 
 Construct a sequence of function classes $({\mathcal{G}}_j)_{j = 0}^\infty$ in $\calH_\lambda(1)$ satisfying the following conditions:
\begin{itemize}
	\item[(i)] For any ${\mathcal{G}}_j$ and any $h_j, g_j\in {\mathcal{G}}_j$, $\|h_j-g_j\|_{L_\infty}\geq \varepsilon /2^j$, and ${\mathcal{G}}_j$ is maximal in the sense that for any $g_j\notin {\mathcal{G}}_j$, there exists some $h_j\in {\mathcal{G}}_j$ such that $\|h_j - g_j\|<\varepsilon/2^j$.
	\item[(ii)] For any ${\mathcal{G}}_{j + 1}$, and any $g_{j + 1}\in{\mathcal{G}}_{j + 1}$, select a unique element $g_j\in{\mathcal{G}}_j$ such that $\|g_{j + 1} - g_j\|_{L_\infty}\leq \varepsilon/2^j$. Thus, there exists a finite sequence $(g_0, g_1, \ldots, g_{j + 1})$ such that $\|g_i - g_{i + 1}\|_{L_\infty}\leq \varepsilon/2^i$ for $i = 0,\ldots,j$, and $g_i\in{\mathcal{G}}_i$. 
\end{itemize}
Therefore, for any $g_{j + 1}, h_{j + 1}\in {\mathcal{G}}_{j + 1}$ with $\|g_{j + 1} - h_{j + 1}\|_{L_\infty}\leq\varepsilon$, there exists two sequences $(g_i)_{i = 0}^{j + 1}$, $(h_i)_{i = 0}^{j + 1}$, such that $g_i, h_i\in{\mathcal{G}}_i$, $\max\{\|g_i - g_{i + 1}\|_{L_\infty}, \|h_i - h_{i + 1}\|_{L_\infty}\}\leq \varepsilon/2^i$, and that
\begin{align*}
\|g_0 - h_0\|_{L_\infty}&\leq \sum_{i = 0}^j\left(\|g_i - g_{i + 1}\|_{L_\infty} + \|h_i - h_{i + 1}\|_{L_\infty}\right) + \|h_{j + 1} - g_{j + 1}\|_{L_\infty}\\
&\leq 2\sum_{i = 0}^j\frac{\varepsilon}{2^i} + \varepsilon \leq 5\varepsilon,
\end{align*}
and hence, by (\ref{eqn:Z_nlambda_g_Orlicz_norm_bound}) one has
\begin{align}
\label{eqn:g_0h_0_diff_bound}
\|\|Z_{n\lambda}(g_0) - Z_{n\lambda}(h_0)\|_\lambda\|_{\psi_2}\leq \frac{5\sqrt{24}C^2_K}{\sqrt{n}}\lambda^{-p/(4m)}\varepsilon.
\end{align}
We also notice that ${\mathcal{G}}_j\subset\calH_\lambda(1)\subset\{f\in\calH:\|f\|_\calH\leq\lambda^{-1/2}\}$, and therefore, the cardinality of ${\mathcal{G}}_j$ can be bounded by the metric entropy of $\{f\in\calH:\|f\|_\calH\leq\lambda^{-1/2}\}$, which is known in the literature \citep{edmunds2008function}:
\begin{align*}
\log|{\mathcal{G}}_j|&\leq \log\mathcal{N}_{[{\boldsymbol{\cdot}}]}\left(\varepsilon/2^j, \{f\in\calH:\|f\|_\calH\leq\lambda^{-1/2}\}, \|\cdot\|_{L_\infty}\right)\nonumber\\
&\leq c_0\lambda^{-p/(2m)}\left(\frac{\varepsilon}{2^j}\right)^{-p/m},
\end{align*}
where $c_0$ is some absolute constant. 

Now suppose $g$, $h$ are arbitrary functions in $\calH_\lambda(1)$ such that $\|g - h\|_{L_\infty}\leq\varepsilon/2$. For any $j\geq 2$, there exists $g_j, h_j\in{\mathcal{G}}_j$ such that
\[
\max\{\|g_j - g\|_{L_\infty}, \|h_j - h\|_{L_\infty}\}\leq \varepsilon/2^j,
\]
and hence, $\|g_j - h_j\|_{L_\infty}\leq \varepsilon$. Therefore, for any $j\geq 2$,
\begin{align}
&\left\|\sup_{g,h\in\mathcal W^m_2(\mathcal X, 1), \|g - h\|_{L_\infty}\leq\varepsilon}\|Z_{n\lambda}(g) - Z_{n\lambda}(h)\|_\lambda\right\|_{\psi_2}\nonumber\\
&\quad\leq \bigg\| \bigg. \sup_{g,h\in\mathcal W^m_2(\mathcal X, 1), \|g - h\|_{L_\infty}\leq\varepsilon}\bigg( \bigg. \|Z_{n\lambda}(g) - Z_{n\lambda}(g_j)\|_\lambda + \|Z_{n\lambda}(g_j) - Z_{n\lambda}(h_j)\|_\lambda  \nonumber\\
&\quad\quad+ \|Z_{n\lambda}(h_j) - Z_{n\lambda}(h)\|_\lambda \bigg. \bigg)\bigg.  \bigg\|_{\psi_2}\nonumber\\
&\quad\leq \frac{2\sqrt{24}C^2_K}{\sqrt{n}}\lambda^{-d/(4m)}
\max\left\{\|g - g_j\|_{L_\infty}, \|h - h_j\|_{L_\infty}\right\} \nonumber\\
&\quad\quad+ \left\|\sup_{g,h\in\mathcal W^m_2(\mathcal X, 1),\|g - h\|_{L_\infty}\leq \varepsilon}\|Z_{n\lambda}(g_j) - Z_{n\lambda}(h_j)\|_\lambda\right\|_{\psi_2}\nonumber\\
&\quad\leq \frac{2\sqrt{24}C^2_K}{\sqrt{n}}\frac{\lambda^{-d/(4m)}\varepsilon}{2^j} + \left\|\max_{g_j, h_j\in{\mathcal{G}}_j,\|g_j - h_j\|_{L_\infty}\leq\varepsilon}\|Z_{n\lambda}(g_j) - Z_{n\lambda}(h_j)\|_\lambda\right\|_{\psi_2}\nonumber.
\end{align}
We focus on the second term of the preceding display. Fix $j\geq 2$, for any $g_j, h_j\in{\mathcal{G}}_j$, consider the finite sequences $(g_0, g_1,\ldots, g_j)$ and $(h_0, h_1,\ldots, h_j)$ such that $g_i, h_i\in{\mathcal{G}}_i$ and $\|g_i - g_{i + 1}\|_{L_\infty}\leq \varepsilon/2^i$, $i = 1,\ldots,j - 1$. Invoking the inequality \eqref{eqn:finite_maximal_inequality}, we have
\begin{align}
&\left\|\max_{g_j,h_j\in{\mathcal{G}}_j, \|g_j - h_j\|_{L_\infty}\leq \varepsilon}\|Z_{n\lambda}(g_j) - Z_{n\lambda}(h_j)\|_\lambda\right\|_{\psi_2}\nonumber\\
&\quad \leq \left\|\max_{g_j,h_j\in{\mathcal{G}}_j, \|g_j - h_j\|_{L_\infty}\leq \varepsilon}\|\{Z_{n\lambda}(g_j) - Z_{n\lambda}(h_j)\} - \{Z_{n\lambda}(g_0) - Z_{n\lambda}(h_0)\}\|_\lambda\right\|_{\psi_2}\nonumber\\
&\quad\quad + \left\|\max_{g_0,h_0\in{\mathcal{G}}_0, \|g_j - h_j\|_{L_\infty}\leq \varepsilon}\|Z_{n\lambda}(g_0) - Z_{n\lambda}(h_0)\|_\lambda\right\|_{\psi_2}\nonumber\\
&\quad\leq \left\|\max_{g_j,h_j\in{\mathcal{G}}_j, \|g_j - h_j\|_{L_\infty}\leq \varepsilon}\|\{Z_{n\lambda}(g_j) - Z_{n\lambda}(h_j)\} - \{Z_{n\lambda}(g_0) - Z_{n\lambda}(h_0)\}\|_\lambda\right\|_{\psi_2}\nonumber\\
&\quad\quad + \frac{2}{\tau}\sqrt{\log\left(1 + |{\mathcal{G}}_0\times{\mathcal{G}}_0|\right)}
\max_{(g_0, h_0)\in{\mathcal{G}}_0\times{\mathcal{G}}_0, \|g_j - h_j\|_{L_\infty}\leq\varepsilon}\left\|\|Z_{n\lambda}(g_0) - Z_{n\lambda}(h_0)\|_\lambda\right\|_{\psi_2}
\nonumber.
\end{align}
Clearly, the second term can be bounded by inequality \eqref{eqn:g_0h_0_diff_bound}: 
\begin{align}
&\frac{2}{\tau}\sqrt{\log\left(1 + |{\mathcal{G}}_0\times{\mathcal{G}}_0|\right)}
\max_{(g_0, h_0)\in{\mathcal{G}}_0\times{\mathcal{G}}_0, \|g_j - h_j\|_{L_\infty}\leq\varepsilon}\left\|\|Z_{n\lambda}(g_0) - Z_{n\lambda}(h_0)\|_\lambda\right\|_{\psi_2}\nonumber\\
&\quad\leq \left(\frac{10\sqrt{24}C^2_K}{\tau}\right)\frac{\lambda^{-p/(4m)}\varepsilon}{\sqrt{n}}\sqrt{\log\left\{1 + \exp\left(2c_0\lambda^{-p/(2m)}\varepsilon^{-p/m}\right)\right\}}\nonumber,
\end{align}
since
\begin{align}
|{\mathcal{G}}_0\times{\mathcal{G}}_0| &= |{\mathcal{G}}_0|^2 \leq \exp(2\log\mathcal{N}_{[{\boldsymbol{\cdot}}]}(\varepsilon,\{\|f\|_\lambda\leq 1\},\|\cdot\|_{L_\infty}))\nonumber\\
&\leq\exp(2c_0\lambda^{-p/(2m)}\varepsilon^{-p/m}),
\end{align}
 it suffices to bound the first term. Write
\begin{align}
&\left\|\max_{g_j,h_j\in{\mathcal{G}}_j, \|g_j - h_j\|_{L_\infty}\leq \varepsilon}\|\{Z_{n\lambda}(g_j) - Z_{n\lambda}(h_j)\} - \{Z_{n\lambda}(g_0) - Z_{n\lambda}(h_0)\}\|_\lambda\right\|_{\psi_2}\nonumber\\
&\quad\leq 2\sum_{i = 0}^{j - 1}\left\|\max_{(g_i, g_{i + 1})\in{\mathcal{G}}_i\times{\mathcal{G}}_{i + 1}, \|g_i - g_{i + 1}\|_{L_\infty}\leq \varepsilon/2^i}
\|Z_{n\lambda}(g_{i + 1}) - Z_{n\lambda}(g_i)\|_\lambda\right\|_{\psi_2}\nonumber\\
&\quad\leq 2\sum_{i = 0}^{j - 1}\frac{2}{\tau}\sqrt{\log(1 + |{\mathcal{G}}_i|\times|{\mathcal{G}}_{i + 1}|)}\nonumber\\
&\quad\quad\times\max_{(g_i, g_{i + 1})\in{\mathcal{G}}_i\times{\mathcal{G}}_{i + 1}, \|g_i - g_{i + 1}\|_{L_\infty}\leq \varepsilon/2^i}\left\|\|Z_{n\lambda}(g_{i + 1}) - Z_{n\lambda}(g_i)\|_\lambda\right\|_{\psi_2}\nonumber\\
&\quad\leq \frac{4\sqrt{24}C^2_K\lambda^{-p/(4m)}}{\tau\sqrt{n}}\sum_{i = 0}^{j - 1}\sqrt{\log\left[1 + \exp\left\{2c_0\lambda^{-p/(2m)}\left(\varepsilon/2^i\right)^{-p/m}\right\}\right]}\varepsilon/2^i\nonumber,
\end{align}
where inequalities \eqref{eqn:Z_nlambda_g_Orlicz_norm_bound} and \eqref{eqn:finite_maximal_inequality} are applied. Bounding the sum by integral, we have
\begin{align}
&\sum_{i = 0}^{j - 1}\sqrt{\log\left[1 + \exp\left\{2c_0\lambda^{-p/(2m)}\left(\varepsilon/2^i\right)^{-p/m}\right\}\right]}\varepsilon/2^i
\nonumber\\
&\quad\leq \sum_{i = 0}^{j - 1}\int_{\varepsilon/2^{i + 1}}^{\varepsilon/2^i}\sqrt{\log\{1 + \exp(2c_0\lambda^{-p/(2m)} x^{-p/m})\}}\mathrm{d}x\nonumber\\
&\quad\leq \int_0^\varepsilon\sqrt{\log \{1 + \exp(2c_0\lambda^{-p/(2m)}x^{-p/m})}\}\mathrm{d}x\nonumber.
\end{align}
Putting all pieces above together, we obtain the following bound:
\begin{align*}
&\left\|\sup_{g,h\in \mathcal W^m_2(\mathcal X, 1), \,\|g - h\|_{L_\infty}\leq\varepsilon}\|Z_{n\lambda}(g) - Z_{n\lambda}(h)\|_\lambda\right\|_{\psi_2}\nonumber\\
&\quad\lesssim \frac{\lambda^{-p/(4m)}}{\sqrt{n}}\bigg[ \bigg.\frac{\varepsilon}{2^j} + \int_0^\varepsilon\sqrt{\log\{1 + \exp(2c_0\lambda^{-p/(2m)}x^{-p/m})\}}\mathrm{d}x \\
&\quad\quad+ \varepsilon\sqrt{\log\{1 + \exp(2c_0\lambda^{-p/(2m)}\varepsilon^{-p/m})\}}\bigg. \bigg]\nonumber.
\end{align*}
By taking $j\to\infty$, we can let the first term in the squared bracket tend to $0$, and hence,
\begin{align*}
&\left\|\sup_{g,h\in  \mathcal W^m_2(\mathcal X, 1),\, \|g - h\|_{L_\infty}\leq\varepsilon}\|Z_{n\lambda}(g) - Z_{n\lambda}(h)\|_\lambda\right\|_{\psi_2}\nonumber\\
&\quad\lesssim \frac{\lambda^{-p/(4m)}}{\sqrt{n}}\bigg[ \bigg. \int_0^\varepsilon\sqrt{\log\{1 + \exp(2c_0\lambda^{-p/(2m)}x^{-p/m})\}}\mathrm{d}x \\
&\quad \quad + \varepsilon\sqrt{\log\{1 + \exp(2c_0\lambda^{-p/(2m)}\varepsilon^{-p/m})\}}\bigg. \bigg]\nonumber\\
&\quad\lesssim \frac{\lambda^{-p/(4m)}}{\sqrt{n}}\int_0^\varepsilon \sqrt{\log\{1 + \exp(2c_0\lambda^{-p/(2m)}x^{-p/m})\}}\mathrm{d}x.
\end{align*}
Now we take $h = 0$, which implies $Z_{n\lambda}(h) = 0$ by the construction of $Z_{n\lambda}$. Furthermore, by the property of reproducing kernel $K_\lambda$ and the Cauchy-Schwarz inequality,
\begin{align*}
\|g - h\|_{L_\infty}&\leq \sup_{\bx\in\calX}|g(\bx)| = \sup_{\bx\in\calX}|\langle g(\cdot), K_\lambda(\bx,\cdot)\rangle_\lambda| \\
&\leq \sup_{\bx\in\calX}\|g\|_{\lambda}\sqrt{\langle K_\lambda(\bx,\cdot), K_\lambda(\bx, \cdot)\rangle_{\lambda}}\leq {C_K}\lambda^{-p/(4m)}.
\end{align*}
Taking $\varepsilon = {C_K}\lambda^{-p/(4m)}$, we obtain
\begin{align}
&\left\|\sup_{g\in \mathcal W^m_2(\mathcal X, 1)}\|Z_{n\lambda}(g)\|_{\lambda}\right\|_{\psi_2}\nonumber\\
&\quad\leq \left\|\sup_{g\in \mathcal W^m_2(\mathcal X, 1), \|g\|_{L_\infty}\leq \varepsilon}\|Z_{n\lambda}(g)\|_{\lambda}\right\|_{\psi_2}\nonumber\\
&\quad\leq \left\|\sup_{g, h\in \mathcal W^m_2(\mathcal X, 1), \|g - h\|_{L_\infty}\leq \varepsilon}\|Z_{n\lambda}(g) - Z_{n\lambda}(h)\|_{\lambda}\right\|_{\psi_2}\nonumber\\
&\quad\lesssim n^{-1/2}\lambda^{-p/(4m)}\int_0^{\epsilon}\sqrt{\log\{ 1 + \exp(2c_0\lambda^{-p/(2m)}x^{-p/m})}\mathrm{d}x\nonumber\\
&\quad\lesssim n^{-1/2}\lambda^{-p(6m - p)/(8m^2)}\nonumber.
\end{align}
Hence, invoking Lemma 8.1 in \cite{kosorok2008introduction}, we finally obtain
\begin{align}
\prob\left(
\sup_{g\in \mathcal W^m_2(\mathcal X, 1)}\left\|Z_{n\lambda}(g)\right\|_{\lambda}> t
\right)\leq 2\exp\left\{-\frac{nt^2}{\kappa_K^2\lambda^{-p(6m - p)/(4m^2)}}\right\}\nonumber,
\end{align}
for some absolute constant $\kappa_K$ depending on $K$ only, completing the proof.
\end{proof}

\section{Proof for Section \ref{subsec:theta_convergence}}
\label{proof:theta_convergence}

Denote $\hat {\bm \theta}_z:=\hat {\bm \theta}_{\lambda, \lambda_z, n}$, $\hat \delta_{z}(\cdot):=\hat \delta_{\lambda, \lambda_z,n}(\cdot)$ and $ \ell_z(\bm\theta,\delta) := \ell_{\lambda, \lambda_z, n}(\bm\theta,\delta) $  in (\ref{equ:KRR_sgp}). 

We need the following Corollary~\ref{corollary:prediction_sup} and Lemma~\ref{lemma:sqrt_conv} to prove theorem~\ref{thm:L_2_convergence}. Corollary~\ref{corollary:prediction_sup} is a direct consequence of Theorem \ref{thm:nonparametric}. We repeatedly use the fact that  for any $f(\cdot) \in L_2(\mathcal X)$
, there exists a constant $C_\rho$ such that $\|f\|_{L_2(\mathcal X)}\leq C_{\rho} \|f\|_{\mathcal H}$ in the following proof. 

\begin{corollary}
Denote $\hat \delta_{z, \bm \theta}=\argmin_{\delta \in \mathcal H}\ell_z(\bm \theta, \delta) $ for each $\bm \theta \in \bm \Theta$. 
Under the Assumptions A1 to A6, for sufficiently large $n$ and any $\alpha>2$ and $C_{\beta} \in (0,1)$, with probability at least $1-\exp\{-(\alpha-2)/3\}-\exp\{-n^{C_{\beta}\beta}\}$,  one has
\begin{align*}
&\sup\limits_{\bm \theta \in \bm \Theta} \|  \hat \delta_{z, \bm \theta}(\cdot) - (y^R(\cdot)- f^M(\cdot, \bm \theta ) )\|_{L_2(\mathcal X)} \\
&\quad \leq 2 \bigg[ \bigg. \sqrt{2}\bigg( \bigg. \sup\limits_{\bm \theta \in \bm \Theta}\|y^R(\cdot)- f^M(\cdot,  {\bm \theta})\|_{L_2(\mathcal X)}   \\
& \quad\quad \quad +\sup\limits_{\bm \theta \in \bm \Theta}\|y^R(\cdot)- f^M(\cdot,  {\bm \theta})\|_{\mathcal H} \bigg) \bigg.+C_K\sigma_0\alpha^{1/2}  \bigg. \bigg] n^{-\frac{m}{2m+p}},
\end{align*}
and
\begin{align*}
\sup_{{\bm{\theta}}\in\bm\Theta}\|  \hat \delta_{z, \bm \theta}(\cdot)\| _\calH
&\leq (2\sqrt{2} + 1)\sup_{{\bm{\theta}}\in\bm\Theta} \| (y^R(\cdot)- f^M(\cdot, \bm \theta ) )\|_{\calH}\\
&\quad + 2\sqrt{2} \sup\limits_{\bm \theta \in \bm \Theta}\|y^R(\cdot)- f^M(\cdot,  {\bm \theta})\|_{L_2(\mathcal X)}   
+ 2\sqrt{2}C_K\sigma_0\alpha^{1/2} 
 \end{align*}
 by choosing $\lambda=n^{-2m/(2 m+ p )}$ and  $\lambda_z =\lambda^{-1/2}$, where $C_K$ is a constant depending on the kernel $K(\cdot, \cdot)$.  
 \label{corollary:prediction_sup}
\end{corollary}

\begin{lemma}
 Under assumptions A1 to A6,
\begin{itemize}
\item[(i)] it holds that
\begin{align*}
&\sup_{{\bm{\theta}}\in\bm\Theta}\left|\frac{1}{n}\sum^n_{i=1}(y^R(\mathbf x_i)- f^M(\mathbf x_i, \bm \theta)- \hat \delta_{z, \bm \theta}(\mathbf x_i) )^2\right.\\
&\qquad \left.
-\int_{\mathbf x \in \mathcal X} (y^R(\mathbf x)- f^M(\mathbf x, \bm \theta)- \hat \delta_{z, \bm \theta}(\mathbf x) )^2d \mathbf x\right|= o_{p}(n^{-1/2}),
\end{align*}
and
\begin{align*}
&\sup_{{\bm{\theta}}\in\bm\Theta}\left|\frac{1}{n}\sum^n_{i=1}(y^R(\mathbf x_i)- f^M(\mathbf x_i, \bm \theta)- \hat \delta_{z, \bm \theta}(\mathbf x_i) ){\epsilon}_i
\right|= o_{p}(n^{-1/2});
\end{align*}
\item[(ii)]  for any $j=1,...,q$, one has
\begin{align*}
&\frac{1}{n}\sum^n_{i=1}(y^R(\mathbf x_i)- f^M(\mathbf x_i, \hat {\bm \theta}_z)- \hat \delta_z(\mathbf x_i) ) \frac{\partial f^M( \mathbf x_i, \hat {\bm \theta}_z )}{\partial \theta_j}  \\
  &\quad \quad   = \int_{\mathbf x \in \mathcal X} (y^R(\mathbf x)- f^M(\mathbf x, \hat {\bm \theta}_z)-\hat \delta_z(\mathbf x) ) \frac{\partial f^M( \mathbf x, \hat {\bm \theta}_z )}{\partial \theta_j} d \mathbf x+o_{p}(n^{-1/2}).
\end{align*}
\end{itemize}
\label{lemma:sqrt_conv}
\end{lemma}
\begin{proof}
Denote
\begin{equation*}
\mathcal W^{m}_2(\mathcal H, B) := \left\{
f(\cdot) = \sum^{\infty}_{j=1} f_j \phi(\cdot)\in L_2(\mathcal X): \sum^{\infty}_{j=1} j^{2m/p} f^2_j \leq B^2
\right\},
\label{equ:sobolev_W_B}
\end{equation*}
and
\begin{align*}
s^2_{i}( \bm \theta, \delta)&:=(y^R(\mathbf x_i)- f^M(\mathbf x_i, \bm \theta)-  \delta(\mathbf x_i))^2, \\
u_{i}( \bm \theta, \delta)&:=(y^R(\mathbf x_i)- f^M(\mathbf x_i, \bm \theta)-  \delta(\mathbf x_i)){\epsilon}_i, \\
r^2_i({\bm{\theta}}, \delta)&: = (y^R(\bx_i) - f^M(\bx_i, {\bm{\theta}}) - \delta(\bx_i))\frac{\partial f^M(\bx, {\bm{\theta}})}{\partial\theta_j}
\end{align*}
for $( \bm \theta, \delta) \in  \bm \Theta \times \mathcal W^{m}_2(\mathcal X,B) $ and some $B>0$ that will be specified later. Define the empirical processes
\begin{align*}
\bar s^2({\bm \theta},\delta)&:=\frac{1}{\sqrt{n}}\sum^n_{i=1} \{s^2_{i}(\delta, \bm \theta)- \E_{\mathbf x_i}[s^2_{i}(\delta, \bm \theta)] \}, \\
\bar u({\bm \theta},\delta)&:=\frac{1}{\sqrt{n}}\sum^n_{i=1} \{u_{i}(\delta, \bm \theta)- \E_{\mathbf x_i,{\epsilon}_i}[u_{i}(\delta, \bm \theta)] \}, \\
\bar r({\bm{\theta}},\delta)&: = \frac{1}{\sqrt{n}}\sum^n_{i=1} \{r_{i}(\delta, \bm \theta)- \E_{\mathbf x_i}[r_{i}(\delta, \bm \theta)] \},
\end{align*}
where 
\begin{align*}
\E_{\mathbf x_i}[s^2_{i}( {\bm \theta}, \delta)] &= \int_{\mathbf x \in \mathcal X} (y^R(\mathbf x)- f^M(\mathbf x, \bm \theta)- \delta(\mathbf x) )^2d \mathbf x,\\
\E_{\mathbf x_i, {\epsilon}_i}[u_{i}( {\bm \theta}, \delta)] &= \int_{\mathbf x \in \mathcal X} (y^R(\mathbf x)- f^M(\mathbf x, \bm \theta)- \delta(\mathbf x) ){\epsilon}_i d \mathbf x = 0,\\
\E_{\mathbf x_i}[r_{i}( {\bm \theta}, \delta)] &= \int_{\mathbf x \in \mathcal X} (y^R(\mathbf x)- f^M(\mathbf x, \bm \theta)- \delta(\mathbf x) )\frac{\partial f^M(\bx,{\bm{\theta}})}{\partial\theta_j}d \mathbf x.
\end{align*} 
By Assumptions A3 and A4, the function classes $\{\partial f^M(\cdot,{\bm{\theta}})/\partial\theta_j:{\bm{\theta}}\in\bm\Theta\}$ and $\mathcal F=\{y^R(\cdot)- f^M(\cdot, \bm \theta), \bm \theta \in \bm \Theta \}$ are Donsker. Note that, by definition, $\mathcal W^m_2(\mathcal X, B)$ is also Donsker. Since both $\mathcal W^m_2(\mathcal X, B)$ and $\mathcal F$ are uniformly bounded, 
the function classes 
\begin{align*}
&\{(y^R(\cdot) - f^M(\cdot,{\bm{\theta}}) - \delta(\cdot))^2:{\bm{\theta}}\in\bm\Theta, \delta\in \calW_2^m(\calX, B)\}, \quad \mbox{and}\\
&\left\{
(y^R(\cdot) - f^M(\cdot,{\bm{\theta}}) - \delta(\cdot))\frac{\partial f^M(\cdot,{\bm{\theta}})}{\partial\theta_j}:{\bm{\theta}}\in\bm\Theta, \delta\in\calW_2^m(\calX, B)
\right\}
\end{align*}
 are also Donsker classes. Furthermore, letting $f_{{\bm{\theta}},\delta}({\epsilon}, \bx) = (y^R(\bx) - f^M(\bx, {\bm{\theta}}) - \delta(\bx)){\epsilon}$, observe that for any $({\bm{\theta}}_1, \delta_1)$ and $({\bm{\theta}}_2, \delta_2)$, the distance 
\begin{align*}
&\left\{\E_0\left[(f_{{\bm{\theta}}_1,\delta_1} - f_{{\bm{\theta}}_2,\delta_2})^2\right]\right\}^{1/2}\\
&\quad= \sigma_0 \|f^M(\cdot, {\bm{\theta}}_1) - \delta_1(\cdot) - f^M(\cdot, {\bm{\theta}}_2) + \delta_2(\cdot)\|_{L_2(\calX)}\\
&\quad\leq \sigma_0\left[\|f^M(\cdot, {\bm{\theta}}_1) - f^M(\cdot, {\bm{\theta}}_2)\|_{L_2(\calX)} + 
\|\delta_1(\cdot ) - \delta_2(\cdot)\|_{L_2(\calX)}
\right]
\end{align*}
can be bounded by the $L_2(\calX)$-distance of functions in $\{f^M(\cdot, {\bm{\theta}}): {\bm{\theta}}\in \bm\Theta\}$ and $\delta(\cdot)\in\calW_2^m(\calX, B)$. 
In addition, by Assumption A4 $\{f^M(\cdot, {\bm{\theta}}):{\bm{\theta}}\in\bm\Theta\}$ and $\calW_2^m(\calX, B)$ are Donsker classes, it follows that the function class
 \[\left\{
  f_{{\bm{\theta}}, \delta}\in C(\mathbb{R}\times \calX):{\bm{\theta}}\in\bm\Theta, \delta\in\calW_2^m(\calX, B)
  \right\}\]
 is also Donsker, since its metric entropy can be upper bounded by those of $\{f^M(\cdot, {\bm{\theta}}):{\bm{\theta}}\in\bm\Theta\}$ and $\calW_2^m(\calX, B)$. 
By Theorem 2.4 in \cite{mammen1997penalized}, for any $t_1>0$ and any $B>0$, there exists $t_2, t_2', t_2''>0$ such that
\begin{align}
\label{equ:lim_sup_geer}
\limsup\limits_{n\to \infty}P\left(\sup\limits_{ \|\delta\|_{\mathcal H}\leq B,\, \bm \theta \in \bm \Theta, \, \| y^R(\cdot)-f^M(\cdot,\bm\theta)-\delta(\cdot) \|_{L_2(\mathcal X)}\leq t_2 } |\bar s^2( {\bm \theta},\delta) |>t_1  \right)<t_1,\\
\label{equ:lim_sup_geer_ii}
\limsup\limits_{n\to \infty}P\left(\sup\limits_{ \|\delta\|_{\mathcal H}\leq B,\, \bm \theta \in \bm \Theta, \, \| y^R(\cdot)-f^M(\cdot,\bm\theta)-\delta(\cdot) \|_{L_2(\mathcal X)}\leq t_2' } |\bar r( {\bm \theta},\delta) |>t_1  \right)<t_1,\\
\label{equ:lim_sup_geer_iii}
\limsup\limits_{n\to \infty}P\left(\sup\limits_{ \|\delta\|_{\mathcal H}\leq B,\, \bm \theta \in \bm \Theta, \, \| y^R(\cdot)-f^M(\cdot,\bm\theta)-\delta(\cdot) \|_{L_2(\mathcal X)}\leq t_2'' } |\bar u( {\bm \theta},\delta) |>t_1  \right)<t_1.
\end{align}
Note that by Corollary \ref{corollary:prediction_sup}, $\sup_{{\bm{\theta}}\in\bm\Theta}\|\hat\delta_{z,{\bm{\theta}}}\|_\calH$ is asymptotically tight, and therefore for any $\varepsilon>0$, there exists $B_0>0$ and some integer $N\in\mathbb{N}_+$, both depending on ${\epsilon}$, such that $P(\sup_{{\bm{\theta}}\in\bm\Theta}\|\hat \delta_{z, \bm \theta}\|_{\mathcal H}>B_0)\leq \varepsilon/3$ for all $n > N$. Now take $B = B_0$, $t_1 = {\epsilon}/3$. Then
we can choose $t_2$ to be a value that satisfies (\ref{equ:lim_sup_geer}), $t_2'$ satisfying \eqref{equ:lim_sup_geer_ii}, and $t_2''$ satisfying \eqref{equ:lim_sup_geer_iii}. 
By Corollary \ref{corollary:prediction_sup} and Assumption A5,
$\sup_{\bm{\theta}}\|\hat\delta_{z,{\bm{\theta}}}(\cdot) - (y^R(\cdot) - f^M(\cdot,{\bm{\theta}}))\|_{L_2(\calX)} = O_P(n^{-2m/(2m + p)})$, and hence
there exists $t_3 > 0$, depending on ${\epsilon}$ and $n$, such that for all $n > N$, it holds that
\[
P\left(\sup_{\bm \theta \in \bm \Theta} \| y^R(\cdot)-f^M(\cdot,\bm\theta)-\hat \delta_{z, \bm \theta}(\cdot) \|_{L_2(\mathcal X)}\geq t_3\right)<\varepsilon/3. 
\]
Without loss of generality, we may require $t_3 \leq \min\{t_2, t_2', t_2''\}$ by taking sufficiently large $n$. 
Then for sufficiently large $n$, we obtain 
\begin{align*}
&P\left(\sup_{{\bm{\theta}}\in\bm\Theta}|\bar s^2(\bm \theta, \hat \delta_{z, \bm \theta} ) |>\varepsilon \right)\\
&\quad\leq P\left(\sup\limits_{\sup_{\bm{\theta}}\|\hat \delta_{z,\bm\theta}\|_{\mathcal H}\leq B_{0},\, \bm \theta \in \bm \Theta,\, \| y^R(\cdot)-f^M(\cdot,\bm\theta)-\hat \delta_{z, \bm \theta}(\cdot) \|_{L_2(\mathcal X)}\leq t_2} |\bar s^2( \bm \theta, \hat \delta_{z, \bm \theta}) |>t_1  \right) \\
\nonumber &\quad \quad + P\left( \sup\limits_{\bm \theta \in \bm \Theta} \| y^R(\cdot)-f^M(\cdot,\bm\theta)-\hat \delta_{z, \bm \theta}(\cdot) \|_{L_2(\mathcal X)}> t_2\right) \\
&\quad\quad + P\left(\sup_{{\bm{\theta}}\in\bm\Theta}\|\hat \delta_{z, \bm \theta}\|_{\mathcal H}>B_0\right)\\
\nonumber &\quad<\varepsilon/3+\varepsilon/3+\varepsilon/3 
=\varepsilon,
\end{align*}
and similarly, 
\[
P\left(\sup_{{\bm{\theta}}\in\bm\Theta}|\bar r(\bm \theta, \hat \delta_{z, \bm \theta} ) |>\varepsilon \right) < {\epsilon}\quad\text{and}\quad
P\left(\sup_{{\bm{\theta}}\in\bm\Theta}|\bar u(\bm \theta, \hat \delta_{z, \bm \theta} ) |>\varepsilon \right) < {\epsilon} .
\]
Therefore, 
\begin{align*}
&\frac{1}{\sqrt{n}}\sup_{{\bm{\theta}}\in\bm\Theta}|\bar s({\bm{\theta}}, \hat\delta_{z,{\bm{\theta}}})|\\
&\quad =\sup_{{\bm{\theta}}\in\bm\Theta}\left|\frac{1}{n}\sum^n_{i=1}(y^R(\mathbf x_i)- f^M(\mathbf x_i, \bm \theta)- \hat \delta_{z, \bm \theta}(\mathbf x_i) )^2\right.\\
&\qquad\qquad \left.
-\int_{\mathbf x \in \mathcal X} (y^R(\mathbf x)- f^M(\mathbf x, \bm \theta)- \hat \delta_{z, \bm \theta}(\mathbf x) )^2d \mathbf x\right|= o_{p}(n^{-1/2}),
\end{align*}
 and 
\begin{align*}
&\frac{1}{\sqrt{n}}\sup_{{\bm{\theta}}\in\bm\Theta}|\bar u({\bm{\theta}}, \hat\delta_{z,{\bm{\theta}}})|
=\sup_{{\bm{\theta}}\in\bm\Theta}\left|\frac{1}{n}\sum^n_{i=1}(y^R(\mathbf x_i)- f^M(\mathbf x_i, \bm \theta)- \hat \delta_{z, \bm \theta}(\mathbf x_i) ){\epsilon}_i\right|= o_{p}(n^{-1/2}),
\end{align*}
completing the proof of (i). The proof of (ii) can be completed by observing that
\begin{align*}
&\left|
\frac{1}{n}\sum_{i = 1}^n(y^R(\bx_i) - f^M(\bx_i, \hat{\bm{\theta}}_z) - \hat\delta_z(\bx_i))\frac{\partial f^M(\bx_i, \hat{\bm{\theta}}_z)}{\partial\theta_j}\right.\\
&\quad\left. - \int_{\bx\in\calX}(y^R(\bx) - f^M(\bx,\hat{\bm{\theta}}_z) - \hat\delta_z(\bx))\frac{\partial f^M(\bx, \hat{\bm{\theta}}_z)}{\partial \theta_j}d\bx
\right|\nonumber\\
&\quad=
\frac{1}{\sqrt{n}}|\bar r(\hat{\bm{\theta}}_z, \hat\delta_{z,\hat{\bm{\theta}}_z})| 
\leq \frac{1}{\sqrt{n}}\sup_{{\bm{\theta}}\in\bm\Theta}|\bar r({\bm{\theta}}, \hat\delta_{z,{\bm{\theta}}})| = o_p(n^{-1/2}).
\end{align*}
\end{proof}

\begin{proof}[Proof for Theorem~\ref{thm:L_2_convergence}]
Without loss of generality, it suffices to prove the case when $\lambda_z=\lambda^{-1/2}$. For the general case when $\lambda_z=O(\lambda^{-1/2})$, the proof follows similarly. 
We first show $\hat{\bm\theta}_{z} \to^p \bm \theta_{L_2}$. By the definition of $\hat{\bm\theta}_{z} $, $\bm \theta_{L_2}$, and the theory of M-estimators (see, Theorem 5.7 in \cite{van2000asymptotic}), it suffices to show that $\lambda^{-1/2}(\ell_z( \bm \theta,\hat \delta_{z, \bm \theta})-\sigma^2_0)  \to^p  \| y^R(\cdot)- f^M(\cdot, \bm \theta)\|^2_{L_2(\mathcal X)}$ uniformly for each $\bm \theta \in \bm \Theta$. 
Note that
\begin{align*}
& \ell_z(\hat \delta_{z, \bm \theta}(\cdot), \bm \theta)\\
&\quad=\frac{1}{n}\sum^n_{i=1}(y^R(\mathbf x_i)- f^M(\mathbf x_i, \bm \theta) -\hat \delta_{z, \bm \theta}(\mathbf x_i) )^2 +\frac{1}{n}\sum^n_{i=1} {\epsilon}ilon^2_i  \\
&\quad\quad +\frac{2}{n}\sum^n_{i=1}(y^R(\mathbf x_i)- f^M(\mathbf x_i, \bm \theta)-\hat \delta_{z, \bm \theta}(\mathbf x_i)  ) {\epsilon}_i +\lambda \| \hat \delta_{z, \bm \theta}\|_{\mathcal H}^2 + \sqrt{\lambda}\| \hat \delta_{z, \bm \theta} \|_{L_2(\mathcal X)}^2 \\
&\quad:=A_n+B_n+C_n+D_n+E_n.
\label{equ:l_z_5_terms}
\end{align*}
For $A_n$, by Lemma \ref{lemma:sqrt_conv} (i) and Corollary \eqref{corollary:prediction_sup}, one has 
\begin{equation}
\label{equ:l_z_term_1}
\sup_{{\bm{\theta}}\in\bm\Theta}\left|\frac{1}{n}\sum^n_{i=1}(y^R(\mathbf x_i)- f^M(\mathbf x_i, \bm \theta) -\hat \delta_{z, \bm \theta}(\mathbf x_i) )^2\right| = o_p(n^{-1/2})
\end{equation}
Since $\E[B_n]=\sigma^2_0$ and ${\V}[B_n]=O(n^{-1})$, Chebyshev's inequality implies $(1/n)\sum^n_{i=1} {\epsilon}ilon^2_i= \sigma^2_0 + O_p(n^{-1/2})$ for $B_n$. 
For $C_n$, 
Lemma \ref{lemma:sqrt_conv} (i) guarantees that 
\[
\sup_{{\bm{\theta}}\in\bm\Theta}\frac{2}{n}\sum^n_{i=1}(y^R(\mathbf x_i)- f^M(\mathbf x_i, \bm \theta)-\hat \delta_{z, \bm \theta}(\mathbf x_i)  ) {\epsilon}_i=o_p(n^{-1/2})
\]
Since $\lambda=O(n^{-2m/(2m+p)})$, by the asymptotic tightness of $\sup_{{\bm{\theta}}}\|\hat\delta_{z,{\bm{\theta}}}\|_\calH$ (Corollary \ref{corollary:prediction_sup}), one has $\sup_{{\bm{\theta}}\in\bm\Theta}\lambda \| \hat \delta_{z, \bm \theta}\|_{\mathcal H}^2 = o_p(n^{-1/2})$. 
By putting the above all pieces together, we obtain
\begin{equation}
\sup_{{\bm{\theta}}\in\bm\Theta}\left|\lambda^{-1/2}(\ell_z(\hat \delta_{z, \bm \theta}(\cdot), \bm \theta)- \sigma^2_0) - \| \hat \delta_{z, \bm \theta}  \|^2_{L_2(\mathcal X)}\right| =  O_p( (\lambda n)^{-1/2} ).
\label{equ:ell_z}
\end{equation}
For any $\bm \theta$, by the Cauchy-Schwarz inequality, one has
\begin{align*}
\nonumber & \left|\| \hat \delta_{z, \bm \theta}  \|^2_{L_2(\mathcal X)}- \| y^R(\cdot)- f^M(\cdot, \bm \theta)\|^2_{L_2(\mathcal X)}\right| \\
 \nonumber &\quad\leq \| (\hat \delta_{z, \bm \theta}(\cdot)- (y^R(\cdot)- f^M(\cdot, \bm \theta)) \|_{L_2(\mathcal X)}  \| \hat \delta_{z, \bm \theta}(\cdot)+ y^R(\cdot)- f^M(\cdot, \bm \theta)\|_{L_2(\mathcal X)} 
\label{equ:L_2_minus}
\end{align*}
Recall that 
\[
\sup_{{\bm{\theta}}\in\bm\Theta}\| (\hat \delta_{z, \bm \theta}(\cdot)- (y^R(\cdot)- f^M(\cdot, \bm \theta)) \|_{L_2(\mathcal X)} =O_p(n^{-m/(2m+d)}) \]
by Corollary \ref{corollary:prediction_sup} and Assumption A4. Using Assumptions A4 and the asymptotic tightness of $\sup_{\bm{\theta}}\|\hat\delta_{z,{\bm{\theta}}}\|_\calH$ (Corollary \ref{corollary:prediction_sup}), one has 
\begin{align*}
&\| \hat \delta_{z, \bm \theta}(\cdot)+ y^R(\cdot)- f^M(\cdot, \bm \theta)\|_{L_2(\mathcal X)} \nonumber\\
&\quad\leq \| \hat \delta_{z, \bm \theta}(\cdot)\|_{L_2(\mathcal X)}+ \sup_{\bm \theta \in \bm \Theta}  \|y^R(\cdot)- f^M(\cdot, \bm \theta)\|_{L_2(\mathcal X)}\nonumber\\
&\quad\leq C_{\rho} \| \hat \delta_{z, \bm \theta}(\cdot)\|_{\mathcal H}+ \sup_{\bm \theta \in \bm \Theta} \|y^R(\cdot)- f^M(\cdot, \bm \theta)\|_{\mathcal H}=O_p(1).
\end{align*}
Thus 
\[
\sup_{{\bm{\theta}}\in\bm\Theta}
\left|\| \hat \delta_{z, \bm \theta}  \|^2_{L_2(\mathcal X)}- \| y^R(\cdot)- f^M(\cdot, \bm \theta)\|^2_{L_2(\mathcal X)}\right| = O_p(n^{-m/(2m + d)}),
\]
and hence,
\[  
\sup_{{\bm{\theta}}\in\bm\Theta}\left|\lambda^{-1/2}(\ell_z( \bm \theta, \hat \delta_{z, \bm \theta})- \sigma^2_0) - \| y^R(\cdot)- f^M(\cdot, \bm \theta)\|^2_{L_2(\mathcal X)}\right| = o_p(1), 
\] 
from which we conclude $\hat {\bm \theta}_z \to^p \bm \theta_{L_2} $. 


Next we derive the convergence rate of $\hat{\bm{\theta}}_z$. 
Apply the Fr{\'e}chet derivative on ${\ell}_{z}$ with regard to $\delta(\cdot)$ and the partial derivative on ${\ell}_{z}$ with regard to $\theta_j$, $j=1,...,q$. For any $g(\cdot) \in \mathcal H$,  $\hat \delta_z$  and $\hat {\bm \theta}_{z}$ satisfy
\begin{align}
\label{equ:frechet_dev_ell}
0&=-\frac{2}{n} \sum^n_{i=1}(y^F_i -f(\mathbf x_i,  \hat{\bm\theta}_z) -\hat \delta_z(\mathbf x_i) )g(\mathbf x_i)  + 2{\lambda}  \langle \hat \delta_z(\cdot), g(\cdot) \rangle_{\mathcal H}\nonumber\\
&\quad+2\sqrt{\lambda}  \langle \hat \delta_z(\cdot), g(\cdot) \rangle_{L_2(\mathcal X)}, \\
\label{equ:dev_theta}
0&=-\frac{2}{n} \sum^n_{i=1}(y^F_i -f(\mathbf x_i, \hat{\bm\theta}_z) -\hat \delta_z(\mathbf x_i) )  \frac{\partial f^M(\mathbf x_i,  \hat {\bm \theta}_z)}{\partial \theta_j}. 
\end{align}
Choosing $g(\cdot) =\frac{\partial f^M(\cdot, \bm {\hat \theta}_z)}{\partial\theta_j}$ and plugging (\ref{equ:dev_theta}) into (\ref{equ:frechet_dev_ell}), one has 
\begin{equation}
\sqrt{\lambda}  \left\langle \hat \delta_z(\cdot),\frac{\partial f^M(\cdot, \hat{\bm\theta}_z )}{\partial \theta_j}\right\rangle_{\mathcal H}+ \left\langle \hat \delta_z(\cdot), \frac{\partial f^M(\cdot, \hat{\bm\theta}_z)}{\partial\theta_j} \right\rangle_{L_2(\mathcal X)}=0.
\label{equ:partial_f_M_theta}
\end{equation}   
Substituting (\ref{equ:dev_theta}) into (\ref{equ:partial_f_M_theta}) and by Lemma~\ref{lemma:sqrt_conv} (ii), we have 
\begin{align*}
0&=-\frac{1}{n} \sum^n_{i=1}(y^F_i -f(\mathbf x_i,  \hat{\bm\theta}_z) -\hat \delta_z(\mathbf x_i) )  \frac{\partial f^M(\mathbf x_i, \hat{\bm\theta}_{z})}{\partial \theta_j}  \\
&= -\int(y^R(\mathbf x)-f^M(\mathbf x, \bm{\hat \theta}_z)) \frac{\partial f^M(\mathbf x, \hat{\bm\theta}_{z})}{\theta_j} d \mathbf x + \left\langle \hat \delta_z(\cdot), \frac{\partial f^M(\cdot, \hat{\bm\theta}_z)}{\partial \theta_j} \right\rangle_{L_2(\mathcal X)}\\
& \quad -\frac{1}{n}\sum^n_{i=1}{\epsilon}_i \frac{\partial f^M(\mathbf x_i, \hat{\bm\theta}_{z})}{\partial \theta_j}  +o_p(n^{-1/2}) \\
&=\int  \frac{ \partial (y^R(\mathbf x)-f^M(\mathbf x, \hat{\bm\theta}_z ) )^2}{\partial \theta_j}   d \mathbf x -\sqrt{\lambda}  \left\langle \hat \delta_z(\cdot),\frac{\partial f^M(\cdot, \hat{\bm\theta}_z)}{\partial \theta_j}\right\rangle_{\mathcal H}\\
&\quad -\frac{1}{n}\sum^n_{i=1}{\epsilon}_i \frac{\partial f^M(\mathbf x_i, \hat{\bm\theta}_{z})}{\theta_j}  +o_p(n^{-1/2}).
\end{align*}
Applying Taylor expansion to the first term on the right-hand side at $\bm \theta_{L_2}$, for any $j=1,...,q$, we obtain
\begin{align}
\nonumber&\left\{ \int  \frac{ \partial^2 (y^R(\mathbf x)-f^M(\mathbf x, \tilde{\bm\theta}_z ) )^2}{\partial\theta_j \partial\bm\theta} d \mathbf x \right\}^T (\hat{\bm\theta}_z-\bm \theta_{L_2}) \\
\nonumber&\quad=\left\{ \int  \frac{ \partial^2 (y^R(\mathbf x)-f^M(\mathbf x, {\bm\theta}_{L_2} ) )^2}{\partial\theta_j \partial\bm\theta} d \mathbf x + o_p(1)\right\}^T (\hat{\bm\theta}_z-\bm \theta_{L_2}) \\
&\quad=\sqrt{\lambda}  \left\langle \hat \delta_z(\cdot),\frac{\partial f^M(\cdot, \hat{\bm\theta}_z )}{\partial \theta_j}\right\rangle_{\mathcal H}+\frac{1}{n}\sum^n_{i=1}{\epsilon}_i \frac{\partial f^M(\mathbf x_i, \hat{\bm\theta}_{z})}{\partial \theta_j} + o_p(n^{-1/2}),
\label{equ:two_derivatives}
\end{align}
where $\tilde{\bm\theta}_z$ lies within the $q$ dimensional rectangle between $\bm \theta_{L_2}$ and $\hat{\bm\theta}_z $. 
Observe that Corollary \ref{corollary:prediction_sup} and assumption A3 imply
\begin{align*}
\left|\left\langle
\hat\delta_z, \frac{\partial f^M(\cdot, \hat{\bm{\theta}}_z)}{\partial\theta_j}
\right\rangle_\calH\right|\leq \|\hat\delta_z\|_\calH\left\|\frac{\partial f^M(\cdot, {\bm{\theta}}_{L_2})}{\partial\theta_j} + o_p(1)\right\|_\calH = O_p(1).
\end{align*}
Now we consider the second term. Define the empirical process
\begin{align*}
G_n({\bm{\theta}}) 
& = \frac{1}{\sqrt{n}}\sum_{i = 1}^n\left[
{\epsilon}_i\frac{\partial f^M(\bx_i, {\bm{\theta}})}{\partial\theta_j} - 
{\epsilon}_i\frac{\partial f^M(\bx_i, {\bm{\theta}}_{L_2})}{\partial\theta_j}
\right].
\end{align*}
and denote
\[
f_{\bm{\theta}}({\epsilon}, \bx) = {\epsilon}\frac{\partial f^M(\bx, {\bm{\theta}})}{\partial\theta_j} - 
{\epsilon}\frac{\partial f^M(\bx, {\bm{\theta}}_{L_2})}{\partial\theta_j}.
\]
Since
\begin{align*}
\E_{{\epsilon}ilon,\mathbf x}\left\{[f_{{\bm{\theta}}_1}({\epsilon},\bx) - f_{{\bm{\theta}}_2}({\epsilon},\bx)]^2\right\} 
&= 
\E_{{\epsilon}ilon,\mathbf x}\left[{\epsilon}^2\left(
\frac{\partial f^M(\bx,{\bm{\theta}}_1)}{\partial\theta_j} - 
\frac{\partial f^M(\bx,{\bm{\theta}}_2)}{\partial\theta_j}
\right)^2\right]\\
&=\sigma_0^2\left\|\frac{\partial f^M(\bx,{\bm{\theta}}_1)}{\partial\theta_j} - 
\frac{\partial f^M(\bx,{\bm{\theta}}_2)}{\partial\theta_j}\right\|_{L_2(\calX)}^2,
\end{align*}
therefore the function class
$
\left\{f_{\bm{\theta}}({\epsilon}, \bx)\in C(\mathbb{R}\times\calX):{\bm{\theta}}\in\bm\Theta\right\}
$
is Donsker by Assumption A3, and hence, $G_n({\bm{\theta}})$ converges weakly to a tight Gaussian stochastic process, denoted by $G(\cdot)$. 
W.l.o.g., we may take $G(\cdot)$ a version that has uniformly continuous sample paths (see Chapter 6 in \cite{van2000empirical}). 
Since $G_n({\bm{\theta}}_{L_2}) = 0$ for all $n$, it follows that $G({\bm{\theta}}_{L_2}) = 0$. By the consistency of $\hat{\bm{\theta}}_z$ and the continuous mapping theorem \citep{van2000asymptotic}, $G_n(\hat{\bm{\theta}}_z) = G({\bm{\theta}}_{L_2}) + o_p(1) = o_p(1)$. Therefore, 
\[
\frac{1}{n}\sum_{i = 1}^n{\epsilon}_i\frac{\partial f^M(\bx_i, \hat{\bm{\theta}}_z)}{\partial\theta_j} = \frac{1}{\sqrt{n}}G_n(\hat {\bm \theta}_{z}) + \frac{1}{n}\sum_{i = 1}^n{\epsilon}_i\frac{\partial f^M(\bx_i, {\bm{\theta}}_{L_2})}{\partial\theta_j} = O_p(n^{-1/2})
\]
To sum up,
\begin{align*}
&\left\{ \int  \frac{ \partial^2 (y^R(\mathbf x)-f^M(\mathbf x, {\bm\theta}_{L_2} ) )^2}{\partial{\bm{\theta}}\partial{\bm{\theta}}^T} d \mathbf x + o_p(1)\right\}
(\hat{\bm{\theta}}_z - {\bm{\theta}}_{L_2})\nonumber\\
&\quad = O_p(n^{-m/(2m + p)}) + O_p(n^{-1/2}) + o_p(n^{-1/2}) = O_p(n^{-m/(2m + p)}),
\end{align*}
completing the proof.
\end{proof}

\section{Proof and additional results for Section \ref{sec:discretized_sgasp}}

The identities in the Lemma~\ref{lemma:identities_r_z} are used repeatedly in the proof of the Theorem~\ref{thm:pred_SGP} and Lemma~\ref{lemma:equivalence_mean_GP}.
\begin{lemma}
\label{lemma:identities_r_z}
Denote $\sigma^2 \mathbf R_{z_d}$ the covariance matrix of $(\delta_{z_d}(\mathbf x_1),...,\delta_{z_d}(\mathbf x_n))^T$, where the $(i, \, j)$ entry being $\sigma^2 K_{z_d}(\mathbf x_i, \mathbf x_j)$ defined in (\ref{equ:sigma_2_K_z_d}). Denote $r_{z_d}(\mathbf x)=(K_{z_d}(\mathbf x, \mathbf x_1), ..., K_{z_d}(\mathbf x, \mathbf x_n))^T$ for any $\mathbf x \in \mathcal X$. One has the following identities
\begin{align}
\label{equ:R_D_z_inv}
  \mathbf R_{z_d}^{-1}&=\mathbf R^{-1}+\frac{\lambda_z}{n}\mathbf I_n, \\
 \label{equ:r_D_z}
 \mathbf r^T_{z_d}(\mathbf x)&= \frac{n}{\lambda_z}\mathbf r^T(\mathbf x) \mathbf {\tilde R}^{-1}=\mathbf r^T(\mathbf x) \mathbf R^{-1} \mathbf R_{z_d},
\end{align}
for any $\mathbf x \in \mathcal X$.
\begin{proof}
By the definitions of $\mathbf R_{z_d}$ and the Woodbury Identity, one has 
\begin{align*}
\mathbf R_{z_d}&=\mathbf R- \mathbf R \mathbf {\tilde R}^{-1} \mathbf R= \mathbf R \left(\mathbf I_n -\left(\mathbf I_n+ \frac{\lambda_z}{n}\mathbf R^{-1}\right)^{-1}\right) \\
&= \mathbf R \left(\frac{\lambda_z}{n} \mathbf R +\mathbf I_n \right)^{-1} = \left(\mathbf R^{-1} + \frac{\lambda_z}{n}\mathbf I_n \right)^{-1},
\end{align*}
from which (\ref{equ:R_D_z_inv}) follows. 

Equation (\ref{equ:r_D_z}) can be shown similarly by noting $\mathbf r^T_{z_d}(\mathbf x)= \mathbf r^T(\mathbf x) -\mathbf r^T(\mathbf x) {\mathbf {\tilde R} }^{-1} \mathbf R$ and the Woodbury Identity.
\end{proof}
\end{lemma}

\begin{proof}[Proof of Theorem~\ref{thm:pred_SGP}]
The predictive mean is as follows 
\begin{align*}
\hat \mu_z(\mathbf x)&=\E[y^F(\mathbf x) \mid \mathbf y^F, \bm \theta, \sigma^2_0, \lambda, \lambda_z] \\
                                     &=f^M(\mathbf x, \bm \theta)+ \mathbf r_{z_d}(\mathbf x)^T \left( \mathbf R_{z_d} +n \lambda \mathbf I_n\right)^{-1} (\mathbf y^F-\mathbf f^M_{\bm \theta} ) \\
                                     &= f^M(\mathbf x, \bm \theta)+\mathbf r(\mathbf x)^T \mathbf R^{-1} \mathbf R_{z_d} \left( \mathbf R_{z_d} +n \lambda \mathbf I_n\right)^{-1} (\mathbf y^F-\mathbf f^M_{\bm \theta}) \\
                                      &= f^M(\mathbf x, \bm \theta)+\mathbf r(\mathbf x)^T \mathbf R^{-1} \left( \mathbf I_n +n\lambda \left( \mathbf R^{-1} +\frac{\lambda_z}{n} \mathbf I_n \right)^{-1}\right)(\mathbf y^F-\mathbf f^M_{\bm \theta}) \\
                                      &=f^M(\mathbf x, \bm \theta)+ \frac{\mathbf r(\mathbf x)^T}{(1+\lambda \lambda_z)} \left(\mathbf R^{-1}+\frac{n\lambda}{1+\lambda \lambda_z}\mathbf I_n\right)^{-1} (\mathbf y^F-\mathbf f^M_{\bm \theta}),
\end{align*}
where the last two equalities follow from (\ref{equ:r_D_z}) and (\ref{equ:R_D_z_inv}), respectively.    

The predictive variance can be obtained using (\ref{equ:r_D_z}) and (\ref{equ:R_D_z_inv}) as follows 
\begin{align*}
K^*_z(\mathbf x, \mathbf x)&= K_{z_d}(\mathbf x, \mathbf x)- \mathbf r^T_{z_d}(\mathbf x) \left( \mathbf R_{z_d} +n \lambda \mathbf I_n\right)^{-1} \mathbf r_{z_d}(\mathbf x) \\
                          &=K(\mathbf x, \mathbf x)-  \mathbf r^T(\mathbf x)  \mathbf {\tilde R}^{-1} \mathbf r(\mathbf x) \\
                          &\quad \quad - (1+\lambda \lambda_z)^{-1} \mathbf r(\mathbf x)^T \left(\mathbf R^{-1}+\frac{n\lambda}{1+\lambda \lambda_z}\mathbf I_n\right)^{-1}  \frac{n}{\lambda_z} \mathbf {\tilde R}^{-1} \mathbf r(\mathbf x) 
                          \end{align*}
                          from which the result follows. 

\end{proof}

\begin{proof}[Proof of Lemma~\ref{lemma:equivalence_mean_GP}]

When $\sigma^2_0=0$, the predictive mean is as follows
\begin{align*}
\E[y^F(\mathbf x) \mid \mathbf y^F, \bm \theta, \sigma^2_0, \lambda, \lambda_z]=&f^M(\mathbf x, \bm \theta)+ \mathbf r_{z_d}(\mathbf x)^T \mathbf R_{z_d}^{-1} (\mathbf y^F-\mathbf f^M_{\bm \theta}) \\
=&f^M(\mathbf x, \bm \theta)+\mathbf r(\mathbf x)^T \mathbf R^{-1} \mathbf R_{z_d}   \mathbf R_{z_d}^{-1} (\mathbf y^F-\mathbf f^M_{\bm \theta}) \\
=&f^M(\mathbf x, \bm \theta)+ \mathbf r(\mathbf x)^T \mathbf R^{-1} (\mathbf y^F-\mathbf f^M_{\bm \theta}).
\end{align*}
The predictive variance can be obtained similarly. 
\end{proof}

\end{document}